\documentclass[12pt]{amsart}

\usepackage{etex}
\usepackage[usenames,dvipsnames]{pstricks} 
\usepackage{epsfig}
\usepackage{graphicx,color}
\usepackage{geometry}
\usepackage[all]{xy}
\usepackage{amssymb,amscd}
\usepackage{cite}
\usepackage{fullpage}
\usepackage{marvosym}
\xyoption{poly}
\usepackage{url}
\usepackage{comment}
\usepackage{float}

\usepackage{tikz}
\usepackage{tikz-cd}
\usetikzlibrary{decorations.pathmorphing}
\newtheorem{introtheorem}{Theorem}

\newtheorem{introcorollary}[introtheorem]{Corollary}
\newtheorem{theorem}{Theorem}[section]
\newtheorem{lemma}[theorem]{Lemma}
\newtheorem{proposition}[theorem]{Proposition}
\newtheorem{corollary}[theorem]{Corollary}
\theoremstyle{definition}
\newtheorem{definition}[theorem]{Definition}

\newtheorem{remark}[theorem]{Remark}

\newtheorem*{question*}{Question}

\newtheorem*{questions*}{Questions}

\newtheorem*{steps*}{Answer/steps}

\newtheorem*{progress*}{Progress}

\newtheorem*{classification*}{Classification}

\newtheorem*{construction*}{Classification}
\newtheorem*{example*}{Example}

\newtheorem*{remark*}{Remark}
\newtheorem*{remarks*}{Remarks}
\newtheorem*{definition*}{Definition}

\usepackage{calrsfs}
\usepackage{url}
\usepackage{longtable}
\usepackage[OT2, T1]{fontenc}
\usepackage{textcomp}
\usepackage{times}
\usepackage[scaled=0.92]{helvet}

\renewcommand{\tilde}{\widetilde}

\newcommand{\C}{\mathbb{C}}
\newcommand{\Q}{\mathbb{Q}}

\newcommand{\R}{\mathbb{R}}
\newcommand{\Z}{\mathbb{Z}}

\newcommand{\F}{\mathbb{F}}

\newcommand{\cO}{\mathcal{O}}

\newcommand{\X}{\mathcal{X}}

\DeclareMathOperator{\SL}{SL}
\DeclareMathOperator{\GL}{GL}

\DeclareMathOperator{\Aut}{Aut}

\DeclareMathOperator{\Sp}{S}
\DeclareMathOperator{\Gal}{Gal}

\DeclareSymbolFont{cyrletters}{OT2}{wncyr}{m}{n}
\DeclareMathSymbol{\Sha}{\mathalpha}{cyrletters}{"58}

\makeatletter

\def\greekbolds#1{%
 \@for\next:=#1\do{%
    \def\X##1;{%
     \expandafter\def\csname V##1\endcsname{\boldsymbol{\csname##1\endcsname}}
     }
   \expandafter\X\next;
  }
}

\greekbolds{alpha,beta,iota,gamma,lambda,nu,eta,Gamma,varsigma,Lambda}

\def\make@bb#1{\expandafter\def
  \csname bb#1\endcsname{{\mathbb{#1}}}\ignorespaces}

\def\make@bbm#1{\expandafter\def
  \csname bb#1\endcsname{{\mathbbm{#1}}}\ignorespaces}

\def\make@bf#1{\expandafter\def\csname bf#1\endcsname{{\bf
      #1}}\ignorespaces} 

\def\make@gr#1{\expandafter\def
  \csname gr#1\endcsname{{\mathfrak{#1}}}\ignorespaces}

\def\make@scr#1{\expandafter\def
  \csname scr#1\endcsname{{\mathscr{#1}}}\ignorespaces}

\def\make@cal#1{\expandafter\def\csname cal#1\endcsname{{\mathcal
      #1}}\ignorespaces} 

\def\do@Letters#1{#1A #1B #1C #1D #1E #1F #1G #1H #1I #1J #1K #1L #1M
                 #1N #1O #1P #1Q #1R #1S #1T #1U #1V #1W #1X #1Y #1Z}
\def\do@letters#1{#1a #1b #1c #1d #1e #1f #1g #1h #1i #1j #1k #1l #1m
                 #1n #1o #1p #1q #1r #1s #1t #1u #1v #1w #1x #1y #1z}
\do@Letters\make@bb   \do@letters\make@bbm
\do@Letters\make@cal  
\do@Letters\make@scr 
\do@Letters\make@bf \do@letters\make@bf   
\do@Letters\make@gr   \do@letters\make@gr
\makeatother

\def\ol{\overline}
\def\wt{\widetilde}
\def\opp{\mathrm{opp}}
\def\ul{\underline}
\def\onto{\twoheadrightarrow}

\def\wh{\widehat}

\newcommand{\<}{\langle}   
\renewcommand{\>}{\rangle} 

\newcommand{\isoto}{\stackrel{\sim}{\longrightarrow}}
\newcommand{\embed}{\hookrightarrow}
\def\Spec{{\rm Spec}\,}

\def\Fpbar{\overline{\bbF}_p}
\def\Fp{{\bbF}_p}
\def\Fq{{\bbF}_q}

\def\Qp{{\bbQ}_p}

\def\Zp{{\bbZ}_p}

\newcommand{\A}{\mathbb A}    
\newcommand{\G}{\mathbb G}
  
\def\makeop#1{\expandafter\def\csname#1\endcsname
  {\mathop{\rm #1}\nolimits}\ignorespaces}
\makeop{Hom}   \makeop{End}   \makeop{Aut}   \makeop{Isom}  \makeop{Pic} 
\makeop{Gal}   \makeop{ord}   \makeop{Char}  \makeop{Div}   \makeop{Lie} 
\makeop{PGL}   \makeop{Corr}  \makeop{PSL}   \makeop{sgn}   \makeop{Spf}
\makeop{Spec}  \makeop{Tr}    \makeop{Nr}    \makeop{Fr}    \makeop{disc}
\makeop{Proj}  \makeop{supp}  \makeop{ker}   \makeop{im}    \makeop{dom}
\makeop{coker} \makeop{Stab}  \makeop{SO}    \makeop{SL}    \makeop{SL}
\makeop{Cl}    \makeop{cond}  \makeop{Br}    \makeop{inv}   \makeop{rank}
\makeop{id}    \makeop{Fil}   \makeop{Frac}  \makeop{GL}    \makeop{SU}
\makeop{Nrd}   \makeop{Sp}    \makeop{Tr}    \makeop{Trd}   \makeop{diag}
\makeop{Res}   \makeop{ind}   \makeop{depth} \makeop{Tr}    \makeop{st}
\makeop{Ad}    \makeop{Int}   \makeop{tr}    \makeop{Sym}   \makeop{can}
\makeop{length}\makeop{SO}    \makeop{torsion} \makeop{GSp} \makeop{Ker}
\makeop{Adm}   \makeop{Mat}

\DeclareMathOperator{\Mass}{Mass}

\newcommand{\dieu}{Dieudonn\'{e} }

\DeclareMathSymbol{\twoheadrightarrow} {\mathrel}{AMSa}{"10}

\DeclareMathOperator{\pr}{pr}

\DeclareMathOperator{\vol}{vol}

\def\sfF{\mathsf{F}}
\def\sfV{\mathsf{V}}

\begin{document}

\title{When is a polarised abelian variety determined by its $\boldsymbol{p}$-divisible group?}

\author{Tomoyoshi Ibukiyama}
\address{Department of Mathematics, Graduate School of Science, 
Osaka University, Toyonaka, Japan}
\email{ibukiyam@math.sci.osaka-u.ac.jp}

\author{Valentijn Karemaker}
\address{Mathematical Institute, Utrecht University, Utrecht, The Netherlands}
\email{V.Z.Karemaker@uu.nl}

\author{Chia-Fu Yu}
\address{Institute of Mathematics, Academia  Sinica and National Center for Theoretic Sciences, Taipei, Taiwan}
\email{chiafu@math.sinica.edu.tw}

\keywords{Gauss problem, Hermitian lattices, abelian varieties, central leaves, mass formula}
\subjclass{14K10 (14K15, 11G10, 11E41, 16H20)}

\begin{abstract} 
We study the Siegel modular variety $\mathcal{A}_g\otimes \overline{\mathbb{F}}_p$ of genus $g$ and its supersingular locus~$\mathcal{S}_g$. As our main result we determine precisely when $\mathcal{S}_g$ is irreducible, and we list all~$x$ in~$\mathcal{A}_g\otimes \overline{\mathbb{F}}_p$ 
for which the corresponding central leaf $\mathcal{C}(x)$ 
consists of one point, that is, for which~$x$ corresponds to a polarised abelian variety which is uniquely determined by its associated polarised $p$-divisible group. The first problem translates to a class number one problem for quaternion Hermitian lattices. The second problem also translates to a class number one problem, whose solution involves mass formulae, automorphism groups, and a careful analysis of Ekedahl-Oort strata in genus $g=4$.
\end{abstract}

\maketitle
\setcounter{tocdepth}{2}

\section{Introduction}

\def\pol{{\rm pol}}
\def\opp{{\rm opp}}
\def\LatR{{\rm Lat}_R}
\def\RLat{{}_{R}{\rm Lat}}
\def\RoLat{{}_{R^{\rm opp}}{\rm Lat}}

Throughout this paper, let $p$ denote a prime number and let $k$ be an algebraically closed field of characteristic $p$. 
Let $(X_1,\lambda_1)$ and $(X_2,\lambda_2)$ be two principally polarised abelian varieties over~$k$. Then 
\begin{equation}\label{eq:Q}
    (X_1,\lambda_1)\simeq (X_2,\lambda_2) \implies (X_1,\lambda_1)[p^\infty]\simeq (X_2,\lambda_2)[p^\infty],
\end{equation}
where $(X_i,\lambda_i)[p^\infty]$ denotes the polarised $p$-divisible group associated to $(X_i,\lambda_i)$. The converse is generally not true. Indeed, the goal of this paper is to determine precisely when the converse to~\eqref{eq:Q} is true. 

We treat this problem by putting it in a geometric context, by considering the moduli space of abelian varieties. So let $\calA_g$ denote the moduli space over $\Fpbar$ of principally polarised abelian varieties of dimension $g\ge 1$. For a point
$x=[(X_0,\lambda_0)]\in \calA_{g}(k)$, denote by
\[
 \calC(x):=\{[(X,\lambda)]\in \calA_{g}(k) :
(X,\lambda)[p^\infty]\simeq  (X_0,\lambda_0)[p^\infty] \} 
\]
the central leaf of $\calA_{g}$ passing through $x$ introduced in \cite{oort:foliation}. Then the problem becomes a very simple question: when does a central leaf $\calC(x)$ consist of only the point $x$ itself?\\ 

Chai and Oort \cite{COirr} proved the Hecke orbit conjecture, stating that the prime-to-$p$ Hecke orbit $\calH^{(p)}(x)$ of any point $x\in \calA_g(k)$ is Zariski dense in the ambient central leaf $\calC(x)$. They also proved that every non-supersingular Newton stratum is irreducible and that every non-supersingular central leaf is irreducible.

Furthermore, it follows from a result of Chai \cite[Proposition~1]{chai}, cf.~Proposition~\ref{prop:chai}, that if $x=[(X_0,\lambda_0)]$ is not supersingular, then $\calC(x)$ has positive dimension. If $x$ is supersingular, then the central leaf is finite. Hence, the converse to~\eqref{eq:Q} can be true only when $X_0$ is a supersingular abelian variety, that is, when $X_0$ is isogenous to a product of supersingular elliptic curves.\\

In this paper we prove supersingular analogues of the results of Chai and Chai--Oort. That is, we determine precisely when a supersingular central leaf $\calC(x)$ (i.e., $x\in \calS_g(k))$ is irreducible (i.e., $\mathcal{C}(x) = \{x \}$), and when the supersingular locus $\calS_g \subseteq \mathcal{A}_g$ is irreducible.

When $g=1$, it is well known that the supersingular locus $\calS_1$ is the same as the unique supersingular central leaf~$\calC(x)$, whose cardinality is the class number of the quaternion $\Q$-algebra ramified at $\{p,\infty\}$. Then 
$\calS_1=\calC(x)$ is irreducible if and only if  
$p\in \{2, 3, 5, 7, 13\}$.

When $g>1$, we will see in Subsection~\ref{ssec:4first} that the size of $\calC(x)$ is again equal to a class number of a certain reductive group, so the question is a type of Gauss problem or class number one problem. To solve this problem, we also answer fundamental questions on arithmetic properties of the polarised abelian varieties in question. 

These answers have applications in particular to determining the geometric endomorphism rings and automorphism groups of polarised abelian varieties in the Ekedahl-Oort strata that are entirely contained in the supersingular locus.\\ 

For any abelian variety $X$ over $k$, the $a$-number
of $X$ is defined by $a(X):=\dim_k \Hom(\alpha_p, X)$,
where $\alpha_p$ is the kernel of the Frobenius morphism on the additive group $\bbG_a$. The $a$-number of the abelian variety corresponding to a point $x \in \calA_{g}(k)$ is denoted by $a(x)$. 
Our main result is the following theorem.

\begin{introtheorem}\label{thm:main}  (Theorem~\ref{thm:main2})
\begin{enumerate}
\item The supersingular locus $\calS_g$ is geometrically irreducible if and only if 
one of the following three cases holds:
\begin{itemize}
\item [(i)] $g=1$ and $p\in \{2,3,5,7,13\}$;
\item [(ii)] $g=2$ and $p\in \{ 2, 3, 5, 7, 11\}$; 
\item [(iii)] $(g, p)=(3,2)$ or $(g,p)=(4,2)$. 
\end{itemize}

\item Let $\calC(x)$ be the central leaf of $\calA_{g}$ passing through  a point $x=[X_0,\lambda_0]\in \calS_{g}(k)$. Then $\calC(x)$ consists of one element if and
only if one of the following three cases holds:

\begin{itemize}
\item [(i)] $g=1$ and $p\in \{2,3,5,7,13\}$;
\item [(ii)] $g=2$ and $p=2,3$; 
\item [(iii)] $g=3$, $p=2$ and $a(x)\ge 2$. 
\end{itemize}
\end{enumerate}
\end{introtheorem} 

\begin{introcorollary}
    A principally polarised abelian variety $(X,\lambda)$ over $k$ is uniquely determined by its polarised $p$-divisible group if and only if $X$ is supersingular, 
    $g=\dim X \leq 3$, and one of (i), (ii), (iii) of Theorem~\ref{thm:main}.(2) holds.
\end{introcorollary}

We first comment on Theorem~\ref{thm:main}.(2). As mentioned above, Case~(i) is
well-known; Case~(ii) is a result due to the
first author~\cite{ibukiyama}. In both cases, the result is independent of the point $x$. In Section~\ref{sec:proof} we prove the remaining cases; namely,
we show that $\vert \calC(x) \vert >1$ for
$g\geq 4$, and that when $g=3$, (iii) lists the only cases
such that $|\calC(x)|=1$. When $g=3$ and $a(x)=3$ (the \emph{principal
  genus} case), the class number
one result is known due to Hashimoto \cite{hashimoto:g=3}. Hashimoto
first computes an explicit class number formula in 
the principal genus case and proves
the class number one result as a direct consequence.

Our method instead uses mass formulae and the
automorphism groups of certain abelian varieties, which is much simpler than proving explicit class number formulae. 
Mass formulae for dimension $g=3$ were very recently provided by F.~Yobuko and the second and third-named authors~\cite{karemaker-yobuko-yu}. In addition, we perform a careful analysis of the Ekedahl-Oort strata in dimension $g=4$; in Proposition~\ref{prop:EO} we show precisely how the Ekedahl-Oort strata and Newton strata intersect.
It is worth mentioning that we do not use any computers in this paper (unlike most papers that treat class number one problems); the only numerical data we use is the well-known table above Lemma~\ref{lem:vn} in Subsection~\ref{ssec:Gaussarith}.

In the course of our proof of Theorem~\ref{thm:main}.(2), in Subsection~\ref{ssec:Eisog} we define the notion of minimal $E$-isogenies (Definition~\ref{def:minE}), where $E$ is any elliptic curve (not necessarily supersingular) over any field~$K$. This generalises the notion of minimal isogenies for supersingular abelian varieties in the sense of Oort \cite[Section 1.8]{lioort}. This new construction of minimal isogenies even has a new (and stronger) universal property since the test object is not required to be an isogeny, cf.~Remark~\ref{rem:min_isog}. 
We also extend the results of Jordan et al.~\cite{JKPRST} on abelian varieties isogenous to a power of an elliptic curve to those with a polarisation in Subsections~\ref{ssec:powers}--\ref{ssec:powerAV}, cf.~Proposition~\ref{prop:equiv}. These results can be paraphrased as follows:

\begin{introtheorem}\label{thm:B}\
Let $E$ be any elliptic curve over any field $K$, let $R = \mathrm{End}(E)$ and denote by $\mathrm{Lat}_R$ (resp.~$\mathrm{Lat}_R^H$) the category of right $R$-lattices (resp.~positive-definite Hermitian such lattices). Also let $\mathcal{A}_E$ (resp.~$\mathcal{A}_E^{\mathrm{pol}}$) denote the category of abelian varieties over $K$ isogenous to a power of $E$ (resp.~fractionally polarised such varieties) and let $\mathcal{A}_{E,\mathrm{ess}}$ (resp.~$\mathcal{A}^{\mathrm{pol}}_{E,\mathrm{ess}}$) be the essential image of the sheaf Hom functor ${\mathcal Hom}_R(-, E): \RLat^\opp \to \calA_E$ constructed in \cite{JKPRST} with inverse $\mathrm{Hom}(-,E)$ (resp.~its fractionally polarised elements).
\begin{enumerate}
\item (Proposition~\ref{prop:equiv}, Corollary~\ref{cor:Aut}.(1))  
    There exists an equivalence of categories $\mathcal{A}^{\mathrm{pol}}_{E,\mathrm{ess}} \longrightarrow \mathrm{Lat}_R^H$.
    Hence, for any $(X,\lambda) \in \mathcal{A}^{\mathrm{pol}}_{E,\mathrm{ess}}$ there exists a unique decomposition of $\mathrm{Aut}(X,\lambda)$ determined by the unique orthogonal decomposition of its associated lattice.
\item (Corollary~\ref{cor:JKPRST}, Corollary~\ref{cor:Aut}.(2)) Suppose that $K = \mathbb{F}_q$ and that either $E$ is ordinary with $R = \mathbb{Z}[\pi]$, or $E$ is supersingular with $K = \mathbb{F}_p$ and $R = \mathbb{Z}[\pi]$, or $E$ is supersingular with $K = \mathbb{F}_{p^2}$ and $R$ has rank $4$ over $\mathbb{Z}$. Then all results in~(1) hold upon replacing $\mathcal{A}^{\mathrm{pol}}_{E,\mathrm{ess}}$ with $\mathcal{A}^{\mathrm{pol}}_{E}$.
\item (Theorem~\ref{thm:pol+JKPRST}) All results in~(1) hold when $E$ is any abelian variety over $K~=~\mathbb{F}_p$ with minimal endomorphism ring $R = \mathbb{Z}[\pi, \bar{\pi}]$ and commutative endomorphism algebra.
\end{enumerate}
\end{introtheorem} 

Finally, we comment on Theorem~\ref{thm:main}.(1). It was proven in \cite[Theorem 4.9]{lioort} that the number of irreducible components of $\mathcal{S}_g$ is a class number of a genus of maximal quaternion Hermitian lattices, namely the class number $H_g(p,1)$ of the principal genus if $g$ is odd and the class number $H_g(1,p)$ of the non-principal genus if $g$ is even. Thus, Theorem~\ref{thm:main}.(1) also solves a Gauss problem or class number one problem. Indeed, the above indicates a clear connection between the arithmetic (\ref{thm:main}.(1)) and geometric (\ref{thm:main}.(2)) class number one problems we are considering.

More precisely, let $B$ be 
a definite quaternion $\Q$-algebra 
and let
$O$ be a maximal order in $B$. Let $V$ be a left $B$-module of rank
$n$, and $f:V\times V\to B$ be a positive-definite quaternion Hermitian form
with respect to the canonical involution $x\mapsto \bar x$. 
For each left $O$-lattice $L$ in $V$ denote by $h(L,f)$
the class number of the isomorphism classes in the genus containing $L$. 
As the main result of the arithmetic part of this paper (Section~\ref{sec:Arith}), in Theorem~\ref{thm:mainarith} we determine precisely when $h(L,f)=1$ for all maximal left $O$-lattices $L$. This is a special case, with a different proof, of the results of \cite[Chapter~9]{KirschmerHab}, cf.~Remark~\ref{rem:Kirschmer}.
For the rank one case, the list of definite quaternion $\Z$-orders of class number one has been determined by Brzezinski~\cite{brzezinski:h=1} in 1995; this was generalised to class number at most two by Kirschmer-Lorch~\cite{KirschmerLorch}.\\ 

The structure of the paper is as follows. The arithmetic theory (Theorem~\ref{thm:main}.(1)) is treated in Section~2, building up to the class number one result in Theorem~\ref{thm:mainarith}. Theorem~\ref{orthogonal} is the unique orthogonal decomposition result for lattices, and Corollary~\ref{autodecomposition} gives its consequence for automorphism groups of such lattices. The geometric theory starts in Section~\ref{sec:GMF}, which recalls mass formulae due to the second and third authors as well as other authors. Section~\ref{sec:aut} treats automorphism groups (cf.~Corollary~\ref{cor:Aut}), through the results collected in Theorem~\ref{thm:B}. Minimal $E$-isogenies are introduced in Subsection~\ref{ssec:Eisog}, and Subsection~\ref{ssec:uniquedec} provides the geometric analogue of Theorem~\ref{orthogonal}.
Finally, Section~\ref{sec:proof} solves the geometric class number one problem for central leaves (Theorem~\ref{thm:main}.(2)), using mass formulae for the case $g=3$ (Subsection~\ref{ssec:g3}) and explicit computations on Ekedahl-Oort strata for the hardest case $g = 4$ (Subsection~\ref{ssec:g4}).

In future work, we plan to extend the techniques of this work to prove that every geometric generic polarised supersingular abelian varieties of dimension $g>1$ in odd characteristic has automorphism group equal to $\{\pm 1\}$, known as a conjecture of Oort.

\subsection*{Acknowledgements}
The first author is supported by JSPS Kakenhi Grants JP19K03424 and JP20H00115. The second author is supported by the Dutch Research Council (NWO) through grants VI.Veni.192.038 and VI.Vidi.223.028.
The third author is partially supported by
the MoST grant 109-2115-M-001-002-MY3 and Academia Sinica grant AS-IA-112-M01.
We thank Brendan Hassett and Akio Tamagawa for helpful discussions.
The authors thank the referees for their careful reading and helpful comments that have improved the manuscript significantly.  

\section{The arithmetic theory}\label{sec:Arith}

\subsection{Uniqueness of orthogonal decomposition}\label{ssec:RSarith}\

Let $F$ be a totally real algebraic number field, 
and let $B$ be either $F$ itself,  a 
CM field over~$F$ (i.e., a totally imaginary
quadratic extension of $F$), or a totally definite quaternion algebra central over~$F$ (i.e., such that any simple component of $B\otimes \R$ is a 
division algebra). These~$B$ are typical $\Q$-algebras for considering positive-definite Hermitian $B$-modules. 
We refer to Remark~\ref{rem:fromintro} for more general algebras $B$ that one may consider.
We may regard~$B^n$ as a left $B$-vector space. As a vector 
space over $F$, we see that $B^n$ can be identified with~$F^{en}$, where $e=1$, $2$, or $4$ according to the choice of $B$ made above. 
Let $O_F$ be the ring of integers of $F$. 
A lattice in $B^n$ is a finitely generated $\Z$-submodule $L
\subseteq B^n$ such that $\Q L=B^n$ (i.e., $L$ contains a basis of $B^n$ over $\Q$); it is called an $O_F$-lattice if $O_F L \subseteq L$.
A  subring $\cO$ of~$B$ is called an order of $B$ if it is a lattice in $B$; $\cO$ is called an $O_F$-order if $\cO$ also contains~$O_F$.
Any element of $\cO$ is integral over $O_F$. 

We fix an order $\cO$ of $B$. 
Put $V=B^n$ and let $f:V\times V\rightarrow B$ be a  
quadratic form, a Hermitian form, or a quaternion Hermitian form according to whether $B=F$, $B$ is CM, or $B$ is quaternionic.
This means that $f$ satisfies
\begin{equation}\label{eq:hermitian}
    \begin{split}
     f(ax,y) & =af(x,y) \qquad  \text{ for any $x$, $y\in V$, $a\in B$}, \\
f(x_1+x_2,y)& =f(x_1,y)+f(x_2,y) \quad  \text{ for any $x_i$, $y \in V$},\\
f(y,x) & = \ol{f(x,y)} \qquad  \text{ for any $x$, $y \in V$},    
    \end{split}
\end{equation}
where $x\mapsto \bar x$ is the canonical involution of $B$ over $F$, that is, 
the trivial map for $F$, the complex conjugation for a fixed embedding $B \subseteq \C$
if $B$ is a CM field,  
or the anti-automorphism of $B$ of order~$2$ such that 
$x+\bar x=\mathrm{Tr}_{B/F}(x)$ for the reduced trace $\mathrm{Tr}_{B/F}$. 
By the above properties, we have $f(x,x)\in F$ for any $x\in V$.
We assume that $f$ is totally positive, that is, for any $x\in V$ and
for any embedding $\sigma:F\rightarrow \R$, we have $f(x,x)^{\sigma}>0$ unless $x=0$. 
A lattice $L\subseteq V$ is said to be a left $\cO$-lattice if $\cO L\subseteq L$. An $\cO$-submodule~$M$ of an $\cO$-lattice $L$ is called an $\cO$-sublattice of $L$; in this case, $M$ is an $\cO$-lattice in the $B$-module $B M$ of possibly smaller rank.
We say that a left $\cO$-lattice $L\neq 0$ is indecomposable if whenever
$L=L_1+L_2$ and $f(L_1,L_2)=0$ for some left $\cO$-lattices $L_1$ and $L_2$, then 
$L_1=0$ or $L_2=0$. 

For quadratic forms over $\Q$,
the following theorem is in \cite[Theorem 6.7.1, p.~169]{kitaoka} and 
\cite[Satz 27.2]{kneser}. The proof for the general case is almost the same and was also given in \cite[Theorem~2.4.9] {KirschmerHab} where the order $\cO$ is maximal.

\begin{theorem}\label{orthogonal}
Assumptions and notation being as above, 
any left $\cO$-lattice $L\subseteq B^n$ has an orthogonal decomposition
\[
L=L_1\perp \cdots \perp L_r
\]
for some indecomposable left $\cO$-sublattices $L_i$.
The set of lattices $\{L_i\}_{1\leq i\leq r}$ is uniquely determined by $L$.
\end{theorem}

\begin{proof}
Any non-zero $x \in L$ is called primitive if there are no $y$,$z\in L$ such that 
$y\neq 0$, $z\neq 0$, and $x=y+z$ with $f(y,z)=0$. 
First we see that any $0\neq x\in L$ is a finite sum of primitive elements of $L$.
If $x$ is not primitive, then we have $x=y+z$ with $0\neq y$, $z\in L$ with 
$f(y,z)=0$.
So we have $f(x,x)=f(y,y)+f(z,z)$ and hence
\[
\mathrm{Tr}_{F/\Q}(f(x,x))=\mathrm{Tr}_{F/\Q}(f(y,y))+\mathrm{Tr}_{F/\Q}(f(z,z)).
\]
Since $f$ is totally positive, we have 
$\mathrm{Tr}_{F/\Q}(f(x,x))=\sum_{\sigma:F\rightarrow \R}f(x,x)^{\sigma}=0$ if and only if 
$x=0$. So we have $\mathrm{Tr}_{F/\Q}(f(y,y))<\mathrm{Tr}_{F/\Q}(f(x,x))$. 
If $y$ is not primitive, we continue the same process. 
We claim  that this process terminates after finitely many steps.
Since $L\neq 0$ is a finitely generated $\Z$-module,  
$f(L,L)$ is a non-zero finitely generated $\Z$-module. 
So the module $\mathrm{Tr}_{F/\Q}(f(L,L))$ is 
a fractional ideal of $\Z$ and we have 
$\mathrm{Tr}_{F/\Q}(f(L,L))=e\Z$ for some $0<e\in \Q$. This means that 
$\mathrm{Tr}_{F/\Q}(f(x,x))\in e\Z_{>0}$ for any $x \in L$. 
So after finitely many iterations, $\mathrm{Tr}_{F/\Q}(f(y,y))$ becomes $0$ and the claim is proved.

We say that primitive elements $x$, $y\in L$ are \emph{connected} if there are primitive elements $z_1$, $z_2$, \ldots, $z_r \in L$ such that 
$x=z_0$, $y=z_r$, and $f(z_{i-1},z_{i})\neq 0$ for $i=1$,\ldots, $r$.
This is an equivalence relation.
We denote by $K_{\lambda}$, for 
$\lambda \in \Lambda$, the equivalence 
classes of primitive elements in $L$. By definition, elements of 
$K_{\lambda_1}$ and $K_{\lambda_2}$ for $\lambda_1\neq \lambda_2$  are orthogonal.
We denote by $L_{\lambda}$ the left $\cO$-module spanned by elements of 
$K_{\lambda}$. Then we have 
\[
L=\perp_{\lambda\in \Lambda}L_{\lambda}.
\]
Since $F\cO=B$, we see that $V_{\lambda}:=FL_{\lambda}$ is a left $B$-vector space and 
$L_{\lambda}$ is an $\cO$-lattice in $V_{\lambda}$. 
Since $\dim_B \sum_{\lambda\in \Lambda}V_{\lambda}=n$, we see that 
$\Lambda$ is a finite set. Hence any primitive element in $L_{\lambda}$ belongs 
to $K_{\lambda}$. Indeed, if $y\in L_{\lambda}\subseteq L$ is primitive, 
then $y\in K_{\mu}$ for some $\mu\in \Lambda$, but if $\lambda\neq \mu$, then 
$y\in K_{\mu}\subseteq L_{\mu}$, so $y=0$, a contradiction.
Now if $L_{\lambda}=N_1\perp N_2$ for some left $\cO$-modules $N_1\neq 0$, $N_2\neq 0$,
then whenever $x+y$ with $x\in N_1$, $y\in N_2$ is primitive, we have $x=0$ or $y=0$.
So if $0\neq x \in N_1$ is primitive and if $f(x,z_1)\neq 0$ for 
some primitive element $z_1\in L_{\lambda}$, then $z_1 \in N_1$. 
Repeating the process, any $y\in K_{\lambda}$ belongs to $N_1$, so that
$N_1=L_{\lambda}$, and hence, $L_{\lambda}$ is indecomposable.

Now if $L=\perp_{\kappa \in K}M_{\kappa}$ for other indecomposable 
lattices $M_{\kappa}$ (indexed by the set $K$), then any primitive element $x$ of $L$ is contained 
in some $M_{\kappa}$ by the definition of primitivity. 
By the same reasoning as before, if $x \in M_{\kappa}$ is primitive, then 
any primitive $y\in L$ connected to $x$ belongs to $M_{\kappa}$. 
This means that there is an injection $\iota:\Lambda\rightarrow K$
such that $L_{\lambda}\subseteq M_{\iota(\lambda)}$.   
Since 
\[
L=\perp_{\lambda\in \Lambda}L_{\lambda}\subseteq \perp_{\lambda\in \Lambda}
M_{\iota(\lambda)}\subseteq L
\] 
we have $L_{\lambda}=M_{\iota(\lambda)}$ and $\iota$ is a bijection.
\end{proof}

\begin{corollary}\label{autodecomposition}
Assumptions and notation being as before, suppose that
$L$ has an orthogonal decomposition 
\[
L=\perp_{i=1}^{r}M_i
\]
where $M_i=\perp_{j=1}^{e_i}L_{ij}$ for some indecomposable left $\cO$-lattices 
$L_{ij}$ such that $L_{ij}$ and $L_{ij'}$  are isometric for any $j$, $j'$, but 
$L_{ij}$ and $L_{i'j'}$ are not isometric for $i\neq i'$. 
Then we have 
\[
\Aut(L)\cong \prod_{i=1}^{r}\Aut(L_{i1})^{e_i}\cdot S_{e_i}
\]
where $S_{e_i}$ is the symmetric group on $e_i$ letters and 
$\Aut(L_{i1})^{e_i}\cdot S_{e_i}$ is a semi-direct product where $S_{e_i}$ normalises 
$\Aut(L_{i1})^{e_i}$. 
\end{corollary}

\begin{proof}
By Theorem \ref{orthogonal}, we see that for any element $\epsilon \in \Aut(L)$, 
there exists $\tau\in S_{e_i}$ such that 
$\epsilon(L_{i1})=L_{i\tau(1)}$, so the result follows.
\end{proof} 

\begin{remark}\label{rem:product}
The proof of Theorem~\ref{orthogonal} also works in the following more general setting:  $B=\prod_i B_i$ is a finite product of $\Q$-algebras $B_i$, where $B_i$ is either a totally real field $F_i$, a CM field over $F_i$, or a totally definite quaternion algebra over $F_i$. Denote by $\bar\cdot$ the canonical involution on~$B$ and $F=\prod_i F_i$ the subalgebra fixed by $\bar\cdot$. Let $\calO$ be any order in $B$, and let $V$ be a faithful left $B$-module equipped with a totally positive Hermitian form $f$, which satisfies the conditions in~\eqref{eq:hermitian} and is totally positive on each factor in $V=\oplus V_i$ with respect to $F=\prod_i F_i$. 
\end{remark}

\begin{remark}\label{rem:fromintro}
By the Albert classification of division algebras, the endomorphism algebra $B = \End^0(A)$ of any simple abelian variety $A$ over any field $K$ is either a totally real field~$F$, a quaternion algebra over $F$ (totally definite or totally indefinite), or a central division algebra over a CM field over~$F$.
The results in this subsection apply to all these classes of algebras, except for totally indefinite quaternion algebras and non-commutative central division algebras over a CM field. Indeed, Theorem~\ref{orthogonal} provides a very general statement about unique orthogonal decomposition of lattices, which enables us to compute the automorphism groups of such lattices via Corollary~\ref{autodecomposition}.

On the geometric side however, in this paper we will be mostly interested in supersingular abelian varieties, which are by definition isogenous to a power of a supersingular elliptic curve; hence, the most important algebras for us to study are the definite quaternion $\Q$-algebras $B = \End^0(E)$ for some supersingular elliptic curve $E$ over an algebraically closed field. We specialise to these algebras in the next subsections (Subsections~\ref{ssec:massarith} and~\ref{ssec:Gaussarith}) and solve a class number one problem for these in Theorem~\ref{thm:mainarith}. And indeed, in Theorem~\ref{thm:main2} we will solve the Gauss problem for the central leaves of all supersingular abelian varieties. 

Allowing $B$ to be a more general definite quaternion $\Q$-algebra (that is, not necessarily ramified only at $\{p,\infty\}$) would prove an extension of the Gauss problem for central leaves from Siegel modular varieties to quaternionic Shimura varieties of higher degree, which are direct generalisations of Shimura curves (that is, fake modular curves).\\    
\end{remark}

\subsection{Quaternionic Hermitian groups and mass formulae}\label{ssec:massarith}\

For the rest of this section, we let $B$ be a definite quaternion $\Q$-algebra central over $\Q$ with discriminant $D$
and let $O$ be a maximal order in $B$. Then $D=q_1\cdots q_t$ is a 
product of $t$ primes, where $t$ is an odd positive integer. 
The canonical involution on $B$ is denoted by $x\mapsto \bar x$.
Let $(V,f)$ be a positive-definite quaternion Hermitian space over $B$ of
rank $n$. That is, $f$ satisfies the properties in Equation~\eqref{eq:hermitian} and 
$f(x,x)\ge 0$ for all $x\in V$ and $f(x,x)=0$ only when $x=0$.
The isomorphism class of $(V,f)$ over $B$ is uniquely
determined by $\dim_B V$. We denote by $G=G(V,f)$ the group of all
similitudes on $(V,f)$; namely,
\[ 
G=\{\alpha\in \GL_B(V): f(x \alpha,y \alpha)=n(\alpha)f(x,y) \quad \forall\, x,y\in V\ \},
\]
where $n(\alpha)\in \Q^\times$ is a scalar depending only on $\alpha$.
For each prime $p$, we write $O_p:=O\otimes_\Z \Zp$, 
$B_p:=B\otimes_\Q \Qp$ and $V_p:=V\otimes_\Q \Qp$, and let
$G_p=G(V_p,f_p)$ be the group of all similitudes on the local quaternion
Hermitian space $(V_p,f_p)$.

Two $O$-lattices $L_1$ and $L_2$ are said to be equivalent, 
denoted $L_1\sim L_2$, if there exists an element 
$\alpha\in G$ such that $L_2=L_1 \alpha$; the equivalence of two $O_p$-lattices
is defined analogously. Two $O$-lattices $L_1$ and $L_2$ 
are said to be in the same genus if $(L_1)_p\sim (L_2)_p$ for all primes~$p$. 
The norm $N(L)$ of an $O$-lattice $L$ is defined to be 
the two-sided fractional $O$-ideal generated by
$f(x,y)$ for all $x,y\in L$. If $L$ is maximal among the $O$-lattices
having the same norm $N(L)$, then it is called a maximal $O$-lattice.  
The notion of maximal $O_p$-lattices in~$V_p$ is defined analogously. Then an $O$-lattice $L$ is maximal if and only if 
the $O_p$-lattice $L_p:=L\otimes_\Z \Zp$ is maximal for all prime 
numbers $p$.
 
For each prime $p$, if $p\nmid D$, then there is only one equivalence
class of maximal $O_p$-lattices in $V_p$, represented by the 
standard unimodular lattice $(O_p^n, f=\bbI_n)$. 
If $p|D$, then there are two equivalence classes of maximal
$O_p$-lattices in $V_p$, represented by the principal lattice
$(O_p^n,f=~\bbI_n)$ and a non-principal lattice
$((\Pi_p O_p)^{\oplus (n-c)}\oplus O_p^{\oplus c},\bbJ_n)$, respectively, where  
$c=~\lfloor n/2\rfloor$, and $\Pi_p$ is a uniformising element in $O_p$ with
$\Pi_p  \ol \Pi_p=p$, and $\bbJ_n=\text{anti-diag}(1,\dots, 1)$ is the anti-diagonal
matrix of size $n$.
Thus, there are $2^t$ genera of maximal $O$-lattices in $V$ when $n\geq 2$. 

For each positive integer $n$ and a pair $(D_1,D_2)$ of positive
integers with $D=D_1D_2$, denote by $\calL_n(D_1,D_2)$ the genus
consisting of maximal $O$-lattices in $(V,f)$ of rank $n$ such 
that for all primes $p|D_1$ (resp.~$p|D_2$) the $O_p$-lattice
$(L_p,f)$ belongs to the principal class (resp.~ the
non-principal class). We denote by $[\calL_n(D_1,D_2)]$ the set of
equivalence classes of lattices in $\calL_n(D_1,D_2)$ and by
$H_n(D_1,D_2):=\# [\calL_n(D_1,D_2)]$ the class number of the genus
$\calL_n(D_1,D_2)$. The mass $M_n(D_1,D_2)$ of $[\calL_n(D_1,D_2)]$ is defined by
\begin{equation}
  \label{eq:Mass}
  M_n(D_1,D_2)=\Mass([\calL_n(D_1,D_2)]):=\sum_{L\in
    [\calL_n(D_1,D_2)]} \frac{1}{|\Aut(L)|},
\end{equation}
where $\Aut(L):=\{\alpha\in G: L\alpha=L\}$. Note that if $\alpha\in \Aut(L)$ then
$n(\alpha)=1$, because $n(\alpha)>0$ and $n(\alpha)\in \Z^\times=\{\pm 1 \}$.

Let $G^1:=\{\alpha\in G: n(\alpha)=1\}$. The class number and mass for a $G^1$-genus
of $O$-lattices are defined analogously to the case of $G$:
two $O$-lattices $L_1$ and $L_2$ are said to be isomorphic, 
denoted $L_1\simeq L_2$, if there exists an element 
$\alpha\in G^1$ such that $L_2=L_1 \alpha$; similarly,
two $O_p$-lattices $L_{1,p}$ and $L_{2,p}$ are said to be isomorphic,
denoted $L_{1,p}\simeq L_{2,p}$ if there exists an element $\alpha_p\in G^1_p$
such that $L_{2,p}=L_{1,p} \alpha_p$.
Two $O$-lattices $L_1$ and $L_2$ 
are said to be in the same $G^1$-genus if $(L_1)_p\simeq (L_2)_p$
for all primes $p$. 
We denote by $\calL_n^1(D_1,D_2)$ the $G^1$-genus
which consists of maximal $O$-lattices in $(V,f)$ of rank $n$ satisfying 
\[ (V_p,f_p)\simeq
  \begin{cases}
    (O_p^n,\bbI_n) & \text{for $p\nmid D_2$}; \\
    ((\Pi_p O_p)^{n-c}\oplus O_p^c,\bbJ_n) & \text{for $p\mid D_2$}, \\
  \end{cases}
\]
where $c:=\lfloor n/2\rfloor$.  
We denote by $[\calL_n^1(D_1,D_2)]$ the set of isomorphism  classes of $O$-lattices in $\calL_n^1(D_1,D_2)$ and by
$H^1_n(D_1,D_2):=\# [\calL^1_n(D_1,D_2)]$ the class number of 
the $G^1$-genus $\calL_n^1(D_1,D_2)$. Similarly, the mass $M^1_n(D_1,D_2)$ of $[\calL^1_n(D_1,D_2)]$ is defined by
\begin{equation}
  \label{eq:Mass1}
  M^1_n(D_1,D_2)=\Mass([\calL^1_n(D_1,D_2)]):=\sum_{L\in
  [\calL^1_n(D_1,D_2)]} \frac{1}{|\Aut_{G^1}(L)|}, 
\end{equation}
where $\Aut_{G^1}(L):=\{\alpha\in G^1: L\alpha=L\}$, which is also equal to $\Aut(L)$. 

\begin{lemma}\label{lm:GvsG1}
The natural map $\iota:[\calL^1_n(D_1,D_2)]\to [\calL_n(D_1,D_2)]$ is a bijection. In particular, we have the equalities
\begin{equation}
  \label{eq:GvsG1}
  M^1_n(D_1,D_2)=M_n(D_1,D_2) \quad \text{and}\quad H^1_n(D_1,D_2)=H_n(D_1,D_2). 
\end{equation}
\end{lemma}
\begin{proof}
  Fix an $O$-lattice $L_0$ in $\calL_n(D_1,D_2)$ and
  regard $G$ and $G^1$ as algebraic groups over $\Q$. 
  Denote by $\wh \Z=\prod_{\ell} \Z_\ell$ the profinite completion of
$\Z$ and by $\A_f=\wh \Z\otimes_{\Z} \Q$ the finite adele ring of $\Q$. By the definition of $G$-genera, the right action of $G(\A_f)$ on $\calL_n(D_1,D_2)$ is transitive, and it induces an isomorphism $\calL_n(D_1,D_2)\simeq U_{D_1,D_2} \backslash G(\A_f)$, where $U_{D_1,D_2}$ is the stabiliser of $L_0\otimes \wh \Z$ in $G(\A_f)$. Since two lattices are isomorphic if and only if they differ by the action of an element in $G(\Q)$, we obtain an isomorphism of pointed sets 
\[  [\calL_n(D_1,D_2)]\simeq U_{D_1,D_2} \backslash G(\A_f)/G(\Q). \] 
Similarly, we also obtain an isomorphism 
\[ [\calL^1_n(D_1,D_2)]\simeq 
    U_{D_1,D_2}^1  \backslash G^1(\A_f)/G^1(\Q), \]  
where $U_{D_1,D_2}^1:=U_{D_1,D_2}\cap G^1(\A_f)$.    
By the construction of these isomorphisms, the natural map $\iota:[\calL^1_n(D_1,D_2)]\to [\calL_n(D_1,D_2)]$ is nothing but the map 
\[ 
\iota: U_{D_1,D_2}^1  \backslash G^1(\A_f)/G^1(\Q) \to U_{D_1,D_2} \backslash G(\A_f)/G(\Q)
\]
induced 
by the inclusion map $G^1(\A_f)\embed G(\A_f)$. 
The map $n$ induces a surjective map $U_{D_1,D_2} \backslash G(\A_f)/G(\Q)\to n(U_{D_1,D_2})\backslash \A_f^\times/\Q^\times_+$. One shows that $n(U_{D_1,D_2})=\wh \Z^\times$ so the latter term is trivial. Then every double coset in $U_{D_1,D_2} \backslash G(\A_f)/G(\Q)$ is represented by an element of norm one. Therefore, $\iota$ is surjective. Let $g_1,g_2\in G^1(\A_f)$ such that $\iota [g_1]=\iota[g_2]$ in the $G$-double coset space. Then $g_1=u  g_2 \gamma $ for some $u\in U_{D_1,D_2}$ and $\gamma\in G(\Q)$. Applying $n$, one obtains $n(\gamma)=1$ and hence $n(u)=1$. This proves the injectivity of $\iota$.  
\end{proof}

For each $n\geq 1$, define
\begin{equation}
  \label{eq:vn}
  v_n:=\prod_{i=1}^n \frac{|\zeta(1-2i)|}{2},
\end{equation}
where $\zeta(s)$ is the Riemann zeta function. For each prime $p$ and
$n\ge 1$, define 
\begin{equation}
  \label{eq:Lnp}
  L_n(p,1):=\prod_{i=1}^n (p^i+(-1)^i)
\end{equation}
and 
\begin{equation}
  \label{eq:L*np}
  L_n(1,p):=
  \begin{cases}
    \prod_{i=1}^c (p^{4i-2}-1) & \text{if $n=2c$ is even;} \\
    \frac{(p-1) (p^{4c+2}-1)}{p^2-1} \cdot \prod_{i=1}^c (p^{4i-2}-1) & \text{if $n=2c+1$ is odd.}   
  \end{cases}
\end{equation}

\begin{proposition}\label{prop:max_lattice} We have
\begin{equation}
  \label{eq:Massformula}
  M_n(D_1,D_2)=v_n \cdot \prod_{p|D_1} L_n(p,1) \cdot \prod_{p|D_2}
  L_n(1,p). 
\end{equation}  

\end{proposition}
\begin{proof}
  When  $(D_1,D_2)=(D,1)$, the formula \eqref{eq:Massformula}  is proved in
  \cite[Proposition~9]{hashimoto-ibukiyama:1}. By Lemma~\ref{lm:GvsG1}, we may replace
  $M_n(D_1,D_2)$ by $M^1_n(D_1,D_2)$ in \eqref{eq:Massformula}. 
  Using the definition, the mass $M^1_n(D_1,D_2)$ can be also interpreted as the volume of the compact set $G^1(\A_f)/G^1(\Q)$ with respect to the Haar measure of $G^1(\A_f)$ which takes the value one on $U_{D_1,D_2}^1$.
  Using this property, we obtain 
\[ \frac{M^1_n(D_1,D_2)}{M^1_n(D,1)}=\frac{\vol(U^1_{D,1})}{\vol(U^1_{D_1,D_2})} 
\]
for any Haar measure on $G^1(\A_f)$.  
It follows that
  \begin{equation}
    \label{eq:massquot}
    \frac{M^1_n(D_1,D_2)}{M^1_n(D,1)}=\prod_{p|D_2} \frac{\vol(\Aut_{G^1_p}(O_p^n,\bbI_n))}{\vol(\Aut_{G^1_p}((\Pi_pO_p)^{n-c}\oplus O_p^c,\bbJ_n))},
  \end{equation}
  where $c=\lfloor n/2\rfloor$ and where $\vol(U_p^1)$ denotes the volume of
  an open compact subgroup $U_p^1\subseteq G^1_p$ for a Haar measure on
  $G^1_p$. The right hand side of \eqref{eq:massquot} also does not depend
  on the choice of the Haar measure. It is easy to see that the dual  
lattice $((\Pi_pO_p)^{n-c}\oplus   O_p^c)^\vee$ of
  $(\Pi_pO_p)^{n-c}\oplus   O_p^c$ with respect to $\bbJ_n$ 
  is equal to $O_p^{c}\oplus (\Pi_p^{-1} O_p)^{n-c}$. Therefore, 
  \[ 
   \Aut_{G^1_p}((\Pi_pO_p)^{n-c}\oplus O_p^c,\bbJ_n)=
  \Aut_{G^1_p}((\Pi_pO_p)^{c}\oplus O_p^{n-c},\bbJ_n).
   \] 
  In Subsection~\ref{ssec:sspmass} we shall see a connection between $M^1_n(p,1)$ or $M^1_n(1,p)$ and certain masses in geometric terms.
  In the notation of Theorem~\ref{thm:sspmass}, which is a reformulation of \cite[Proposition~3.5.2]{harashita}, we have
  \begin{equation}
    \label{eq:localquot}
    \frac{\vol(\Aut_{G^1_p}(O_p^n,\bbI_n))}{\vol(\Aut_{G^1_p}((\Pi_pO_p)^{c}\oplus O_p^{n-c},\bbJ_n))}=\frac{\Mass(\Lambda_{n,p^c})}{\Mass(\Lambda_{n,p^0})}
    =\frac{L_{n,p^c}}{L_{n,p^0}}=\frac{L_n(1,p)}{L_n(p,1)}
\end{equation}
by \eqref{eq:npgc}. Then Equation~\eqref{eq:Massformula} follows from \eqref{eq:massquot}, \eqref{eq:localquot}, and \eqref{eq:Massformula} for $(D_1,D_2)=(D,1)$.  
\end{proof}

\subsection{The Gauss problem for definite quaternion Hermitian maximal lattices}\label{ssec:Gaussarith}\

In this subsection we determine for which $n$ and $(D_1,D_2)$ the class
number $H_n(D_1,D_2)$ is equal to one. 
The Bernoulli numbers $B_n$ are defined by (cf. \cite[p.~91]{serre:arith})
\begin{equation}
  \label{eq:Bernoulli}
  \frac{t}{e^t-1}=1-\frac{t}{2} +\sum_{n=1}^\infty  B_{2n}
  \frac{t^{2n}}{(2n)!}. 
\end{equation}
For each $n\ge 1$, we have
\begin{equation}
  \label{eq:zeta2n}
  B_{2n}=(-1)^{(n+1)} \frac{2 (2n)!}{(2\pi)^{2n}} \zeta(2n)
\end{equation}
and 
\begin{equation}
  \label{eq:zeta1-2n}
  \frac{|\zeta(1-2n)|}{2} =
  \frac{|B_{2n}|}{4n}=\frac{(2n-1)!\zeta(2n)}{(2\pi)^{2n}} .
\end{equation}
Below is a table of values of $|B_{2n}|$ and $|\zeta(1-2n)|/2$:

\begin{center}
\begin{tabular}{|c|c|c|c|c|c|c|c|c|c|c|c|c|}
\hline
$n$ & 1 & 2 & 3 & 4 & 5 & 6 & 7 & 8 & 9 & 10 & 11 & 12 \\ \hline
$|B_{2n}|$ & $\frac{1}{6}$ & $\frac{1}{30}$ & $\frac{1}{42}$
& $\frac{1}{30}$ & $\frac{5}{66}$ & $\frac{691}{2730}$ 
& $\frac{7}{6}$ & $\frac{3617}{510}$ & $\frac{43867}{798}$ 
& $\frac{174611}{330}$ & $\frac{864513}{138}$ &
$\frac{236364091}{2730}$ 
\\ \hline
$\frac{|\zeta(1-2n)|}{2}$ & $\frac{1}{24}$ & $\frac{1}{240}$ 
& $\frac{1}{504}$
& $\frac{1}{480}$ & $\frac{1}{264}$ & $\frac{691}{2730\cdot 24}$ 
& $\frac{1}{24}$ & $\frac{3617}{510\cdot 32}$ & $\frac{43867}{798\cdot 36
  }$ 
& $\frac{174611}{330\cdot 40}$ & $\frac{864513}{138\cdot 44}$ &
$\frac{236364091}{2730\cdot 48}$ 
\\ \hline
\end{tabular}
\end{center}
We have (cf.~\eqref{eq:vn})
\begin{equation}
  \label{eq:valuevn}
  \begin{split}
  &v_1=\frac{1}{2^3\cdot 3}, \quad  v_2=\frac{1}{2^7\cdot 3^2\cdot
5}, \quad v_3=\frac{1}{2^{10}\cdot 3^4 \cdot
5\cdot 7}, \\ 
&v_4=\frac{1}{2^{15}\cdot 3^5 \cdot
5^2\cdot 7}, \quad  v_5=\frac{1}{2^{18}\cdot 3^6 \cdot
5^2\cdot 7\cdot 11}.    
  \end{split}
\end{equation}
   
\begin{lemma}\label{lem:vn}
If $n\geq 6$, then either the numerator of $v_n$ is not one or $v_n>1$. 
\end{lemma}

\begin{proof}
Put $A_n=|\zeta(1-2n)|/2$. 
First, by 
  \[ \zeta(2n)<1+\int_{2}^\infty
  \frac{1}{x^{2n}}dx=1+\frac{2^{1-2n}}{2n-1}, \]
and since $\zeta(2n+2) > 1$, we have
\[ \frac{A_{n+1}}{A_n}> \frac{(2n+1)(2n)}{(2\pi)^2\cdot
  \zeta(2n)}> \left (\frac{2n}{2\pi}\right )^2
\cdot \frac{1+\frac{1}{2n}}{1+\frac{2^{1-2n}}{2n-1}}>1 \quad
\text{for $n\ge 4$}.  \]
From the table and the fact 
that $A_n$ is increasing for $n\ge 4$ which we have just proved, 
we have
\[ v_n=\prod_{i=1}^6 A_i \cdot \prod_{i=7}^{11} A_i \cdot 
\prod_{i=12}^n A_i
> \frac{1}{504^6}\cdot 1 \cdot (1803)^{n-11} \quad \text{for $n\ge 12$,} \]
since it follows from the table that $A_1, \ldots, A_6 \ge \frac{1}{504}$ and $A_{12} > 1803$.
 Thus, $v_n>1$ for $n\geq 17$.  

By a classical result of Clausen and von Staudt (see \cite[Theorem 3.1, p.~41]{AIK14}),  $B_{2n}\equiv -\sum_{(p-1)|2n} (1/p) \mod 1$ where $p$ are primes. So if $n\le 17$ (even for $n\le 344$), then $B_{2n}$ has denominators only for primes such that $p-1\le 34$ (or $p-1 \le 344\cdot 2$) and this does not include $691$. Thus, for $6\le n\le 344$, we have $691|v_n$. This proves the lemma.
\end{proof}

\begin{corollary}\label{cor:ge6}
  For $n\geq 6$, we have $H_n(D_1,D_2)>1$.
\end{corollary}

\begin{proof}
  By Lemma~\ref{lem:vn}, either $v_n>1$ or 
  the numerator of $v_n$ is not one. From the mass formula \eqref{eq:Mass}, 
  either $M_n(D_1,D_2)>1$ or the numerator of $M_n(D_1,D_2)$
  is not one. Therefore, $H_n(D_1,D_2)>1$. 
\end{proof}

\begin{proposition}\label{prop:np2}
  We have $H_3(2,1)=1$, $H_3(1,2)=1$, and $H_4(1,2)=1$.
\end{proposition}

\begin{proof}
It follows from Proposition~\ref{prop:max_lattice} and Equations~\eqref{eq:L*np} and~\eqref{eq:valuevn} that 
\[
M_3(1,2) = \frac{1}{2^{10} \cdot 3^2 \cdot 5} \qquad \text{ and } \qquad M_4(1,2) = \frac{1}{2^{15}\cdot 3^2 \cdot 5^2}.
\]
It follows from \cite[p.~699]{hashimoto-ibukiyama:2}, cf.~\cite[Section 5]{ibukiyama}, that the unique lattice $(L,h)$ in the non-principal genus $H_2(1,2)$ has an automorphism group of cardinality $1920 = 2^7 \cdot 3 \cdot 5$.

Consider the lattice $(O,p\mathbb{I}_1) \oplus (L, h)$ contained in $\calL_3(1,2)$. By Corollary~\ref{autodecomposition} we see that
\[
\Aut((O,p\mathbb{I}_1) \oplus (L, h)) \simeq \Aut((O,p\mathbb{I}_1)) \cdot \Aut((L, h)) = O^{\times} \cdot \Aut((L,h)).
\]
Since $O^{\times} = E_{24} \simeq \SL_2(\F_3)$ has cardinality $24$ (cf.~\cite[Equation~(57)]{karemaker-yobuko-yu}), it follows that 
\[
\vert \Aut((O,p\mathbb{I}_1) \oplus (L, h)) \vert = 24 \cdot 1920 = 2^{10} \cdot 3^2 \cdot 5 = \frac{1}{M_3(1,2)},
\]
showing that the lattice $(O,p\mathbb{I}_1) \oplus (L, h)$ is unique and hence that $H_3(1,2) = 1$.

Next, consider the lattice $(L, h)^{\oplus 2}$ contained in $\calL_4(1,2)$. Again by Corollary~\ref{autodecomposition} we see that
\[
\Aut((L, h)^{\oplus 2}) \simeq \Aut((L, h))^2 \cdot C_2
\]
which has cardinality 
\[
1920^2 \cdot 2 = 2^{15} \cdot 3^2 \cdot 5^2 = \frac{1}{M_4(1,2)},
\]
showing that also $(L, h)^{\oplus 2}$ is unique and therefore $H_4(1,2) = 1$.
Finally, we compute that
\[ 
M_3(2,1)=\frac{1}{2^{10}\cdot 3^4}=\frac{1}{24^3 \cdot 3!}=\frac{1}{|\Aut(O^3,\bbI_3)|},
\  \text{and therefore}\   H_3(2,1)=1. 
\]
\end{proof}

\begin{theorem}\label{thm:mainarith}
  The class number $H_n(D_1,D_2)$ is equal to one if and only if $D=p$
  is a prime number and one of the following holds:
\begin{enumerate}
\item $n=1$, $(D_1,D_2)=(p,1)$ and $p\in \{2,3,5,7,13\}$;
\item $n=2$, and either $(D_1,D_2)=(p,1)$ with $p=2,3$ or 
$(D_1,D_2)=(1,p)$ with $p \in \{2,3,5,7,11\}$;
\item $n=3$, and either $(D_1,D_2)=(2,1)$ or $(D_1,D_2)=(1,2)$;
\item $n=4$ and $(D_1,D_2)=(1,2)$.
\end{enumerate}
\end{theorem}

\begin{proof}
\begin{enumerate}
\item When $n=1$ we only have the principal genus class number and $H_1(D,1)$ is the class number $h(B)$ of $B$. The corresponding Gauss problem is a classical result: 
$h(B)=1$ if and only if $D\in \{2,3,5,7,13\}$; see the list in \cite[p.~155]{vigneras}. We give an alternative proof of this fact for the reader's convenience. 
  Suppose that $H_1(D,1)=1$ and $[\calL_n(D,1)]$ is represented by $L$. Then
  \begin{equation}
    \label{eq:M1}
    M_1(D,1)=\frac{\prod_{p|D} (p-1)}{24} =\frac{1}{m}, \quad \text{where $m= \vert \Aut(L)\vert \in 2\bbN $.}
  \end{equation}
The discriminant $D$ has an odd number of prime divisors, since $B$ is a definite quaternion algebra. That the numerator of $M_1(D,1)$ is $1$
  implies that 
  every prime factor $p$ of~$D$ must satisfy
      $(p-1)|24$ and hence $p\in\{2,3,5,7,13\}$. 
      Suppose that $D$ has more than one prime
      divisor; using the condition \eqref{eq:M1}, 
      $D$ must then be $2\cdot 3\cdot 7=42$. Using the class number formula
      (see \cite{eichler-CNF-1938, vigneras}, cf. Pizer~\cite[Theorem 16, p.~68]{pizer:arith})
\[ 
H_1(D,1)=\frac{\prod_{p|D} (p-1)}{12}  +\frac{1}{4} \prod_{p|D}
      \left ( 1-\left (\frac{-4}{p} \right ) \right )+\frac{1}{3} \prod_{p|D}
      \left ( 1-\left (\frac{-3}{p} \right ) \right ), 
      \]
    we calculate that $H_1(42,1)=2$. Hence, $D$ must be a prime $p$, which is in $\{2,3,5,7,13\}$. Conversely, we check that $H_1(p,1)=1$ for these primes.

\item See Hashimoto-Ibukiyama
\cite[p.~595]{hashimoto-ibukiyama:1},
\cite[p.~696]{hashimoto-ibukiyama:2}. One may still want to verify $H_2(D_1,D_2)>1$ for pairs $(D_1,D_2)$ not in the data there. Using the class number formula in \cite{hashimoto-ibukiyama:2} we compute that $M_2(1,2\cdot 3\cdot 11)=1/2$ and $H_2(1,2\cdot 3 \cdot 11)=9$. For the remaining cases, one can show that either the numerator of $M_2(D_1,D_2)$ is not equal to $1$ or $M_2(D_1,D_2)>1$, by the same argument as that used below for $n \geq 3$. 

\item[(3)+(4)]
The principal genus part for $n=3$ with $D=p$ a prime is due to Hashimoto \cite{hashimoto:g=3}, based
 on an explicit class number formula.
 We shall prove directly that for $n\geq 3$, (3)
      and (4) are the only cases for which $H_n(D_1,D_2)=1$. In particular, our proof of the principal genus part of
 (3) is independent of Hashimoto's result. 
      By
      Corollary~\ref{cor:ge6}, it is enough to treat the cases
      $n=3,4,5$, so we assume this. 
      We have $L_{n+1}(p,1)=L_n(p,1)(p^{n+1}+(-1)^{n+1})$,
      and
\[ L_2(1,p)=(p^2-1), \quad L_3(1,p)=(p-1)(p^6-1), \]
\[ L_4(1,p)=(p^2-1)(p^6-1), \quad L_5(1,p)=(p-1)(p^6-1)(p^{10}-1). \]
In particular, $(p^3-1)$ divides both $L_n(p,1)$ and $L_n(1,p)$ for
$n=3,4,5$.
Observe that if $L_n(p,1)$ or $L_n(1,p)$ has a prime factor greater than $11$,
then $H_n(D_1,D_2)>1$ for all $(D_1,D_2)$ with $p|D_1 D_2$; this follows from Proposition~\ref{prop:max_lattice} and \eqref{eq:valuevn}.
We list a prime factor $d$ of $p^3-1$ which is greater than $11$: 
\begin{center}
\begin{tabular}{ |c|c|c|c|c|c| }
 \hline
$p$        & 3  & 5  & 7  & 11 & 13   \\ \hline
$d|p^3-1$  & 13 & 31 & 19 & 19 & 61   \\ \hline
\end{tabular}
\end{center}
Thus, $H_n(D_1,D_2)>1$ for $n=3,4,5$ and $p|D$ for some prime $p$ with  $3\le p \le 13$. It remains to treat the cases $p\ge 17$ and $p=2$.
We compute that $M_3(17,1) \doteq 7.85$ and $M_4(1,17) \doteq 4.99$. One sees
that $M_3(1,17)>M_3(17,1)$, $M_5(17,1)>M_3(17,1)$ and
$M_4(17,1)>M_4(1,17)$. Therefore $M_n(p,1)>1$ and $M_n(1,p)>1$ for
$p\ge 17$. Thus, for $n=3,4,5$, $H_n(D_1,D_2)=1$ implies that $D=2$. One
checks that $31|L_5(2,1)$, $31|L_5(1,2)$ and $17|L_4(2,1)$. Thus
\[ H_5(2,1)>1, \quad H_5(1,2)>1, \quad \text{and} \quad H_4(2,1)>1. \]
It remains to show that $H_3(2,1)=1$, $H_3(1,2)=1$ and $H_4(1,2)=1$, which is done in Proposition~\ref{prop:np2}.       
\end{enumerate}
\end{proof}

\begin{remark}\label{rem:Kirschmer}
After completing this paper it came to our attention that Kirschmer also proved the unique orthogonal decomposition result (Theorem~\ref{orthogonal}) by adapting Kneser's proof, in Theorem 2.4.9 of his Habilitation \cite{KirschmerHab}. Moreover, in \cite[Chapter~9]{KirschmerHab}, he obtained more general results than Theorem~\ref{thm:mainarith}, which hold over any totally real algebraic number field $F$. When considering only maximal lattices over $F=\Q$ our result agrees with his results, although our method is different. For $n\geq 3$, we do not compute genus symbols and class numbers; instead we only use mass formulae and analyse the size and the numerator of the mass in question. This simplifies the computation and allows us to give a computer-free proof of Theorem~\ref{thm:mainarith} (of course based on earlier known results for $n\leq 2$). The same strategy is also applied in our geometric setting in Sections~\ref{sec:GMF}-\ref{sec:proof}. For this reason, we decided to keep our more elementary proof for interested readers. 
\end{remark}

\section{The geometric theory: mass formulae and class numbers}\label{sec:GMF}

\subsection{Set-up and definition of masses}\label{ssec:not}\

For the remainder of this paper,
let $p$ be a prime number, let $g$ be a positive integer, and let $k$ be an
algebraically closed field of characteristic $p$. Unless
stated otherwise, $k$ will be the field of definition of abelian varieties.

The cardinality of a finite set $S$ will be denoted by $\vert S\vert $. Let $\alpha_p$ be the unique local-local finite group scheme of order $p$ over $\Fp$;
it is defined to be the kernel of the Frobenius morphism on the additive group $\G_a$
over $\Fp$. As before, denote by $\wh \Z=\prod_{\ell} \Z_\ell$ the profinite completion of
$\Z$ and by $\A_f=\wh \Z\otimes_{\Z} \Q$ the finite adele ring of $\Q$. Let $B_{p,\infty}$ denote the definite quaternion $\Q$-algebra of discriminant $p$.
Fix a quaternion Hermitian $B_{p,\infty}$-space $(V,f)$ of rank $g$, let $G=G(V,f)$ be the quaternion Hermitian group associated to $(V,f)$ which by definition is the group of unitary similitudes of $(V,f)$, and $G^1\subseteq G$ the subgroup consisting of elements $g \in G$ of norm $n(g)=1$. We regard $G^1$ and $G$ as algebraic groups over $\Q$. 

For any integer $d\ge 1$, let $\calA_{g,d}$ denote the (coarse) moduli
space over $\Fpbar$ of $g$-dimensional polarised abelian varieties
$(X,\lambda)$ with polarisation degree $\deg(\lambda)=d^2$.
An abelian variety over~$k$ is said to be \emph{supersingular} if
it is isogenous to a product of supersingular elliptic curves; it is said to be \emph{superspecial} if it is isomorphic to a product of supersingular elliptic curves. 
For any $m \geq 0$, let $\calS_{g,p^m}$ be the
supersingular locus of $\calA_{g,p^m}$, which consists of all
polarised supersingular abelian varieties in $\calA_{g,p^m}$.
Then $\calS_g:=\mathcal{S}_{g,1}$ is the moduli space of 
$g$-dimensional principally polarised 
supersingular abelian varieties.

If $S$ is a finite set of objects with finite
automorphism groups in a specified category, 
    the \emph{mass} of $S$ 
    is defined to be the weighted
    sum
\[
 \Mass(S):=\sum_{s\in S} \frac{1}{\vert \Aut(s)\vert }.      
\]
For any $x = (X_0, \lambda_0) \in \mathcal{S}_{g,p^m}(k)$, we define 
\begin{equation}\label{eq:Lambdax}
\Lambda_{x} = \{ (X,\lambda) \in \mathcal{S}_{g,p^m}(k) : (X,\lambda)[p^{\infty}] \simeq  (X_0, \lambda_0)[p^{\infty}] \},
\end{equation}
where $(X,\lambda)[p^{\infty}]$ denotes the polarised $p$-divisible
group associated to $(X,\lambda)$. We define a group scheme $G_x$ over $\Z$ as follows. For
any commutative ring $R$, the group of its $R$-valued points is
defined by  
\begin{equation}\label{eq:aut}
G_{x}(R) = \{ \alpha \in (\text{End}(X_0)\otimes _{\mathbb{Z}}R)^{\times} : \alpha^t \lambda_0 \alpha  = \lambda_0\}. 
\end{equation}
Since any two polarised supersingular abelian varieties are isogenous,
i.e., there exists a quasi-isogeny $\varphi: X_1\to X_2$ such that $\varphi^* \lambda_2=\lambda_1$, 
the algebraic group $G_x\otimes \Q$ is independent of~$x$ (up to isomorphism) and it is known to be isomorphic to $G^1$. We shall fix an isomorphism $G_x\otimes \Q \simeq G^1$ over $\Q$ and regard $U_x:=G_x(\wh \Z)$ as an open compact subgroup of $G^1(\A_f)$. 
By \cite[Theorem 2.1]{yu:2005}, there is a natural bijection between the following pointed sets: 
\begin{equation}
  \label{eq:smf:1}
\Lambda_x \simeq G^1(\Q)\backslash G^1(\A_f)/U_x. 
\end{equation}
In particular, $\Lambda_x$ is a finite set. The mass of $\Lambda_x$ 
is then defined as  
\begin{equation}
  \label{eq:Massx}
  \mathrm{Mass}(\Lambda_{x}) = \sum_{(X,\lambda) \in \Lambda_{x}} \frac{1}{\vert
  \mathrm{Aut}(X,\lambda)\vert}.
\end{equation}
If $U$ is an open compact subgroup of $G^1(\A_f)$, the \emph{arithmetic mass} for $(G^1,U)$ is defined by
\begin{equation}
  \label{eq:arithmass}
  \Mass(G^1,U):=\sum_{i=1}^h \frac{1}{|\Gamma_i|}, \quad \Gamma_i:=G^1(\Q)\cap c_i U c_i^{-1},
\end{equation}
where $\{c_i\}_{i=1,\ldots, h}$ is a complete set of representatives of the double coset space
$ G^1(\Q)\backslash G^1(\A_f)/U$. 
The definition of $\Mass(G^1,U)$  is independent of the choices of representatives $\{c_i\}_i$. 
Then we have the equality (cf.~ \cite[Corollary 2.5]{yu:2005})
\begin{equation}
  \label{eq:smf:2}
  \Mass(\Lambda_x)=\Mass(G^1,U).  
\end{equation}

\subsection{Superspecial mass formulae}\label{ssec:sspmass}\

For each integer $c$ with $0 \leq c \leq \lfloor g/2 \rfloor$,
let $\Lambda_{g,p^c}$ denote the set of isomorphism classes of
$g$-dimensional polarised superspecial  abelian varieties $(X,
\lambda)$ whose polarisation $\lambda$ satisfies $\ker(\lambda) \simeq
\alpha_p^{2c}$. The mass of $\Lambda_{g,p^c}$ is 
\[
\mathrm{Mass}(\Lambda_{g,p^c}) = \sum_{(X,\lambda)\in \Lambda_{g,p^c}}
\frac{1}{\vert \mathrm{Aut}(X,\lambda) \vert}. 
\]

Note that the $p$-divisible group
of a superspecial abelian variety of given dimension is unique up to
isomorphism. Furthermore, the polarised $p$-divisible group associated
to any member in~$\Lambda_{g,p^c}$ is unique up to isomorphism, 
cf.~\cite[Proposition 6.1]{lioort}. Therefore,
if $x = (X_0, \lambda_0)$ is any member in $\Lambda_{g,p^c}$, then we have 
$\Lambda_x = \Lambda_{g,p^c}$ (cf.~\eqref{eq:Lambdax}). In
particular, the mass $\Mass(\Lambda_{g,p^c})$ of the superspecial locus $\Lambda_{g,p^c}$ is a special case of $\Mass(\Lambda_x)$.   

We fix a supersingular elliptic curve $E$ over
$\mathbb{F}_{p^2}$ such that its Frobenius endomorphism $\pi_E$ satisfies $\pi_E=-p$, and let ${E_k}=E\otimes_{\mathbb{F}_{p^2}} k$ (note that $k \supseteq \mathbb{F}_{p^2}$). 
It is known that every polarisation on ${E^g_k}$ is
defined over $\mathbb{F}_{p^2}$, that is, it descends
uniquely to a polarisation on $E^g$ over~$\F_{p^2}$.  
For each integer~$c$ with $0\leq c \leq \lfloor g/2 \rfloor$, we denote by $P_{p^c}(E^g)$ the set of isomorphism classes of polarisations $\mu$ on $E^g$ such that $\mathrm{ker}(\mu) \simeq \alpha_p^{2c}$; we define $P_{p^c}({E^g_k})$ similarly, and have the identification $P_{p^c}({E^g_k})=P_{p^c}(E^g)$. 
As superspecial abelian varieties of dimension $g>1$ are unique up to isomorphism, there is a bijection $P_{p^c}(E^g) \simeq \Lambda_{g,p^c}$ when $g>1$.
For brevity, we shall also write $P(E^g)$ for $P_1(E^g)$.

\begin{theorem}\label{thm:sspmass}
  For any $g \ge 1$ and $0 \leq c \leq \lfloor g/2 \rfloor$, we have
  \[ \mathrm{Mass}(\Lambda_{g,p^c})=v_g \cdot L_{g,p^c},\]
  where $v_g$ is defined in \eqref{eq:vn} and where
  \begin{equation}
    \label{eq:Lgpc}
      L_{g,p^c} =\prod_{i=1}^{g-2c} (p^i + (-1)^i)\cdot \prod_{i=1}^c
  (p^{4i-2}-1) 
 \cdot \frac{\prod_{i=1}^g
  (p^{2i}-1)}{\prod_{i=1}^{2c}(p^{2i}-1)\prod_{i=1}^{g-2c} (p^{2i}-1)}.
  \end{equation}
\end{theorem}

\begin{proof}
   This follows from \cite[Proposition
      3.5.2]{harashita} by the functional equation for $\zeta(s)$. See \cite[p.~159]{ekedahl} and \cite[Proposition 
      9]{hashimoto-ibukiyama:1} for the case where $c=0$ (the principal genus case). See
      also 
      \cite{yu2} for a geometric proof in the case where $g=2c$ (the non-principal genus case).  
\end{proof}

Clearly, $L_{g,p^0}=L_g(p,1)$ (see~\eqref{eq:Lnp}). One can also see from \eqref{eq:Lgpc} that
for $c= \lfloor g/2 \rfloor$,
\begin{equation}
  \label{eq:npgc}
L_{g,p^c}=  \begin{cases}
    \prod_{i=1}^c (p^{4i-2}-1) & \text{if $g=2c$ is even;} \\
    \frac{(p-1) (p^{4c+2}-1)}{p^2-1} \cdot \prod_{i=1}^c (p^{4i-2}-1) & \text{if $g=2c+1$ is odd,}   
  \end{cases}  
\end{equation}
and therefore $L_{g,p^c}=L_g(1,p)$, cf.~\eqref{eq:L*np}.
For $g=5$ and $c=1$, one has
\begin{equation}
  \label{eq:Lambda5p}
  \Mass(\Lambda_{5,p})=v_5 \cdot (p-1)(p^2+1)(p^3-1)(p^4+1)(p^{10}-1), 
\end{equation}
noting that this case is different from either the principal genus or the non-principal genus case. 

\begin{lemma}\label{lem:poly}
  For any $g \ge 1$ and $0 \leq c \leq \lfloor g/2 \rfloor$,
  the local component $L_{g,p^c}$ in \eqref{eq:Lgpc}
  is a polynomial in $p$ over $\Z$ of degree 
  $(g^2+4gc-8c^2+g-2c)/2$. Furthermore, the minimal degree occurs precisely when $c=0$ if $g$ is odd and when $c= g/2$ if $g$ is even. 
\end{lemma}

\begin{proof}
  It suffices to show that the term
  \[ A:=\frac{\prod_{i=1}^g
  (p^{2i}-1)}{\prod_{i=1}^{2c}(p^{2i}-1)\prod_{i=1}^{g-2c} (p^{2i}-1)} \]
is a polynomial in $p$ with coefficients in $\Z$. Notice that
$A=[g;2c]_{p^2}$, where
\[ [n;k]_q:=\frac{\prod_{i=1}^n (q^i-1)}{\prod_{i=1}^k(q^i-1)\cdot \prod_{i=1}^{n-k}(q^i-1)}, \quad n\in \bbN, \ k=0,\dots, n. \]
It is known that  $[n;k]_q\in \Z[q]$; cf.~\cite{exton}. Alternatively, one considers the recursive relation $[n~+~1~;~k]_q=[n;k]_q+q^{n-k+1} [n;k-1]_q$ and concludes that $[n;k]_q\in \Z[q]$ by induction.

The degree of $L_{g, p^c}$ is 
    \begin{equation}
    \label{eq:degree}
    \begin{split}
      &\sum_{i=1}^{g-2c} i + \sum_{i=1}^c (4i-2) + \sum_{i=g-2c+1} ^g 2i - \sum_{i=1}^{2c} 2i   \\
      &= \frac{1}{2}\left [(g-2c)(g-2c+1)+c\cdot 4c+2c\cdot(4g-4c+2)-2c(4c+2) \right ] \\
      &= \frac{1}{2}\left [ g^2+4gc-8c^2+g-2c \right ].  
    \end{split} 
  \end{equation}
  The degree is a polynomial function of degree 2 in $c$ with negative leading coefficient. So the minimum occurs  either at $c=0$ or at $c=\lfloor g/2 \rfloor$; the former happens if $g$ is odd and the latter happens
  if $g$ is even.    
\end{proof}

If $g=2m$ is even, then the polynomial $L_{g,1}$ has degree $g(g+1)/2=2m^2+m$ and $L_{g,p^m}$ has degree $2m^2$.

\subsection{Mass formulae and class number formulae for 
supersingular abelian surfaces and threefolds}

\subsubsection{Non-superspecial supersingular abelian surfaces}\label{ssec:cng2}\

Let $x=(X_0,\lambda_0)$ be a principally polarised supersingular abelian surface over $k$. If $X_0$ is superspecial, then $\Lambda_x=\Lambda_{2,p^0}$ and the class number formula for $|\Lambda_{2,p^0}|$ is obtained in \cite{hashimoto-ibukiyama:1}. We assume that $X_0$ is not superspecial, that is, $a(X_0)=1$. In this case there is a unique (up to isomorphism) polarised superspecial abelian surface $(Y_1,\lambda_1)$ such that $\ker(\lambda_1) \simeq \alpha_p^2$ and an isogeny $\phi:(Y_1,\lambda_1)\to (X_0,\lambda_0)$ of degree $p$ which is compatible with polarisations. Furthermore, there is a unique polarisation $\mu_1$ on $E^2$ such that $\ker(\mu_1) \simeq \alpha_p^2$ and $(Y_1,\lambda_1)\simeq (E^2,\mu_1)\otimes_{\F_{p^2}} k$.
Then $x$ corresponds to a point $t$ in $\bbP^1(k)=\bbP^1_{\mu_1}(k):=\{\phi_1:(E^2,\mu_1)\otimes k \to (X,\lambda) \text{ an isogeny of degree $p$} \}$, called the Moret-Bailly parameter for $(X_0,\lambda_0)$.  

The condition $a(X_0)=1$ implies that $t\in \bbP^1(k)\setminus \bbP^1(\F_{p^2})=k \setminus \F_{p^2}$. 
We consider two different cases, corresponding to the structures of $\End(X_0)$: the case $t\in k\setminus \F_{p^4}$, which we call the first case (I), and the case $t\in \F_{p^4} \setminus \F_{p^2}$, called the second case (II). The following explicit formula for the class number of a non-superspecial supersingular ``genus'' $\Lambda_x$ is due to the first-named author \cite{ibukiyama}.

\begin{theorem}\label{thm:nsspg2}
  Let $x=(X_0,\lambda_0)$ be a principally polarised supersingular abelian surface over~$k$ with $a(X_0)=1$ and let $h$ be the cardinality of $\Lambda_x$.

\begin{enumerate}
\item In case (I), i.e., when $t\in \mathbb{P}^1(k) \setminus \mathbb{P}(\F_{p^4})$, we have
  \[ h=
    \begin{cases}
      1 & \text{if $p=2$}; \\
      \frac{p^2(p^4-1)(p^2-1)}{5760} & \text{if $p\ge 3$}. 
    \end{cases}
  \]
\item In case (II), i.e., when $t\in \mathbb{P}(\F_{p^4}) \setminus \mathbb{P}(\F_{p^2})$, we have
  \[ h=
    \begin{cases}
      1 & \text{if $p=2$}; \\
      \frac{p^2(p^2-1)^2}{2880} & \text{if } p\equiv \pm 1 \bmod 5 \text{ or } p=5; \\
      1+\frac{(p-3)(p+3)(p^2-3p+8)(p^2+3p+8)}{2880} & \text{if } p\equiv \pm 2 \bmod 5. \\
    \end{cases}
  \]
\item For each case, we have $h=1$ if and only if $p=2,3$. 
\end{enumerate}
\end{theorem}
\begin{proof}
 Parts (1) and (2) follow from Theorems 1.1 and 3.6 of \cite{ibukiyama}. 
  Part (3) follows from the table in Section 1 of \cite{ibukiyama}.
\end{proof}

\begin{theorem}\label{thm:massg2}
  Let $x=(X_0,\lambda_0)$ and $t\in \bbP^1(k)$ be as in
  Theorem~\ref{thm:nsspg2}. Then
   \begin{equation}
    \label{eq:massg2}
    \Mass(\Lambda_{x})=\frac{L_p}{5760}    ,
  \end{equation}
with
  \[ L_p=
    \begin{cases}
      (p^2-1)(p^4-p^2), & \text{ if } t\in \bbP^1(\F_{p^4}) \setminus \bbP^1(\F_{p^2});\\
      2^{-e(p)}(p^4-1)(p^4-p^2) & \text{ if }t\in \bbP^1(k) \setminus \bbP^1(\F_{p^4}),\\
    \end{cases}
  \]
  where $e(p)=0$ if $p=2$ and $e(p)=1$ if $p>2$.
\end{theorem}
\begin{proof}
  See \cite[Theorem 1.1]{yuyu}; also cf.~\cite[Proposition 3.3]{ibukiyama}. 
\end{proof}

\begin{corollary}\label{cor:p2g2aut}
   Let $x=(X_0,\lambda_0)$ and $t\in \bbP^1(k)$ be as in
  Theorem~\ref{thm:nsspg2}. Assume that $p=2$. Then 
  \begin{equation}
    \label{eq:autg2}
 \vert \Aut(X_0,\lambda_0) \vert=
    \begin{cases}
      160 & \text{ if } t\in \bbP^1(\F_{p^4}) \setminus \bbP^1(\F_{p^2});\\
      32 & \text{ if } t \in \bbP^1(k) \setminus \bbP^1(\F_{p^4}).
    \end{cases}   
  \end{equation}
\end{corollary}
\begin{proof}
  By Theorem~\ref{thm:nsspg2}, we have $|\Lambda_{x}|=1$ in both cases. The mass formula  (cf.~Theorem~\ref{thm:massg2}) for $p=2$ yields
  \[ 
  \Mass(\Lambda_{x})=
    \begin{cases}
      1/160 & \text{ if } t\in \bbP^1(\F_{p^4}) \setminus \bbP^1(\F_{p^2});\\
      1/32 & \text{ if } t\in \bbP^1(k) \setminus \bbP^1(\F_{p^4}).\\
    \end{cases}
  \]
  This proves \eqref{eq:autg2}. 
\end{proof}

\subsubsection{Supersingular abelian threefolds}\label{ssec:mfg3}\

We briefly describe the framework of polarised
flag type quotients as developed in
\cite{lioort}.
Let $E/\F_{p^2}$ be the elliptic curve fixed in Subsection~\ref{ssec:sspmass}.
An $\alpha$-group of rank $r$ over an $\Fp$-scheme~$S$  
is a finite flat group scheme which is Zariski-locally 
isomorphic to $\alpha_p^r$ over an open subset. For an abelian scheme $X$ over $S$,  put
$X^{(p)}:=X\times_{S,F_S} S$, where $F_S:S\to S$ denotes the absolute Frobenius morphism on $S$. Denote by $F_{X/S}:X\to X^{(p)}$ and $V_{X/S}: X^{(p)}\to X$ the relative Frobenius and Verschiebung morphisms, respectively. If $f:X\to Y$ is a morphism of abelian varieties, we also write $X[f]$ for $\ker(f)$.  
  
\begin{definition}\label{def:PFTQ}(cf.~\cite[Section 3]{lioort}) Let $g$ be a positive integer.
\begin{enumerate}
\item For any polarisation $\mu$ on $E^g$ such that $\ker(\mu)=E^g[F]$ if $g$ is even and $\ker(\mu) = 0$ otherwise, 
a $g$-dimensional \emph{polarised flag
  type quotient (PFTQ)} with respect to $\mu$ is a chain of polarised
abelian varieties over a base $\F_{p^2}$-scheme $S$
\[ (Y_\bullet,\rho_\bullet):(Y_{g-1},\lambda_{g-1}) \xrightarrow{\rho_{g-1}} (Y_{g-2},\lambda_{g-2})\cdots \xrightarrow{\rho_{2}}(Y_1,\lambda_1) \xrightarrow{\rho_1} (Y_0, \lambda_0),\]
such that: 
\begin{itemize}
\item [(i)] $(Y_{g-1},\lambda_{g-1}) = ({E^g}, p^{\lfloor (g-1)/2 \rfloor}\mu)\times_{\Spec \F_{p^2}} S$;
\item [(ii)] $\ker(\rho_i)$ is an $\alpha$-group of rank $i$ for $1\le i\le g-1$;
\item [(iii)] $\ker(\lambda_i) \subseteq Y_i [\sfV^j \circ \sfF^{i-j}]$
  for $0\le i\le g-1$ and $0\le j\le \lfloor i/2 \rfloor$, where
  $\sfF=F_{Y_i/S}$ and $\sfV=V_{Y_i/S}$.
\end{itemize}

An isomorphism of $g$-dimensional polarised flag type quotients is a
chain of isomorphisms $(\alpha_i)_{0\le i \le g-1}$ of polarised abelian
varieties such that $\alpha_{g-1}={\rm id}_{Y_{g-1}}$. 
\item A $g$-dimensional polarised flag
  type quotient $(Y_\bullet,\rho_\bullet)$ is said to be
\emph{rigid} if 
\[ \ker(Y_{g-1}\to Y_i)=\ker (Y_{g-1}\to Y_0)\cap
Y_{g-1}[\sfF^{g-1-i}], \quad  \text{for $1\le i \le g-1$}. \]
\item Let $\mathcal{P}_{\mu}$ (resp.~$\calP'_\mu$) denote the moduli
space over $\F_{p^2}$ of 
$g$-dimensional (resp.~rigid) polarised flag
type quotients with respect to $\mu$.   
\end{enumerate}
\end{definition}

We introduce the notation and some properties of minimal isogenies for supersingular abelian varieties.

\begin{lemma}\label{lem:minisog}
  Let $X$ be a supersingular abelian variety over $k$. Then there
  exists a pair $(Y,\varphi)$, where $Y$ is a superspecial abelian
  variety and $\varphi: Y\to X$ is an isogeny, such that for any pair
  $(Y',\varphi')$ as above there exists a unique isogeny $\rho: Y'\to Y$
  such that $\varphi'=\varphi\circ \rho$. 

  Dually, there exists a pair
  $(Z,\gamma)$, where $Z$ is a superspecial abelian variety and
  $\gamma: X\to Z$ is an isogeny, such that for any pair $(Z',\gamma')$ as above
  there exists a unique isogeny $\rho: Z\to Z'$ such that
  $\gamma'=\rho\circ \gamma$.  
\end{lemma}
\begin{proof}
  See \cite[Lemma 1.8]{lioort}; also see \cite[Corollary
  4.3]{yu:mrl2010} for an independent proof. 
\end{proof}

\begin{definition}\label{def:minisog}
  Let $X$ be a supersingular abelian variety over $k$.
  We call the pair $(Y,\varphi:Y\to~X)$ or the pair $(Z,\gamma:X\to
  Z)$ as in Lemma~\ref{lem:minisog} \emph{the minimal isogeny} of $X$.
\end{definition}

\begin{proposition}\label{prop:minisoglift}
    Let $\varphi: Y\to X$ be the minimal isogeny of a supersingular abelian variety~$X$. Then every endomorphism $\sigma$ of $X$ lifts uniquely to an endomorphism $\sigma'$ of $Y$.
\end{proposition}
\begin{proof}
    This follows from the (local) statement \cite[Proposition 4.8] {yu:mrl2010}. Indeed, the element $\sigma':= \varphi^{-1} \sigma \varphi$ in $\End^0(Y)$ belongs to $\End(Y)$ if and only if $\sigma'$ belongs to $\End(Y[p^\infty])$, and the latter follows from \cite[Proposition 4.8]{yu:mrl2010}.
\end{proof}    

Now let $g=3$. According to \cite[Section 9.4]{lioort}, $\mathcal{P}_{\mu}$ is 
a two-dimensional geometrically irreducible scheme over $\mathbb{F}_{p^2}$. 
The projection to the last member gives 
a proper ${\mathbb{F}}_{p^2}$-morphism 
\begin{align*}
\mathrm{pr}_0 : \mathcal{P}_{\mu} & \to \mathcal{S}_{3,1}, \\
(Y_\bullet, \rho_\bullet) & \mapsto (Y_0, \lambda_0).
\end{align*}
Moreover, for each  principally polarised supersingular abelian
threefold $(X,\lambda)$ there exist a  principal polarisation $\mu \in P(E^3)$ and a
polarised flag type quotient $y \in \mathcal{P}_{\mu}$ such that
$\mathrm{pr}_0(y) = [(X, \lambda)] \in \mathcal{S}_{3,1}$, cf.~\cite[Proposition 5.4]{katsuraoort}. Put differently, the morphism 
\begin{equation}\label{eq:moduli}
\mathrm{pr}_0: \coprod _{\mu \in P(E^3)}\mathcal{P}_{\mu} \rightarrow \mathcal{S}_{3,1}
\end{equation}
is surjective and generically finite. 
We define the mass function on $\calP_\mu(k)$ as follows:
\begin{equation}
  \label{eq:massfcn}
  \Mass: \calP_\mu(k) \to \Q, \quad \Mass(y):=\Mass(\Lambda_x), \ x=\mathrm{pr}_0(y).
\end{equation}

We now describe the geometry of $\calP_{\mu}$. First of all, the geometric structure is independent of the choice of $\mu$;   
see \cite[Section 3.10]{lioort}.
The truncated map
\[ \pi: ((Y_2,\lambda_2) \to (Y_1,\lambda_1) \to (Y_0, \lambda_0)) \mapsto
  ((Y_2,\lambda_2) \to (Y_1,\lambda_1)) \]
induces a morphism $\pi: \calP_\mu \to \bbP^2$, since the target space is the family of subgroups of order $p$ of  $Y_2[\sfF]=E^3[\sfF]=\alpha_p^3$, which is isomorphic to $\bbP^2$. The image of $\pi$ is isomorphic to the Fermat curve $C$ defined by the equation $X_{1}^{p+1}+X_{2}^{p+1}+X_{3}^{p+1} = 0$.
Moreover, as a fibre space over $C$, $\calP_\mu$ is isomorphic to $\mathbb{P}_{C}(\mathcal{O}(-1)\oplus \mathcal{O}(1))$;
see \cite[Sections 9.3-9.4]{lioort} and~\cite[Proposition 3.5]{karemaker-yobuko-yu}. According to \cite[Section 9.4]{lioort} (cf.~\cite[Definition 3.14]{karemaker-yobuko-yu}),
there is a section $s:C\isoto T\subseteq \calP_\mu$ of $\pi$. Furthermore, one has $\calP_\mu'=\calP_\mu \setminus T$. 

We pull back the $a$-numbers of the points of $\calS_{3}$ to  the $a$-numbers of the points of $\calP_\mu$, by setting $a(y):=a(\pr_0(y))$ for $y\in \calP_\mu(k)$.
We shall write a point $y\in \calP_\mu(k)$ as $(t,u)$, where $t=\pi(y)$ and $u\in \pi^{-1}(t)=: \bbP^1_t(k)$.
\begin{lemma}\label{lm:a_strata}
Let $y=(t,u) \in \calP_\mu(k)$ be a point corresponding to a PFTQ.
\begin{enumerate}
\item If $y\in T$ then $a(y)=3$.
\item If $t\in C(\F_{p^2})$, then $a(y)\ge 2$. Moreover, $a(y)=3$ if and only if $u\in \bbP^1_t(\F_{p^2})$.
\item We have $a(y)=1$ if and only if $y\notin T$ and $t\not\in C(\F_{p^2})$.   
\end{enumerate} 
\end{lemma}
\begin{proof}
  See \cite[Sections 9.3-9.4]{lioort}.
\end{proof}

\begin{theorem}\label{introthm:a2}
  Let $y = (t,u) \in \mathcal{P}_{\mu}(k)$ be a point such that $t\in C(\F_{p^2})$.
  Then 
\[
  \mathrm{Mass}(y)=\frac{L_p}{2^{10}\cdot 3^4\cdot 5\cdot 7}, 
\]
where 
\[
 L_p=
\begin{cases}
  (p-1)(p^2+1)(p^3-1) & \text{if } u\in
  \mathbb{P}_t^1(\mathbb{F}_{p^2}); \\ 
  (p-1)(p^3+1)(p^3-1)(p^4-p^2) & \text{if }
  u\in\mathbb{P}_t^1(\mathbb{F}_{p^4})\setminus
  \mathbb{P}_t^1(\mathbb{F}_{p^2}); \\ 
  2^{-e(p)}(p-1)(p^3+1)(p^3-1) p^2(p^4-1) & \text{ if
  } u \not\in 
  \mathbb{P}_t^1(\mathbb{F}_{p^4}),
\end{cases} 
\]
where $e(p)=0$ if $p=2$ and $e(p)=1$ if $p>2$. 
\end{theorem}
\begin{proof}
  See \cite[Theorem A]{karemaker-yobuko-yu}.
\end{proof}

Theorem~\ref{introthm:a2} gives the mass formula for points with $a$-number greater than or equal to $2$. 
To describe the mass formula for points with $a$-number $1$, we need the construction of an auxiliary divisor $\calD\subseteq \calP'_\mu$, cf.~\cite[Definition 5.16]{karemaker-yobuko-yu}, and a function $d:C(k) \setminus C(\F_{p^2})\to \{3,4,5,6\}$, cf.~\cite[Definition 5.12]{karemaker-yobuko-yu} that is proven in \cite[Proposition 5.13]{karemaker-yobuko-yu} to be related to the field of definition of the parameter $t$. The function $d$ is surjective when $p\neq 2$, and it only takes value $3$ when $p=2$. 

\begin{theorem}\label{introthm:a1}
  Let $y = (t,u) \in \mathcal{P}'_{\mu}(k)$ be a point such that $t\not\in C(\F_{p^2})$.
Then 
\[
  \mathrm{Mass}(y)=\frac{p^3 L_p}{2^{10}\cdot 3^4\cdot 5\cdot 7}, 
\]
where 
\[
\begin{split}
L_p = \begin{cases}
2^{-e(p)}p^{2d(t)}(p^2-1)(p^4-1)(p^6-1) & \text{ if } y \notin \calD; \\
p^{2d(t)}(p-1)(p^4-1)(p^6-1) & \text{ if } t \notin C(\mathbb{F}_{p^6}) \text{ and } y \in \calD; \\
p^6(p^2-1)(p^3-1)(p^4-1) & \text{ if } t \in C(\mathbb{F}_{p^6}) \text{ and } y \in \calD.
\end{cases}
\end{split}
\]
\end{theorem}
\begin{proof}
  See \cite[Theorem B]{karemaker-yobuko-yu}.
\end{proof}

\begin{remark}
In \cite{{karemaker-yobuko-yu}} the authors define a stratification on $\calP_\mu$ and $\calS_{3}$ which is the coarsest one so that the mass function is constant.   
Using Theorem~\ref{introthm:a2}, the locus of $\calS_{3}$ with $a$-number $\ge 2$ decomposes into three strata: one stratum with $a$-number $3$ and two strata with $a$-number~$2$.

In the locus with $a$-number $1$, the stratification depends on $p$. When $p=2$, the $d$-value is always $3$ and Theorem~\ref{introthm:a1} gives three strata, which are of dimension $0$, $1$, $2$, respectively. When $p\neq 2$, the $d$-value $d(t)=3$ if and only if $t\in C(\F_{p^6})$, cf. \cite[Proposition 5.13]{karemaker-yobuko-yu}. In this case, Theorem~\ref{introthm:a1} says that the mass function depends only on the $d$-value of $t$ and on whether or not $y\in \calD$, and hence it gives eight strata. The largest stratum is the open subset whose preimage consists of points $y=(t,u)$ with $d(t)=6$ and $y\not\in \calD$, and the smallest mass-value stratum is the zero-dimensional locus whose preimage consists of points $y=(t,u)$ with $d(t)=3$ and $y\in \calD$. Note that the mass-value strata for which the points $y=(t,u)$ have $d$-value less than~6 and are in the divisor $\calD$ are also zero-dimensional.
Besides the  superspecial locus, in which  points have $a$-number three,
the smaller mass-value stratum with $a$-number $2$ also has dimension $0$.  

For every point $x$ in the largest stratum, one has
\begin{equation}
  \label{eq:asympopen}
  \Mass(\Lambda_x)\sim \frac{p^{27}}{2^{11}\cdot 3^4\cdot 5\cdot 7} \quad \text{as $p\to \infty$.}
\end{equation}
On the other hand, for every point $x$ in the superspecial locus, one has
\begin{equation}
  \label{eq:asympsp}
    \Mass(\Lambda_x)\sim \frac{p^{6}}{2^{10}\cdot 3^4\cdot 5\cdot 7} \quad \text{as $p\to \infty$.}
\end{equation}
\end{remark}

From all known examples, we observe that the mass $\Lambda_x$ is a polynomial function in $p$ with $\Q$-coefficients. It is plausible to expect that this holds true as well for any $x$ in $\calS_g$ and for arbitrary~$g$.
Under this assumption, it is of interest to determine the largest degree of $\Mass(\Lambda_x)$, viewed as a polynomial in $p$. For $g\le 3$, it is known \cite{yuyu,karemaker-yobuko-yu} that
the smallest degree is  $g(g+1)/2$, which occurs when $x$ is superspecial. This is expected to be true for general $g$.

\section{The geometric theory: Automorphism groups of polarised abelian varieties}\label{sec:aut}

\subsection{Powers of an elliptic curve}\label{ssec:powers}\

Let $E$ be an elliptic curve over a field $K$ with canonical polarisation $\lambda_E$ and let $(X_0, \lambda_0) = (E^n, \lambda_{\mathrm{can}})$, where $\lambda_{\mathrm{can}} = \lambda_E^n$ equals the product polarisation on $E^n$. Denote by $R:=\End(E)$ the endomorphism ring of $E$ over $K$ and by $B=\End^0(E)$ its endomorphism algebra; $B$ carries the canonical involution $a \mapsto \bar a$. 
Then $B$ is either $\Q$, an imaginary quadratic field, or the definite quaternion $\Q$-algebra $B_{p,\infty}$ of prime discriminant $p$. 
We identify the endomorphism ring $\End(E^t)=\{a^t: a\in \End(E)\}$ with $\End(E)^\opp$. Via the isomorphism $\lambda_E$, the (anti-) isomorphism $\End(E^t)=\End(E)^\opp \isoto \End(E)$ maps $a^t$ to $\lambda_E^{-1} a^t \lambda_E=\bar a$. In other words, using the polarisation $\lambda_E$ we identify $\End(E)^\opp=\End(E^t)$ with 
$\{\bar a: a \in \End(E)\}$. 

The set $\Hom(E,X_0)=R^n$ is a free right $R$-module whose elements we view as column vectors. It carries a left $\End(X_0)$-module structure and it follows that 
$\End(X_0)=\Mat_n(R)=\End_{R}(R^n)$ and $\End^0(X_0)=\Mat_n(B)=\End_{B}(B^n)$, where $B^n$ naturally identifies with $\Hom(E,X_0)\otimes \Q$. 
The map $\End(X_0) \to \End(X_0^t)$, sending $a$ to its dual $a^t$, induces an isomorphism of rings $\End(X_0)^{\rm opp}\simeq \End(X_0)$. 
The Rosati involution on $\End^0(X_0)=\Mat_n(B)$ induced by $\lambda_0$ is given by $A \mapsto A^* = \bar{A}^T$. Let $\calH_n(B)$ be the set of positive-definite Hermitian\footnote{Strictly speaking, one should call such a matrix $H$ symmetric, Hermitian or quaternion Hermitian according to whether 
$B$ is $\Q$, an imaginary quadratic field, or $B_{p,\infty}$.} matrices $H$ in $\Mat_n(B)$, satisfying $H=H^*$ and $v^* H v>0$ for every non-zero vector $v\in B^n$. 
A \emph{fractional polarisation} on an abelian variety $X$ is an element $\lambda$ in $\Hom(X,X^t)\otimes_\Z \Q$ such that there exists a positive integer $N$ for which $N \lambda$ is a polarisation on $X$.
Let $\calP(X_0)_{\Q}$ denote the set  of fractional polarisations on $X_0$.

\begin{lemma}\label{lm:PH}
The map $\lambda \mapsto \lambda_0^{-1} \lambda$ gives a bijection 
$\calP(X_0)_\Q \isoto \calH_n(B)$, under which $\lambda_0$ corresponds to the identity~${\mathbb I}_n$.
\end{lemma}
\begin{proof}
This is shown in \cite[7.12-7.14]{OortEO} for the case where $X_0$ is a superspecial abelian variety over an algebraically closed field of characteristic $p$ and the same argument holds for the present situation. 
\end{proof}

For each $H\in \calH_n(B)$, we define a Hermitian form on $B^n$ by 
\begin{equation}
    \label{eq:hBn}
    h: B^n\times B^n \to B, \quad h(v_1, v_2):= v_1 ^* \cdot H \cdot v_2, \quad \text{$v_1, v_2\in B^n$}. 
\end{equation}
If $H=\lambda_0^{-1} \lambda$ is the corresponding Hermitian form for $\lambda$, then the Rosati involution induced by~$\lambda$ is the adjoint of $h$: $A\mapsto H^{-1} \cdot A^* \cdot H$.   
The correspondence mentioned above induces an identification of automorphism groups
\begin{equation}\label{eq:AutAut0}
\mathrm{Aut}(X_0, \lambda) = \mathrm{Aut}(R^n, H):= 
\{ A \in \mathrm{GL}_n(B) : A(R^n) = R^n \text{ and } A^* \cdot H \cdot A = H \}.
\end{equation}
In particular, the identification \eqref{eq:AutAut0} induces an identification of automorphism groups
\begin{equation}\label{eq:AutAut1}
\mathrm{Aut}(X_0, \lambda_0) = \mathrm{Aut}(R^n, \mathbb{I}_n),
\end{equation}
and we know that
\begin{equation}\label{eq:AutAut2}
\begin{split}
\mathrm{Aut}(R^n, \mathbb{I}_n) 
&= \{ A \in \mathrm{GL}_n(R) : A^* \cdot A = \mathbb{I}_n \} \\
&\simeq (R^\times)^n \cdot S_n,
\end{split}    
\end{equation}
where the last equality follows from the analogous result in 
\cite[Theorem 6.1]{karemaker-yobuko-yu}. \\

If $E$ is the unique supersingular elliptic curve over $\ol \F_2$ up to isomorphism, then $R^\times=E_{24}\simeq \SL_2(\F_3)$, where $E_{24}$ is the binary tetrahedral group of order 24 (see \cite[Theorem 3.7, p.~17]{vigneras}, cf. \cite[(57)]{karemaker-yobuko-yu}). Then the automorphism group $\Aut(X_0,\lambda_0)$ has $(24)^n n!$ elements by \eqref{eq:AutAut1} and \eqref{eq:AutAut2}. We expect that this is the maximal size of $\Aut(X,\lambda)$
for any $n$-dimensional principally polarised abelian variety $(X,\lambda)$ over any field $K$. We show a partial result towards confirming this expectation.

\begin{proposition}\label{prop:maxsizeaut}
For $n\leq 3$, the number $(24)^n  n!$ is the maximal order of the automorphism group of an $n$-dimensional principally polarised abelian variety $(X,\lambda)$ over any field $K$.
\end{proposition}
\begin{proof}
Since any principally polarised abelian variety$(X,\lambda)$ is of finite type over the prime field of $K$, it admits a model $(X_1,\lambda_1)$ over a finitely generated $\Z$-algebra $S$ such that $\Aut_{K}(X,\lambda)=\Aut_S(X_1,\lambda_1)$. Taking any $\Fpbar$-point $s$ of $S$ with residue field $k(s)$, one has $\Aut_{S}(X_1,\lambda_1)\subseteq \Aut_{\Fpbar}((X_1, \lambda_1) \otimes_S k(s))$. 
Thus, without loss of generality we may assume that the ground field~$K$ is the algebraically closed field $\Fpbar$ for some prime $p$. Further we can assume that $(X,\lambda)$ is defined over a finite field $\Fq$ with $\End(X)=\End(X\otimes \Fpbar)$. 
Note that $\Aut(X,\lambda)$ is a finite subgroup of $\Aut(X)$ and hence a finite subgroup of $\End^0(X)^\times$. We will bound the size of $\Aut(X,\lambda)$ by a maximal finite subgroup $G'$ of $\End^0(X)^\times$.

When $n=1$, it is well known that $24$ is the maximal cardinality of $\Aut(E)$ of an elliptic curve $E$ over $\Fpbar$ for some prime $p$ and it is realised by the supersingular elliptic curve over $\ol \F_2$, cf.~\cite[V. Proposition 3.1, p. 145]{vigneras} .

Suppose $n=2$. If $X$ simple, then $X$ is either ordinary or almost ordinary. By Tate's Theorem, the endomorphism algebra $\End^0(X)$ is a CM field and $G'$ consists of its roots of unity, so $|G'|\le 12$.  
If $X$ is isogenous to $E_1\times E_2$ where $E_1$ is not isogenous to $E_2$, then $\End^0(E_1\times E_2)=\End^0(E_1)\times \End^0(E_2)$ and any maximal finite subgroup $G'$ of $\End^0(E_1)^\times \times \End^0(E_2)^\times$ is of the form 
$\Aut(E_1')\times \Aut(E_2')$ for elliptic curves $E_i'$ isogenous to $E_i$. This reduces to the case $n=1$ and hence $|G'|\le 24^2$.
Suppose now that $X\sim E^2$. If $L=\End^0(E)$ is imaginary quadratic, then $\End^0(X)\simeq \Mat_2(L)\simeq \End^0(\wt E^2)$ for a complex elliptic curve~$\wt E$ with CM by $L$. By \cite{birkenhake-lange}, $G'$ has order $\le 96$. 
Thus, we may assume that $E$ is supersingular so that $L$ is a quaternion algebra. 
If $X$ is superspecial, then the classification of $\Aut(X,\lambda)$ has been studied by Katsura and Oort; we have $|\Aut(X,\lambda)| \le 1152$  by \cite[Table 1, p. 137]{katsuraoort:compos87}. If~$X$ is non-superspecial, then $\Aut(X,\lambda) < \Aut(\wt X,\wt \lambda)$, where $(\wt X,\wt \lambda)$ is the superspecial abelian variety determined by the minimal isogeny of $(X,\lambda)$. The classification of $\Aut(\wt X,\wt \lambda)$ has been studied by the first author in~\cite{ibukiyama:autgp1989}. By \cite[Lemma 2.1, p.~132 and Remark 1, p.~343]{ibukiyama:autgp1989}, $\Aut(\wt X,\wt \lambda)$ has order $\le 720$ if $p>2$ and has order $1920$ if $p=2$. For the case $p=2$ and $a(X)=1$, the automorphism group $\Aut(X,\lambda)$ has order either $32$ or $160$ by Corollary~\ref{cor:p2g2aut}.
This proves the case $n=2$.

Now let $n=3$. 
Write $X\sim \prod_{i=1}^r X_i^{n_i}$ as the product of isotypic components up to isogeny, where the $X_i$'s are mutually non-isogenous simple factors. By induction and by the same argument as for $n=2$, we reduce to the case where $r=1$, that is, $X$ is elementary. Thus, we need to bound the size of maximal finite subgroups $G'$ in the simple $\Q$-algebra $\End^0(X)$. Finite subgroups in a division ring or in a certain simple $\Q$-algebra have been studied by Amitsur~\cite{amitsur:55} and Nebe~\cite{nebe:98}.
A convenient list for our case is given by Hwang-Im-Kim~\cite[Section 5]{hwang-im-kim:g3}. From this list we see that $|G'|\le 24^3\cdot 6$ and that equality occurs exactly when $\End^0(X)\simeq\Mat_3(B_{2,\infty})$; see Theorem 5.13 of \emph{loc.~cit.} This proves the proposition.   
\end{proof}

\begin{remark}\label{rem:Autchar0}
Similarly, if $E$ is the unique elliptic curve with CM by $\Z[\zeta_3]$ over $\C$ up to isomorphism, then $R^\times=\Z[\zeta_3]^\times=\mu_6$ and the automorphism group $\Aut(X_0,\lambda_0)$ has $6^n n!$ elements by \eqref{eq:AutAut1} and \eqref{eq:AutAut2}. We expect that this is the maximal size of $\Aut(X,\lambda)$
for any $n$-dimensional principally polarised abelian variety $(X,\lambda)$ over any field $K$ of characteristic \emph{zero}.
\end{remark}

\subsection{Abelian varieties isogenous to a power of an elliptic curve}\label{ssec:isogpowers}\ 

Let $E/K$, $R = \End(E)$ and $B = \End^0(E)$ be as in the previous subsection. Let $\calA$ denote the category of abelian varieties over $K$ and $\calA^\pol$ denote that of abelian varieties $(X,\lambda)$ together with a fractional polarisation over $K$; we call $(X,\lambda)$ a $\Q$-polarised abelian variety. Let $\calA_E$ (resp.~$\calA^\pol_E$) be the full subcategory of $\calA$ (resp.~of $\calA^\pol$) consisting of abelian varieties that are isogenous to a power of $E$ over $K$. By an $R$-lattice we mean a finitely presented torsion-free $R$-module. Denote by $\LatR$ and $\RLat$ the categories of right $R$-lattices and left $R$-lattices, respectively. We may write $R^\opp=\{a^T: a\in R\}$ with multiplication $a^T b^T:=(ba)^T$. 
For a right $R$-module $M$, we write $M^\opp:=\{m^T: m\in M\}$ for the left $R^\opp$-module defined by $a^T m^T=(ma)^T$ for $a\in R$ and $m\in M$. The functor $I:M\mapsto M^\opp$
induces an equivalence of categories from $\LatR$ to $\RoLat$. 
A Hermitian form on $M$ here will mean a non-degenerate Hermitian form $h: M_\Q \times M_\Q \to B$ in the usual sense, where $M_
\Q:=M\otimes \Q$. 
A Hermitian $R$-lattice is an $R$-lattice together with a Hermitian form. 
If $h$ takes $R$-values on $M$, we say $h$ is integral. 
Let $\LatR^{\rm H}$ (resp. $\RoLat^{\rm H}$) denote the category of positive-definite Hermitian right $R$-lattices (resp.~left $R^\opp$-lattices). The functor 
\[
I:\LatR^{\rm H}\to \RoLat^{\rm H}
\]
induces an equivalence of categories. 

To each $\Q$-polarised abelian variety $(X,\lambda)$ in $\calA^\pol_E$, we associate a pair $(M,h)$, where 
\begin{equation}\label{eq:M}
   M:=\Hom(E,X) 
\end{equation}
is a right $R$-lattice, and where
\begin{equation}\label{eq:h}
    h=h_\lambda: M_\Q \otimes M_\Q \to B, \quad h_\lambda(f_1,f_2):=\lambda_E^{-1} f_1^t \lambda f_2
\end{equation}
is a pairing on $M_\Q$.

\begin{lemma}\label{lm:Mh}
\begin{enumerate}
    \item The pair $(M,h)$ constructed above is a positive-definite Hermitian $R$-lattice. The Hermitian form $h$ is integral on $M$ if and only if $\lambda$ is a polarisation, and it is perfect if and only if $\lambda$ is a principal polarisation.
    \item Let $\lambda$ be a fractional polarisation on $X_0=E^n$. Then the associated Hermitian form $h_\lambda$ on $M:=\Hom(E,X_0)=R^n$ defined in \eqref{eq:h} is the Hermitian form defined in \eqref{eq:hBn}.
\end{enumerate}
\end{lemma}
\begin{proof}
\begin{enumerate} 
\item One checks that 
\begin{equation}\label{eq:h1}
h(f_1 a,f_2)=\lambda_E^{-1} a^t f_1^t \lambda f_2 =(\lambda_E^{-1} a^t \lambda_E) \lambda_E^{-1} f_1^t \lambda f_2=\bar a h(f_1,f_2)    
\end{equation}
and $h(f_1, f_2 a)=h(f_1,f_2)a$. Moreover, 
\begin{equation}\label{eq:h2}
 \ol{h(f_1,f_2)}=\lambda_E^{-1} ( \lambda_E^{-1} f_1^t \lambda f_2 )^t \lambda_E= \lambda_E^{-1} f_2^t  \lambda  f_1 \lambda_E^{-1} \lambda_E=h(f_2,f_1),   
\end{equation}
so the $R$-lattice is indeed Hermitian.
For $f\neq 0\in M$, we have $h(f,f)=\lambda_E^{-1} f^*\lambda$. Since $f^*\lambda$ is a fractional polarisation on $E$,  the composition $\lambda_E^{-1} f^*\lambda$ is a positive element in $B$, which is a positive rational number in our case. This shows that $h$ is positive-definite. The last two statements are clear as the polarisation $\lambda_E$ is principal. 
\item For $f_1, f_2 \in \Hom(E,X)_\Q=B^n$, we have $h(f_1,f_2)=\lambda_E^{-1} f_1^t \lambda_0 \lambda_0^{-1} \lambda f_2$. If we write $f_1=(a_1, \dots, a_n)^T\in B^n$ and $\lambda_0^{-1} \lambda f_2 =(b_1, \dots, b_n)^T=:\ul b$, then $\lambda_E^{-1} f_1^t \lambda_0=(\bar a_1, \dots, \bar a_n)$ and $h(f_1,f_2)=\sum_{i=1}^n \bar a_i b_i= f_1^* \cdot\ul b= f_1^*\cdot H\cdot f_2$ for $H = \lambda_0^{-1} \lambda$.    
\end{enumerate}
\end{proof}

\def\calHom{\mathcal{Hom}}

The sheaf Hom functor ${\mathcal Hom}_R(-, E):\RLat^\opp \to \calA_E$ produces a fully faithful functor whose essential image will be denoted by $\calA_{E,\mathrm{ess}}$. We refer to \cite{JKPRST} for the construction and properties of ${\mathcal Hom}_R(-, E)$. The functor $\Hom(-, E): \calA_E \to \RLat^\opp$ provides the inverse on $\calA_{E,\mathrm{ess}}$. The following result can be regarded as a polarised version of the construction in \cite{JKPRST}.

\begin{proposition}\label{prop:equiv}
The functor $(X,\lambda)\mapsto (M,h)$ introduced in Equations~\eqref{eq:M} and~\eqref{eq:h} induces an equivalence of categories 
\[
\calA_{E,\mathrm{ess}}^\pol \longrightarrow \LatR^H.
\]
Moreover, $\lambda$ is a polarisation if and only if $h$ is integral, and it is a principal polarisation if and only if $h$ is a perfect pairing on $M$. 
\end{proposition}
\begin{proof}
Let $T: \calA_E \to \calA_{E^t}$ be the functor sending $X$ to $X^t$; it induces an anti-equivalence of categories. The composition $\Hom(-, E^t)\circ T$ sends $X$ to $\Hom(X^t,E^t)$ and $I\circ \Hom(E,-)$ sends $X$ to $\Hom(E,X)^\opp$. The map that sends $f\in \Hom(E,X)$ to $f^t\in \Hom(X^t,E^t)$ gives a natural isomorphism $I\circ \Hom(E,-)\to \Hom(-, E^t)\circ T$. Restricted to $\calA_{E,\mathrm{ess}}$, the functor $\Hom(-, E^t)\circ T$ is an equivalence of categories. Therefore, $\Hom(E,-)$ induces an equivalence of categories from $\calA_{E,\mathrm{ess}}$ to $\LatR$.

The dual $M^t:=\Hom_R(M,R)$ of a right $R$-lattice $M$, which a priori is a left $R$-lattice, may be regarded as a right $R$-lattice via $f\cdot a:=\bar a f$. 
This is simply the right $R^\opp$-module $(M^t)^\opp$ with the identification 
$R^\opp=\{\bar a: a\in R\}$.
Suppose that $M = \Hom(E,X)$ is in the essential image of the equivalence, coming from some $(X,\lambda) \in \calA_{E,\mathrm{ess}}$.
We claim that the map
\begin{equation}
    \label{eq:Mt}
    \begin{split}
    \varphi: \Hom(E,X^t) &\to M^t \\
    \alpha & \mapsto (\varphi_{\alpha}: m \mapsto \lambda_E^{-1} \alpha^t m)
\end{split}
\end{equation}
is an isomorphism of right $R$-lattices. Indeed, it is injective by construction.
For surjectivity, pick any $\psi \in M^t$. Since $\psi \in \Hom(\Hom(E,X),\Hom(E,E))$ and the functor $\Hom(E,-)$ is fully faithful, there exists a unique map $\tilde{\psi}\in \Hom(X,E)$ such that $\psi(f)=\tilde{\psi} \circ f$ for all maps $f\in \Hom(E, X)=M$. 

\begin{figure}[H]\label{fig:psitilde}
    \begin{center}
    \begin{tikzcd}
    E \arrow["\psi(f)"]{r} \arrow["f"]{d} & E \\
    X \arrow["\tilde{\psi}"]{ur} & 
    \end{tikzcd}
    \end{center}
    \end{figure}
Then $\psi(m) = \tilde{\psi} m$ and we have $\tilde{\psi}^t \in \Hom(E^t, X^t)$. Considering $\alpha = \tilde{\psi}^t \lambda_E \in \Hom(E,X^t)$, it follows from the construction that
\[
\varphi_{\alpha}(m) = \lambda_E^{-1} \alpha^t m = \lambda_E^{-1} \lambda_E \tilde{\psi} m = \psi(m) 
\]
for all $m \in M$, hence $\psi = \varphi_{\alpha}$, which proves the claim.

To prove the proposition, it remains to show that for any $X \in \calA_{E,\mathrm{ess}}$ 
we have a bijection between fractional polarisations on $X$ in $\Hom(X,X^t)\otimes \Q$ and positive-definite Hermitian forms on $M_{\mathbb{Q}}$ in $\Hom(M,M^t) \otimes \mathbb{Q}$. By the definition $\Hom(E,X) =M$, the isomorphism $\Hom(E,X^t)\simeq M^t$, and the fact that the functor $\Hom(E,-)$ is fully faithful, the natural map $\Hom(X, X^t) \to \Hom(M, M^t)$ is an isomorphism. Note that the induced isomorphism $\Hom(X, X^t) \otimes \Q \to \Hom(M, M^t)\otimes\Q$ is the same as the construction in Equation~\eqref{eq:h}. 
Hence, for every positive-definite Hermitian form $h$ on $M_{\mathbb{Q}}$, there exists a unique symmetric element $\lambda_1\in \Hom(X,X^t)_\Q$ such that $h_{\lambda_1} = h$ and it suffices to show that $\lambda_1$ is a fractional polarisation on $X$.

Any quasi-isogeny $\beta: X \to E^n$ induces an isomorphism $\beta_* :M_\Q \to \Hom(E,E^n) \otimes \Q=B^n$ of $B$-modules. 
Let $\lambda := \beta_* \lambda_1$ be the pushforward map in $\Hom(E^n,(E^n)^t) \otimes \Q$, and  let $h_\lambda: B^n \times B^n \to B$ be the Hermitian form defined by \eqref{eq:h}. Then $\beta_*: (M_\Q,h) \to (B^n, h_\lambda)$ is an isomorphism of $B$-modules with pairings. Since $h$ is a positive-definite Hermitian form by assumption, so is the pairing $h_\lambda$. Let $H\in \calH_n(B)$ be the positive-definite Hermitian matrix corresponding to $h_\lambda$ with respect to the standard basis. By Lemma~\ref{lm:Mh}.(2), $H$ is equal to $\lambda_{\rm can}^{-1} \lambda$, where $\lambda_{\rm can}$ is as defined in Subsection 4.1. Since $H\in \calH_n(B)$, by Lemma~\ref{lm:PH} the map $\lambda$ is a fractional polarisation and therefore $\lambda_1$ is a fractional polarisation, as required. 
\end{proof}

By \cite[Theorem 1.1]{JKPRST} we obtain the following consequence.
The main improvement to \cite{JKPRST} is dealing with polarisations.
\begin{corollary}\label{cor:JKPRST}
Let $E$ be an elliptic curve over a finite field $K=\Fq$ with Frobenius endomorphism $\pi$ and endomorphism ring $R = \mathrm{End}(E)$. The functor $\Hom(E,-): \calA_E^\pol \to \LatR^H$ induces an equivalence of categories if and only if one of the following holds:
\begin{itemize}
    \item $E$ is ordinary and $\Z[\pi]=R$;
    \item $E$ is supersingular, $K=\Fp$ and $\Z[\pi]=R$; or
    \item $E$ is supersingular, $K=\F_{p^2}$ and $R$ has rank $4$ over $\mathbb{Z}$.
\end{itemize}
\end{corollary}

\begin{remark}
A few results similar to Proposition~\ref{prop:equiv} exist in the literature. The first case of Corollary~\ref{cor:JKPRST} is proven in \cite{KNRR}. More precisely, when $E$ is ordinary and $R = \mathbb{Z}[\pi]$, in \cite[Theorem 3.3]{KNRR} the constructions of \cite{JKPRST} are used to derive an equivalence of categories between $\mathcal{A}^{\mathrm{pol}}_E$ and $\mathrm{Lat}_R^H$.\\ 
When $R = \mathbb{Z}[\pi]$, Serre's tensor construction (cf.~\cite{Lauter}) gives an analogue of Corollary~\ref{cor:JKPRST} in some cases, when replacing ${\mathcal Hom}$ with $\otimes_R E$. The tensor construction is used in \cite[Theorem A]{Amir} for a ring $R$ with positive involution, a projective finitely presented right $R$-module $M$ with an $R$-linear map $h: M \to M^t$, and an abelian scheme $A$ over a base $S$ with $R$-action via $\iota: R \hookrightarrow \End_S(A)$ and an $R$-linear polarisation $\lambda: A \to A^t$, to prove that $h \otimes \lambda: M \otimes_R A \to M^t \otimes_R A^t$ is a polarisation if and only if $h$ is a positive-definite $R$-valued Hermitian form. Also, for a superspecial abelian variety $X$ over an algebraically closed field $k$ of characteristic $p$ it is shown in \cite[7.12-7.14]{OortEO} that the functors $X \mapsto M = \Hom(E,X)$ and $M \mapsto M \otimes_R E = X$ yield bijections between principal polarisations on $X$ and positive-definite perfect Hermitian forms on $M$.
\end{remark}

For any elliptic curve $E$ over a field $K$, we know that $B = \End^0(E)$ satisfies the conditions in Section~\ref{sec:Arith} and in particular those of Corollary~\ref{autodecomposition}.
This means that when $E$ is defined over a finite field and is in one of the cases of Corollary~\ref{cor:JKPRST}, then we may  apply the categorical constructions above to automorphism groups, in order to obtain the following result.

\begin{corollary}\label{cor:Aut}
\begin{enumerate}
\item For any $(X,\lambda) \in \calA^{\mathrm{pol}}_{E,\mathrm{ess}}$, the lattice $(M,h)$ associated to $(X,\lambda)$ admits a unique orthogonal decomposition 
\[
M = \perp_{i=1}^r \left (\perp_{j=1}^{e_i} M_{ij} \right).
\]
for which $M_{ij}$ is isomorphic to $M_{i'j'}$ if and only if $i=i'$.
Hence, we have that
\begin{equation}
    \label{eq:AutXl}
    \Aut(X,\lambda) \simeq \Aut(M,h) \simeq \prod_{i=1}^r \Aut(M_{i1}, h|_{M_{i1}})^{e_i} \cdot S_{e_i}.
\end{equation}
\item Let $E$ be an elliptic curve over a finite field $K=\mathbb{F}_q$ such that Corollary~\ref{cor:JKPRST} applies. Then for any $(X,\lambda) \in \calA^{\mathrm{pol}}_{E}$, the automorphism group $\Aut(X,\lambda)$ can be computed as in Equation~\eqref{eq:AutXl}.
\end{enumerate}
\end{corollary}

\begin{corollary}\label{cor:Autsp}
Let $R$ be a maximal order in the definite quaternion $\Q$-algebra $B_{p,\infty}$. Let ${\rm Sp}^{\rm pol}$ be the category of fractionally polarised superspecial abelian varieties over an algebraically closed field $k$ of characteristic $p$.
Then there is an equivalence of categories between ${\rm Sp}^{\rm pol}$ and $\LatR^{\rm H}$. Moreover, for any object $(X,\lambda)$ in ${\rm Sp}^{\rm pol}$, the automorphism group $\Aut(X,\lambda)$ can be computed as in Equation~\eqref{eq:AutXl}.
\end{corollary}

\begin{proof}
Choose an elliptic curve $E$ over $\F_{p^2}$ with Frobenius endomorphism $\pi=-p$ and endomorphism ring $\End(E)\simeq R$. 
Then the category $\calA_{E}$ is the same as that of superspecial abelian varieties over $\F_{p^2}$ with Frobenius endomorphism $-p$, because every supersingular abelian variety $X/\F_{p^2}$ with Frobenius endomorphism $-p$ is superspecial. To see this, we use contravariant \dieu theory. Indeed, let $M$ be the (contravariant) \dieu module of $X$; then we have $\sfF^2 M=pM$, which implies that $\sfF M=\sfV M$ and that $a(M)=g$, and hence that $M$ is superspecial. 
The functor sending each object $X$ in $\calA_E$ to $X\otimes_{\F_{p^2}} k$ induces an equivalence of categories between $\calA_E$ and the category of superspecial abelian varieties over $k$ (cf.~\cite[Proposition 5.1]{yu:iumj18}). Thus, it induces an equivalence of categories between $\calA_E^{\rm pol}$ and $\Sp^{\rm pol}$. 
By Corollary~\ref{cor:JKPRST}, there is an equivalence of categories between $\Sp^{\rm pol}$ and $\LatR^{\mathrm{H}}$. 
The last statement of the corollary follows from Corollary~\ref{cor:Aut}.  
\end{proof}

\subsection{Abelian varieties that are quotients of a power of an abelian variety over $\Fp$.}\label{ssec:powerAV} \

In this subsection only, we let $E$ denote an abelian variety over $K=\Fp$ such that its endomorphism algebra $B=\End^0(E)$ is commutative, and we put $R=\End(E)$ as before. 
This assumption on~$B$ means that $E$ does not have a repeated simple factor (i.e., it is squarefree) nor a factor that is a supersingular abelian surface with Frobenius endomorphism $\sqrt{p}$. Since every abelian variety over a finite field is of CM type, the algebra $B$ is a product of CM fields. Denote again by $a\mapsto \bar a$ the canonical involution of $B$. Let $R=\End(E)$ and fix a polarisation~$\lambda_E$ on~$E$. We will use the same notation and terminology as in previous subsections, except that we let $\calA_E$ (resp.~$\calA_E^\pol$) be the full subcategory of $\calA$ (resp.~$\calA^{\pol}$) consisting of abelian varieties which are quotients of a power of $E$ over $\Fp$.    

Recall that an $R$-module $M$ is called \emph{reflexive} if the canonical map $M\to (M^{t})^{t}$ is an isomorphism, where $M^t:=\Hom_R(M,R)$. 
If $\Z[\pi_E,\bar \pi_E]=R$, where $\pi_E$ denotes the Frobenius endomorphism of $E$, then $R$ is Gorenstein and every $R$-lattice is automatically reflexive \cite[Theorem~11 and Lemma~13]{CS15}.

\begin{theorem}\label{thm:CS+JKPRST}
{\rm (\!\cite[Theorem 8.1]{JKPRST}, \cite[Theorem 25]{CS15}) }
  Let $E$ be an abelian variety over $\Fp$ as above and assume that $\Z[\pi_E,\bar \pi_E]=R$. Then the functor ${\mathcal Hom}_R(-,E)$ induces an anti-equivalence of categories
\begin{equation}
    \label{eq:avqEn}
\RLat \longrightarrow \calA_{E}
\end{equation}
and $\Hom(-,E)$ is its inverse functor. Moreover, the functor ${\mathcal Hom}_R(-,E)$ is exact, and it is isomorphic to the Serre tensor functor $M \mapsto M^t\otimes B$. 
\end{theorem} 
 Also see \cite[Theorem 3.1]{yu:jpaa2012} for a construction of a bijection from the set of isomorphism classes in $\RLat$ to that in $\calA_E$. The category $\calA_E$ contains more objects than those which are isogenous to a power of $E$ in the case where $E$ is not simple. Note that an abelian variety $X/\Fp$ lies in $\calA_E$ if and only if there is a $\Q$-algebra homomorphism $\Q[\pi_E]\to \Q[\pi_X]$ mapping $\pi_E$ to the Frobenius endomorphism $\pi_X$ of $X$. 
Let $\RLat^{\rm f}$ (resp.~$\LatR^{\rm f}$) denote the full subcategory consisting of left  (resp.~right) $R$-lattices $M$ such that $M_\Q$ is a free $B$-module of finite rank. 
Similarly, let $\RLat^{\mathrm{f},H}\subseteq \RLat^{H}$ (resp.~$\LatR^{\mathrm{f},H}\subseteq \LatR^{H}$) be the full subcategory of positive-definite Hermitian left (resp.~right) $R$-lattices $(M,h)$ with free $B$-module $M_\Q$. The functor ${\mathcal Hom}_R(-,E)$ induces an anti-equivalence of categories from $\RLat^{\mathrm{f}}$ to the subcategory $\calA_E^{\mathrm{f}}$ consisting of abelian varieties isogenous to a power of $E$. Moreover, we prove the following result about polarised varieties.

\begin{theorem}\label{thm:pol+JKPRST}
 Let $(E,\lambda_E)$ be a principally polarised abelian variety over $\Fp$ with the assumptions as in Theorem~\ref{thm:CS+JKPRST}. Then the following hold.
 \begin{enumerate}
   \item The functor $(X,\lambda)\mapsto (M,h)$ introduced in Equations~\eqref{eq:M} and~\eqref{eq:h} induces an equivalence of categories 
\[
\calA_{E}^\pol \longrightarrow \LatR^H.
\]
   \item For any $\Q$-polarised abelian variety $(X,\lambda)$ over $\Fp$ in $\calA_E^\pol$, the automorphism group $\Aut(X,\lambda)$ can be computed as in Equation~\eqref{eq:AutXl}. 
 \end{enumerate}
\end{theorem}

\begin{proof}
\begin{itemize}
    \item[(1)] We first show that $(M,h)$ is a positive-definite Hermitian $R$-lattice. By Equations \eqref{eq:h1} and \eqref{eq:h2} in the proof of  Lemma~\ref{lm:Mh}, $h$ is Hermitian and it remains to show that $h$ is positive-definite. Let $E_i$ ($1\le i\le r$) be the simple abelian subvarieties of $E$ and let $\varphi=\sum_{i} \iota_i: \prod_{i=1}^r E_i \to E$ be the canonical isogeny with inclusions $\iota_i:E_i \subseteq E$. Then we have an inclusion $M\subseteq \bigoplus_{i=1}^r M_i$, where $M_i=\Hom(E_i, X)$. 
Let $\lambda_{E_i}$ be the restriction of the polarisation $\lambda_E$ to~$E_i$. The isogeny~$\varphi$ induces an isomorphism from $B$ onto a product $ \prod_{i} B_i$ of CM fields $B_i=\End^0(E_i)$, and the decomposition $M_\Q=\bigoplus_{i=1}^r M_{i,\Q}$ respects the decomposition $B\simeq \prod_{i} B_i$. Moreover, we have $(M_\Q,h)= \perp_{i=1}^r (M_{i,\Q}, h_i)$, where $h_i$ is the restriction of $h$, which is also induced from the polarisation $\lambda_{E_i}$. Let $f=(f_i)\in M_\Q$ be a non-zero vector. Then $h(f,f)=(h_i(f_i,f_i))_i=(\lambda_{E_i}^{-1} f_i^*\lambda )_i$ and $\lambda_{E_i}^{-1} f_i^*\lambda$ is a totally positive element whenever $f_i\neq 0$. This shows that $h$ is positive-definite. Then the same argument as in
Proposition~\ref{prop:equiv} proves the equivalence. Note that the principal polarisation $\lambda_E$ ensures there is a natural isomorphism $\Hom(E,X^t)\simeq M^t$.   
    \item[(2)] This follows from Theorem~\ref{orthogonal} and Corollary~\ref{autodecomposition} in the extended setting where $B$ is a product of CM fields; see Remark~\ref{rem:product}.
\end{itemize}
\end{proof}

\subsection{Minimal $\boldsymbol{E}$-isogenies}\label{ssec:Eisog}\

As in Subsections~\ref{ssec:powers} and~\ref{ssec:isogpowers}, we again let $E$ be an elliptic curve over a field $K$ with canonical polarisation $\lambda_E$, and $R:=\End(E)$ the endomorphism ring of $E$ over $K$, and $B=\End^0(E)$ its endomorphism algebra. In particular, we again let $\calA_E$ (resp.~$\calA^\pol_E$) be the full subcategory of $\calA$ (resp.~of $\calA^\pol$) consisting of abelian varieties that are isogenous to a power of $E$ over $K$.

In this subsection, we define a  notion of a minimal $E$-isogeny, generalising that of a minimal isogeny as introduced by Li-Oort (cf.~\cite[Section 1.8]{lioort}, also see Definition~\ref{def:minisog}), and satisfying a stronger universal property.

\begin{lemma}\label{lm:minE}
Let $X$ be an object in $\calA_E$. Then there exist an object $\wt X$ in $\calA_{E, {\rm ess}}$ and an isogeny $\gamma: X\to \wt X$ such that for any morphism $\phi: X \to Y$ with object $Y$ in $\calA_{E, {\rm ess}}$, there exists a unique morphism $\alpha: \wt X \to Y$ such that $\alpha\circ \gamma=\phi$. Dually, there exist an object $\wt X$ in $\calA_{E, {\rm ess}}$ and an isogeny $\varphi: \wt X\to X$ that satisfy the analogous universal property.
\end{lemma}
\begin{proof}
We first construct a morphism $\gamma: X \to \wt X$, where $\wt X$ is an object in $\calA_{E,{\rm ess}}$. 
It will be more convenient to adopt the contravariant functors.
Let $M:=\Hom(X,E)$ and let $\wt X:= {\mathcal Hom}_R(M,E)$. The abelian variety $\wt X$ represents the functor 
\[ S \mapsto \Hom_R(M,E(S)),  \]
for any $K$-scheme $S$. Define a morphism $\gamma:X\to \wt X$ by
\begin{equation}\label{eq:minEisog}
 \gamma: X(S)\to \wt X(S)=\Hom_R(M,E(S))\quad  \text{ mapping } \quad x \mapsto \left (\gamma_x: f \mapsto f(x)\in E(S) \right ),   
\end{equation}
for all $f\in M=\Hom(X,E)$. 

Now let $Y$ be an object in  $\calA_{E, {\rm ess}}$ and $\phi: X \to Y$ be a morphism. Using \eqref{eq:minEisog}, we also have a morphism $\gamma_Y: Y\to \wt Y$ which is an isomorphism as 
the functor $\Hom(-,E)$ induces an equivalence on $\calA_{E, {\rm ess}}$. 
The morphism $\phi: X \to Y$ induces a map $M_Y:=\Hom(Y,E) \to M=\Hom(X,E)$ by precomposition with $\phi$.
This map also induces, after applying the functor ${\mathcal Hom}_R(-,E)$,  a morphism $\beta: \wt X \to \wt Y$. We claim that the diagram 
\begin{equation}\label{eq:min_cd}
    \begin{CD}
      X @>{\gamma}>> \wt X \\
      @VV{\phi}V @VV{\beta}V \\
      Y @>{\gamma_Y}>{\sim}> \wt Y
    \end{CD}
\end{equation}
commutes; we will show this by proving it on $S$-points for any $K$-scheme $S$.
Let $x\in X(S)$ and $g:Y\to E$. We have $\beta (\gamma_x)(g)=\gamma_x(g \circ \phi)=g(\phi(x))$. On the other hand $\gamma_Y(\phi(x))(g)=g(\phi(x))$. This shows the claim. 
Let $\alpha:=\gamma_Y^{-1} \circ \beta: \wt X \to Y$, so we have 
$\alpha\circ \gamma=\phi$ by commutativity. 

Finally, take $Y = E^n$ and any isogeny $\phi: X \to E^n$ and let $\alpha:\wt X \to E^n$ be the unique morphism satisfying $\alpha\circ \gamma=\phi$. Since $\dim X=\dim \wt X=\dim E^n = n$, it follows that $\gamma$ is an isogeny. The dual construction is entirely analogous.
\end{proof}

\begin{definition}\label{def:minE}
Let $X$ be an object in $\calA_E$. We call the isogeny $\gamma: X \to \wt X$ (resp. $\varphi:\wt X \to X$) constructed in Lemma~\ref{lm:minE} the \emph{minimal $E$-isogeny of $X$} and the abelian variety $\wt X$ the \emph{$E$-hull} of $X$. 
\end{definition}

\begin{remark} \label{rem:min_isog}
If $E/K$ is a supersingular elliptic curve over an algebraically closed field $K=k$ of characteristic $p$, then $\calA_E$ is the category of supersingular abelian varieties over $k$ and $\calA_{E,{\rm ess}}$ is the category of superspecial abelian varieties over $k$. In this case, a minimal $E$-isogeny $\gamma:X\to \wt X$ or $\varphi:\wt X \to X$ of a supersingular abelian variety $X$ is precisely the minimal isogeny of $X$ in the sense of Oort, cf.~\cite[Definition~2.11]{karemaker-yobuko-yu}. By Lemma~\ref{lm:minE}, the minimal isogeny $(X, \gamma: X\to \wt X)$ satisfies the stronger universal property where the test object $\phi:X\to Y$ does not have to be an isogeny.  
\end{remark}

\begin{lemma}\label{lm:product_compatibility}
Let $X$ be an object in $\calA_{E,{\rm ess}}$. Suppose there are abelian varieties $X_1, \dots, X_r$ in~$\calA_E$ and there is an isomorphism $\phi:X_1\times \dots \times X_r \isoto X$. Then each abelian variety $X_i$ lies in  $\calA_{E,{\rm ess}}$.
\end{lemma}

\begin{proof}
According to the construction of minimal $E$-isogenies, let 
\[ M:=\Hom(X,E)\simeq \prod_{i=1}^r M_i \quad \text{and} \quad \wt X:={\mathcal Hom}_R(M,E)\simeq \prod_{i=1}^r \wt X_i, \]
where 
\[ M_i:=\Hom(X_i,E)\quad  \text{and} \quad \wt X_i:={\mathcal Hom}_R(M_i,E). \] 
By the definition of $\gamma$ in Equation~\eqref{eq:minEisog}, we have 
\[ \gamma=(\gamma_i)_i: X \simeq X_1\times \dots \times X_r \longrightarrow \wt X\simeq \wt X_1 \times \dots  \times \wt X_r, \quad \text{where} \quad \gamma_i:X_i \to \wt X_i.   \]
By applying the universal property of the minimal $E$-isogeny with $Y = X$ and $\phi = \mathrm{id}$, there is a unique isogeny $\alpha: \wt X \to X$ such that $\alpha\circ \gamma=\mathrm{id}$. This shows that $\gamma$ is an isomorphism, which means that each $\gamma_i:X_i \to \wt X_i$ is an isomorphism. In particular, every abelian variety $X_i$ lies in $\calA_{E,{\rm ess}}$.
\end{proof}

\begin{lemma}\label{lm:prod_compt2}
Let $X_1,\dots, X_r$ be objects in $\calA_{E}$. Then $(\gamma_i)_i: X=\prod_{i=1}^r X_i \to \prod_{i=1}^r \wt X_i$ is the minimal $E$-isogeny of $X$.
\end{lemma}

\begin{proof}
For any $Y\in \calA_{E,\mathrm{ess}}$, as in Equation~\eqref{eq:min_cd}, we obtain the following commutative diagram:
\begin{equation}\label{eq:min_cd2}
    \begin{CD}
      \prod_{i=1}^r X_i @>{(\gamma_i)_i}>> \prod_{i=1}^r \wt X_i \\
      @VV{\sum_i \phi_i}V @VV{\beta=\sum_i\beta_i}V \\
      Y @>{\gamma_Y}>{\sim}> \wt Y.
    \end{CD}
\end{equation}
Then the unique morphism $\alpha = \gamma_Y^{-1} \circ \beta$ satisfies the desired property $\alpha\circ {(\gamma_i)_i}=\sum_i \phi_i$. 
\end{proof}

\begin{proposition}
\label{prop:AutX2}
Let $(X,\lambda)\in \calA_{E}^{\mathrm{pol}}$ and let $\varphi:(\wt X, \wt \lambda)\to (X,\lambda)$ be the minimal $E$-isogeny of~$(X,\lambda)$, where $\wt \lambda$ is chosen to be $\varphi^* \lambda$. 
Then 
\begin{equation}\label{eq:AutX2}
\Aut(X,\lambda)=\{\alpha\in \Aut(\wt X, \wt \lambda): \alpha(H)=H\},
\end{equation}
where $H:=\ker(\varphi)$ is the kernel of the morphism $\varphi$.
\end{proposition}

\begin{proof}
By the universal property of minimal $E$-isogenies, every $\sigma_0\in\Aut(X,\lambda)$ uniquely lifts to an automorphism $\sigma\in \Aut(\wt X, \wt \lambda)$. To see this, note that we have $\sigma_0 \varphi = \varphi \sigma$, and $\wt \lambda=\varphi^* \lambda$, and 
$\sigma_0^*\lambda=\lambda$, so 
\[ \sigma^* \wt \lambda=\sigma^* \varphi^* \lambda = \varphi^* \sigma_0^* \lambda=\varphi^* \lambda =\wt \lambda. \]
Since $X=\wt X/H$, an element $\sigma\in \Aut(\wt X, \wt \lambda)$ descends to an element $\sigma_0 \in \Aut(X,\lambda)$ if and only if $\sigma(H)=H$.  
\end{proof}

\subsection{Unique decomposition property}\label{ssec:uniquedec}

\begin{definition}
Let $(X,\lambda)$ be a $\Q$-polarised abelian variety over $K$. We say $(X,\lambda)$ \emph{indecomposable} if whenever we have an isomorphism $(X_1,\lambda_1)\times (X_2,\lambda_2)=(X_1\times 
X_2, \lambda_1\times \lambda_2) \simeq (X,\lambda)$, either $\dim X_1=0$ or $\dim X_2=0$. 
\end{definition}

By induction on the dimension of $X$, every object $(X,\lambda)$ in $\calA^\pol$ decomposes into a product of indecomposable objects. 
\begin{definition}\label{def:RS}
 An object $(X,\lambda)$ in $\calA^\pol$ is said to have  \emph{the Remak-Schmidt property} if for any two decompositions into indecomposable objects $\phi=(\phi_i)_i: \prod_{i=1}^r (X_i,\lambda_i)\isoto (X,\lambda)$ and $\phi'=(\phi'_j)_j: \prod_{j=1}^s(X_j',\lambda_j')\isoto(X,\lambda)$, we have $r=s$ and there exist a permutation $\sigma\in S_r$ and an isomorphism $\alpha_i: (X_i,\lambda_i)\isoto (X_{\sigma(i)}', \lambda_{\sigma(i)}')$ such that $\phi_i=\phi_{\sigma(i)}' \circ \alpha_i$ for every $1 \le i \le r$. 
\end{definition}

\begin{theorem}
Any $\Q$-polarised abelian variety $(X,\lambda)$ in $\calA^\pol$ admits the Remak-Schmidt property.
\end{theorem}

\begin{proof}
This is nothing but a categorical formulation of the unique decomposition of $\Q$-polarised abelian varieties into indecomposable $\Q$-polarised abelian subvarieties. When the ground field~$K = \bar K$ is algebraically closed, this is proved by Debarre and by Serre \cite{Debarre}, using a result of Eichler \cite{Eichler}. 

Let $\bar K$ and $K_s$ be an algebraic and a separable closure of $K$, respectively.
Let $(X,\lambda)_{\bar K}=\prod_{i\in I} (X_i,\lambda_i)$ be a decomposition into indecomposable $\Q$-polarised abelian subvarieties over $\bar K$, where $\lambda_i:=\lambda|_{X_i}$.
Let $\varepsilon_i: X \to X_i \to X \in \End_{\bar K}(X)$ be the idempotent corresponding to the component $X_i$. Since $\varepsilon_i\in \End_{K_s}(X)$ by Chow's Theorem, each subvariety $X_i:=\ker (1-\varepsilon_i)$ is defined over $K_s$
and the theorem is proved for $K_s$. 

We now show it for an arbitrary ground field. So let $(X,\lambda)_{K_s}=\prod_{j\in J} (Y_j,\lambda_j)$ be a decomposition into indecomposable $\Q$-polarised abelian subvarieties over $K_s$. For each $\sigma\in \Gamma_K :=\Gal(K_s/K)$, we have 
\[ \prod_{j\in J}(Y_j, \lambda_j)=(X,\lambda)_{K_s}=\sigma(X,\lambda)_{K_s}=\prod_{j\in J} \sigma(Y_j,\lambda_j).\] 
By the uniqueness of the decomposition over $K_s$, we obtain that $\sigma(Y_j)=Y_{\sigma(j)}$ for a unique $\sigma(j)\in J$; this gives an action of $\Gamma_K$ on $J$. Let $I$ denote the set of $\Gamma_K$-orbits of $J$ and put $(X_i,\lambda_i):=\prod_{j\in i}(Y_j,\lambda_j)$ for each $i\in I$. Since now $(X_i,\lambda_i)$ is the smallest $\Q$-polarised abelian subvariety defined over $K$ containing $(Y_j,\lambda_j)$ for any $j \in i$, it is in particular indecomposable. Since the abelian subvarieties $Y_j$ are uniquely determined by $X$ up to permutation, so are the abelian subvarieties $X_i$.  
\end{proof}

We remark that if polarisations are not taken into consideration, then the decomposition of an (unpolarised) abelian variety into indecomposable subvarieties is far from unique; see Shioda~\cite{shioda}. 

\section{The geometric theory: the Gauss problem for central leaves}\label{sec:proof}

\subsection{First results and reductions}\label{ssec:4first}\

Let $x=[(X_0,\lambda_0)]\in \calA_g(k)$ be a point and let $\calC(x)$ be the central leaf passing through $x$.

\begin{proposition}[Chai]\label{prop:chai}
The central leaf $\calC(x)$ is finite if and only if $X_0$ is supersingular. In particular, a necessary condition for $|\calC(x)|=1$ is that $x\in \calS_{g}(k)$.   
\end{proposition}
\begin{proof}
  It is proved in \cite[Proposition 1]{chai} that the prime-to-$p$ Hecke orbit $\calH^{(p)}(X_0,\lambda_0)$ (i.e., the points obtained from $(X_0,\lambda_0)$ by polarised prime-to-$p$ isogenies) is finite if and only if $X_0$ is supersingular. Since $\calH^{(p)}(X_0,\lambda_0)\subseteq \calC(x)$, the central leaf $\calC(x)$ is finite only if $X_0$ is supersingular. When $X_0$ is supersingular, we have $\calC(x)=\Lambda_x$ by definition, and hence $\calC(x)$ is finite, cf.~\eqref{eq:smf:1} .  
\end{proof}

From now on we assume that $x\in \calS_g(k)$. In this case
\[ 
\calC(x)=\Lambda_x\simeq G^1(\Q)\backslash G^1(\A_f)/U_x,
\]
where $U_x=G_x(\wh \Z)$ is an open compact subgroup. Similarly, for 
$0\leq c\leq [g/2]$, we have
\[ \Lambda_{g,p^c}\simeq G^1(\Q)\backslash G^1(\A_f)/U_{g,p^c}, \]
where $U_{g,p^c}=G_{x_c}(\wh \Z)$ for a base point $x_c\in
\Lambda_{g,p^c}$. 

\begin{lemma}\label{lem:Lgpc}
  For every point $x\in \calS_g(k)$, there exists a (non-canonical)
  surjective morphism $\pi:\Lambda_x \twoheadrightarrow \Lambda_{g,p^c}$ for some integer
  $c$ with $0\le c\le \lfloor g/2 \rfloor$. Moreover, one can select a base point~$x_c'$ in $\Lambda_{g,p^c}$ so that $G_x(\Zp)$ is contained in $G_{x_c'}(\Zp)$
  and $\pi$
  is induced from the identity map on $G^1(\A_f)$
  \begin{equation}
    \label{eq:Gxc}
    G^1(\Q)\backslash G^1(\A_f)/U_x \longrightarrow
    G^1(\Q)\backslash G^1(\A_f)/U_{x'_c}.  
  \end{equation}
\end{lemma}

\begin{proof}
  We have 
\[ G_{x_c}(\Zp)\simeq \Aut_{G^1(\Qp)} \left ( (\Pi_p O_p)^{n-c}\oplus O_p^c,
\bbJ_g\right )=:P_c. \]
By \cite[Theorem~3.13, p.~150]{platonov-rapinchuk}, the subgroups $P_c$ for
  $c=1,\dots, [g/2]$ form a complete set of representatives of maximal parahoric subgroups of
  $G^1(\Q_p)$ up to conjugacy. 
  So $G_x(\Zp)$ is contained in $g_p^{-1} G_{x_c}(\Z_p) g_p$ for
  some integer $c$ with $0\le c\leq \lfloor g/2 \rfloor$ and some element $g_p\in
  G^1(\Qp)$. Thus,  we have a surjective map
  \begin{equation}
    \label{eq:Gx}
    G^1(\Q)\backslash G^1(\A_f)/U_x \twoheadrightarrow 
    G^1(\Q)\backslash G^1(\A_f)/g_p^{-1}U_{g,p^c}g_p  
    \xrightarrow{\cdot g_p} G^1(\Q)\backslash G^1(\A_f)/U_{g,p^c}\simeq \Lambda_{g,p^c}.   
  \end{equation}
  This gives a surjective map $\Lambda_x\twoheadrightarrow \Lambda_{g,p^c}$. The base point $x_c'$ is chosen so that $U_{x_c'}=g_p^{-1}U_{g,p^c}g_p$.
\end{proof}

Let $\varphi: \wt x=(\wt X_0,\wt \lambda_0)\to x=(X_0,\lambda_0)$ be the  minimal isogeny for $x$, as constructed in Lemma~\ref{lem:minisog} and Definition~\ref{def:minisog}. 
Then $U_x\subseteq U_{\wt x}$ and we have a surjective map $\Lambda_x\twoheadrightarrow \Lambda_{\wt x}$ which is induced from the natural map
\begin{equation}
  \label{eq:minisog}
      G^1(\Q)\backslash G^1(\A_f)/U_x \longrightarrow
      G^1(\Q)\backslash G^1(\A_f)/U_{\wt x}.
\end{equation}
If the open compact subgroup $U_{\wt x}$ is maximal, then $U_{\wt x}$ is conjugate to $U_{g,p^c}$ for some $0\le c\le \lfloor g/2\rfloor$ and the map $\pi: \Lambda_x \twoheadrightarrow \Lambda_{g,p^c}$ in Lemma~\ref{lem:Lgpc} is realised by the minimal isogeny $\varphi$. 

\begin{lemma}\label{lem:g1}
 Let $x$ be a point in $\calS_{g}(k)$. If $g=1$, then $|\Lambda_x|=1$ if and only if $p\in \{2,3,5,7,13\}$. 
\end{lemma}

\begin{proof}
In this case, the orbit $\Lambda_x$ is the supersingular locus
$\Lambda_{1,1}$. The assertion is well-known and also follows from
Theorem~\ref{thm:mainarith}.(1).
\end{proof}

\begin{lemma}\label{lem:g2}
Let $x$ be a point in $\calS_{g}(k)$. If $g=2$, then $|\Lambda_x|=1$ if and only if $p\in \{2,3\}$.
\end{lemma}

\begin{proof}
For the superspecial case, by
the first part of Theorem~\ref{thm:mainarith}.(2) 
we have $H_2(p,1)=1$ if and only if $p=2,3$. For the non-superspecial
case, it follows from Theorem~\ref{thm:nsspg2}.(3) that $|\Lambda_x|=1$ for every non-superspecial point $x\in \calS_{2}(k)$ if and only if $p=2, 3$.  
\end{proof}

\begin{lemma}\label{lem:g5+}
Let $x$ be a point in $\calS_{g}(k)$. If $g\ge 5$, then $|\Lambda_x|>1$.  
\end{lemma}

\begin{proof}
We first show that $|\Lambda_{g,p^c}|>1$ for all primes $p$ and all integers $c$ with $0\le c \le \lfloor g/2 \rfloor$. From Theorem~\ref{thm:sspmass} we have $\Mass(\Lambda_{g,p^c})=v_g \cdot L_{g,p^c}$. Using Lemma~\ref{lem:vn} and the proof of Corollary~\ref{cor:ge6}, we show that $|\Lambda_{g,p^c}|>1$ for all $g\ge 6$, all primes $p$ and all $c$.
By Theorem~\ref{thm:mainarith}, we have $|\Lambda_{5,p^0}|=H_{5}(p,1)>1$ and $|\Lambda_{5,p^2}|=H_{5}(1,p)>1$ for all primes $p$. Using Theorem~\ref{thm:sspmass} and \eqref{eq:Lambda5p}, we have $\Mass(\Lambda_{5,p})=v_5\cdot L_{5,p^1}=\Mass(\Lambda_{5,p^0})(p^5+1)$ and $(p^3-1)$ divides $L_{5,p}$. From this the same proof of Theorem~\ref{thm:mainarith} shows that $|\Lambda_{5,p}|>1$ for all primes $p$.
By Lemma~\ref{lem:Lgpc}, for every point $x\in \calS_{g}(k)$ we have $|\Lambda_x|\ge |\Lambda_{g,p^c}|$ for some $0\le c \le \lfloor g/2 \rfloor$. Therefore, $|\Lambda_x|>1$.     
\end{proof}

For any matrix $A=(a_{ij})\in \Mat_g(\F_{p^2})$, write $A^*=\ol A^ T=(a_{ji}^p)$, where $\ol A=(a_{ij}^p)$ and $T$ denotes the transpose. Let
\[  U_g(\Fp):=\{A\in \Mat_g(\F_{p^2}) : A\cdot A^*= {\bbI}_g \} \]
denote the unitary group of rank $g$ associated to the quadratic extension $\F_{p^2}/\Fp$. Let ${\rm Sym}_g(\F_{p^2})$ $\subseteq \Mat_g(\F_{p^2})$ be the subspace consisting of all symmetric matrices and let $\Sym_g(\F_{p^2})^0\subseteq \Sym_g(\F_{p^2})$ be the subspace consisting of matrices $B=(b_{ij})$ with  $b_{ii}=0$ for all $i$.  

\begin{definition}\label{def:EGH}
Let $\calE\subseteq \Mat_g(\F_{p^2})$ be a maximal subfield of degree $g$ over $\F_{p^2}$ stable under the involution $*$. Let
  \begin{align}
    \label{eq:group_dc1}
    G&:=\left \{
      \begin{pmatrix}
        \bbI_g & 0 \\ B & \bbI_g 
      \end{pmatrix}
      \begin{pmatrix}
        A & 0 \\ 0 & \ol A  
      \end{pmatrix}\in \GL_{2g}(\F_{p^2}): A\in U_g(\Fp),
      B\in {\rm Sym}_g(\F_{p^2})\, \right \};     \\
    \calE^1&:=\{ A\in \calE^\times: A^* A=\bbI_g\}= \calE^\times \cap U_g(\Fp); \\
    H&:=\left \{     \begin{pmatrix}
        \bbI_g & 0 \\ B & \bbI_g 
      \end{pmatrix}
      \begin{pmatrix}
        A & 0 \\ 0 & \ol A  
      \end{pmatrix}: A\in \calE^1, \quad B\in {\rm Sym}_g(\F_{p^2})^0 \right \}; \label{eq:group_dcH} \\
    \Gamma &:= \left \{     \begin{pmatrix}
        \bbI_g & 0 \\ B & \bbI_g 
      \end{pmatrix}
      \begin{pmatrix}
        A & 0 \\ 0 & \ol A  
      \end{pmatrix}: A\in \diag(\F_{p^2}^1, \dots, \F_{p^2}^1) \cdot S_g,\  B\in \diag(\F_{p^2}, \dots, \F_{p^2}) \, \right \},  
  \end{align}
where $\F_{p^2}^1\subseteq \F_{p^2}^\times$ denotes the subgroup of norm one elements and $S_g$ denotes the symmetric group of $\{1,\dots, g\}$.
\end{definition}

\begin{lemma}\label{lm:group_dcoset}
Using the notation introduced in Definition~\ref{def:EGH}, the following statements hold.
  \begin{enumerate}
  \item Up to isomorphism, the double coset space $(\diag(\F_{p^2}^1, \dots, \F_{p^2}^1) \cdot S_g) \backslash U_g(\Fp)/\calE^1$ is independent of the choice of $\calE$.
  \item For $p=2$, up to isomorphism, the double coset space $\Gamma \backslash G/H$ is independent of the choice of $\calE$.
  \end{enumerate}  
\end{lemma}

\begin{proof}
\begin{enumerate}
\item  We know that $\calE^1$ is a cyclic group of order $p^g+1$ and choose a generator $\eta$ of~$\calE^1$. One has $\eta ^* \eta =\bbI_g$ and $\calE=\F_{p^2}[\eta]$. Suppose that $\calE_1$ is another maximal subfield stable under~$*$. We will first show that $\calE_1$ is conjugate to $\calE$ under $U_g(\Fp)$. By the Noether-Skolem theorem, there is an element $\gamma\in \GL_g(\F_{p^2})$ such that $\calE_1=\gamma \calE \gamma^{-1}$. Clearly, $\calE_1=\F_{p^2}[\eta_1]$ is generated by $\eta_1:=\gamma \eta \gamma^{-1}$ and $\eta_1$ has order $p^g+1$. We also have $\eta_1^* \eta_1=\bbI_g$; this follows from the fact that the norm-one subgroup $\calE^1_1\subseteq \calE_1^\times$ is the unique subgroup of order $p^g+1$ and that $\eta_1$ has order $p^g+1$. It follows from $\eta_1^* \eta_1 =\bbI_g$ that $(\gamma^{-1})^* \eta ^* \gamma^* \gamma \eta \gamma^{-1}=\bbI_g$. Putting $\alpha=\gamma^* \gamma$, we find that
 \[ \eta^* \alpha \eta = \alpha\quad  \text{and} \quad \alpha \eta \alpha^{-1}=\eta. \]
 That is, $\alpha$ commutes with $\calE$, and $\alpha\in \calE^\times$ because $\calE$ is a maximal subfield. As $\alpha=\gamma^* \gamma$, $\alpha$ lies in the subfield $F\subseteq \calE$ fixed by the automorphism $*$ of order $2$. Since the norm map $N:\calE^\times \to F^\times, x \mapsto x^* x$ is surjective, we have $\alpha=\beta^* \beta$ for some $\beta\in \calE^\times$. Let $\gamma_1:=\gamma \beta^{-1}$. Then
 \[ \gamma_1^* \gamma_1 =(\gamma \beta^{-1})^* (\gamma \beta^{-1})=(\beta^{-1})^* \gamma^* \gamma \beta^{-1}=(\beta^{-1})^* \alpha \beta^{-1}=(\beta^{-1})^* \beta^* \beta \beta^{-1}=\bbI_g. \]
 Therefore, $\calE_1=\gamma_1 \calE \gamma_1^{-1}$ and $\gamma_1\in U_g(\Fp)$.
Right translation by $\gamma_1^{-1}$ gives an isomorphism 
$(\diag(\F_{p^2}^1, \dots, \F_{p^2}^1) \cdot S_g) \backslash U_g(\Fp)/\calE^1\simeq (\diag(\F_{p^2}^1, \dots, \F_{p^2}^1) \cdot S_g) \backslash U_g(\Fp)/\calE^1_1.$
This proves (1).

\item We may regard $U_g(\Fp)$ as a subgroup of $G$ via the map $A\mapsto
\begin{pmatrix}
  A & 0 \\
  0 & \ol A 
\end{pmatrix}$, and $\Sym_g(\F_{p^2})$ as a normal subgroup of $G$ via the map $B\mapsto \begin{pmatrix}
  \bbI_g & 0 \\
  B & \ol \bbI_g 
\end{pmatrix}$, so that $G$ is the semi-direct product $\Sym_g(\F_{p^2}) \rtimes U_g(\Fp)$;
conjugation by $G$ on $\Sym_g(\F_{p^2})$ gives an action of $U_g(\F_p)$ on $\Sym_g(\F_{p^2})$ via $A\cdot B= \ol A B \ol A^T$, where $A\in U_g(\F_{p})$ and $B\in \Sym_g(\F_{p^2})$. Suppose that $\calE_1$ is another maximal subfield of $\Mat_g(\F_{p^2})$ stable under $*$ and that
$H_1\subseteq G$ is the extension of $\Sym_g(\F_{p^2})^0$ by $\calE_1^1$ as in \eqref{eq:group_dcH}. As in part (1) of the lemma, it suffices to show that $H_1$ is conjugate to $H$ under~$G$. And to do so, it suffices to show they are conjugate under $U_g(\Fp)$, since this is a subgroup of $G$. In part~(1) we have shown that $\calE_1=\gamma_1 \calE \gamma_1^{-1}$ and $\gamma_1\in U_g(\Fp)$, so it follows that $\calE_1^1$ and $\calE^1$ are conjugate under $U_g(\Fp)$. Therefore, we are reduced to showing that ${\rm Sym}_g(\F_{p^2})^0$ is stable under the action of $U_g(\F_{p})$. Since $p=2$, one checks directly that the diagonal entries of the matrix $A (I_{ij}+I_{ji}) \ol A^T$ are all zero, where $I_{ij}$ is the matrix whose entries are $1$ at $(i,j)$ and zero elsewhere. Since the elements $I_{ij}+I_{ji}$ for $i\neq j$ generate ${\rm Sym}_g(\F_{p^2})^0$, we find that it is indeed stable under the action of $U_g(\F_{p})$.
This proves (2).
\end{enumerate}
\end{proof}

Let $\F_q$ be a finite field of characteristic $p$. Let $(V_0,\psi_0)$ be a non-degenerate symplectic space over $\F_q$ of dimension $2c$ and denote by $A\mapsto A^\dagger$ the symplectic involution on $\End(V_0)$ with respect to $\psi_0$. For any $k$-subspace $W$ of $V_0\otimes_{\Fq} k$, the endomorphism algebra of~$W$ over $\Fq$ is defined as
\begin{equation}\label{eq:endW}
   \End(V_0,W):=\{ A\in \End(V_0): A(W)\subseteq W \},
\end{equation} 
and the automorphism group of $W$ in the symplectic group $\Sp(V_0)$ is defined as
\begin{equation}
  \label{eq:SpW}
  \Sp(V_0,W):=\Sp(V_0)\cap \End(V_0,W).
\end{equation}
We denote by $\psi:V_0\otimes_{\Fq} k\times V_0\otimes_{\Fq} k \to k$ the extension of $\psi_0$ by $k$-linearity. 
Let $C_n$ denote a cyclic group of order $n$.

\begin{proposition}\label{prop:Cq^2+1}
  If $W$ is a non-zero isotropic $k$-subspace of $V_0\otimes_{\Fq} k$ such that $\Sp(V_0,W)\supseteq C_{q^{c}+1}$, then $\End(V_0,W)\simeq \Mat_{{2c}/d}(\F_{q^d})$ for some positive integer $d|{2c}$ such that $\ord_2(d)=\ord_2({2c})$. Moreover, if ${2c}$ is a power of $2$, then $\End(V_0,W)\simeq \F_{q^{2c}}$ and  
  $\Sp(V_0,W)=C_{q^{c}+1}$. 
\end{proposition}
\begin{proof}
Let $\eta$ be a generator of $C_{q^{c}+1}$ and let $\calE=\Fq[\eta]$ be the $\Fq$-subalgebra of $\End(V_0)$ generated by $\eta$. Since $\vert C_{q^{c}+1}\vert$ is prime to $q$, the group algebra $\Fq[C_{q^{c}+1}]$ is semi-simple and it maps onto $\calE$. On the other hand, the finite field $\F_{q^{2c}}$ is the smallest field extension of~$\Fq$ which contains an element of order $q^{c}+1$, so $\calE$ contains a copy $F$ of $\F_{q^{2c}}$ in $\End(V_0)$. Since $\dim V_0 = {2c} = [\F_{q^{2c}}:\F_q]$, we see that $F$ is a maximal subfield of $\End(V_0)$ and hence $\calE=F$.

Since $C_{q^{c}+1}\subseteq \Sp(V_0)$ and $\ord(\eta)=q^{c}+1$, we have that $\eta^\dagger=\eta^{-1}\in C_{q^{c}+1}$ and $\eta^{\dagger}\neq \eta$. So $\calE$ is stable under $\dagger$, and $\dagger$ is an automorphism of $\calE$ of order $2$. Moreover, $C_{q^{c}+1}$ is equal to the subgroup $\calE^1=\{a\in \calE^\times: N_{\calE/\calE_0}(a)=1\}$ of norm one elements in $\calE^\times$, where $\calE_0$ is the subfield of $\calE$ fixed by $\dagger$. 

Let $\Sigma_\calE:=\Hom_{\Fq}(\calE, \ol \F_{p})$ denote the set of embeddings of $\calE$ into $\Fpbar$; it is equipped with a left action by $\Gal(\F_{q^{2c}}/\Fq)=\Gal(\calE/\Fq)=\< \sigma \>\simeq \Z/{2c}\Z$, which acts simply transitively. Arrange $\Sigma_{\calE}=\{\sigma_i: i\in \Z/{2c}\Z\}$ in such a way that $\sigma\cdot \sigma_{i}=\sigma_{i+1}$ for all $i\in \Z/{2c}\Z$ and denote by $V^i$ the $\sigma_i$-isotypic eigenspace of $V_0\otimes k$. Then $V_0\otimes_{\Fq} k=\oplus_{i\in \Z/{2c}\Z} V^i$ is a decomposition into simple $(\calE\otimes_{\Fq} k)$-submodules. Since $W\subseteq V_0\otimes_{\Fq} k$ is an $(\calE\otimes_{\Fq}k)$-submodule, there is a unique and non-empty subset $J\subseteq \Z/{2c}\Z$ such that $W=\oplus_{i\in J} V^i$. Note that the involution
$\dagger$ acts on $\Sigma_{\calE}$ from the right and one has $\sigma_i^\dagger=\sigma_{i+c}$.

We claim that $J\cap J^\dagger=\emptyset$. For $i,j\in \Z/{2c}\Z$, one computes that 
\[ \sigma_i(a)\psi(v_1,v_2)=\psi(a\cdot v_1, v_2)=\psi(v_1, a^\dagger \cdot v_2)=\sigma_{j+c}(a) \psi(v_1,v_2) \] 
for any $a\in \calE$, $v_1\in V^i$ and $v_2\in V^j$. It follows that $\psi(V^i, V^j)=0$ if $i-j\neq c$ in $\Z/{2c}\Z$. Since $\psi$ is non-degenerate, the latter is also a necessary condition. Since $W$ is isotropic, $J$ does not contain $\{i,i+c\}$ for any $i$ and therefore $J\cap J^\dagger=\emptyset$, as claimed. 

We represent the matrix algebra $\End(V_0)$ over $\Fq$ as a cyclic algebra, cf.~\cite[Theorem 30.4]{reiner:mo}:
\[ \End(V_0)=\calE[z], \qquad z^{2c}=1, za z^{-1}=\sigma(a) \ \text{ for all } a\in \calE. \]
Multiplication by $z$ maps $V^i$ onto $V^{i-1}$:
\[ a\cdot zv= z(\sigma^{-1}(a)\cdot v)=z \sigma_{i-1}(a) v= \sigma_{i-1}(a) z v, \quad \forall \, a\in \calE, v\in V^i.\] 
Consider an element $x=\sum_{i\in \Z/{2c}\Z} a_i z^i\in \End(V_0,W)$; if $a_i\neq 0$, then $J$ is stable under the shift by $-i$. Let $d\ge 1$ be the smallest integer with $d|{2c}$ such that 
$J$ is stable under the shift by~$-d$. Then $\End(V_0,W)=\calE[z^d]\simeq \Mat_{{2c}/d}(\F_{q^d})$. Since $J\cap J^\dagger=\emptyset$, we have $d\nmid~c$. Therefore, $d$ is a positive divisor of ${2c}$ such that $\ord_2(d)=\ord_2({2c})$. This proves the first statement. 
When ${2c}$ is a power of $2$, the condition on $d$ implies $d={2c}$ and therefore $\End(V_0,W)=\calE$. This implies that $\Sp(V_0,W)=\calE^1=C_{q^{c}+1}$ and hence proves the second statement.
\end{proof}

\subsection{The case $\boldsymbol{g=3}$}\label{ssec:g3}\ 

\begin{lemma}\label{lem:p2coset}
We use the notation for $G, \Gamma, H$ defined in Definition~\ref{def:EGH}. For $g=3$ and $p=2$, we have $|\Gamma \backslash G/H|=2$. 
\end{lemma}

\begin{proof}
Put $U:={\rm Sym}_g(\F_{p^2})$, embedded into $\mathrm{GL}_{2g}(\F_{p^2})$ via $B \mapsto \left( \begin{smallmatrix}
        \bbI_g & 0 \\ B & \bbI_g 
      \end{smallmatrix}\right)$. Then $U_\Gamma:=U\cap \Gamma \simeq \diag(\F_{p^2}, \dots, \F_{p^2})$ and $U_H:=U\cap H \simeq {\rm \Sym}_g(\F_{p^2})^0$. 
      Consider the surjective map induced by the natural projection
\[ \pr: \Gamma\backslash G/H \to (\diag(\F_{p^2}^1, \dots, \F_{p^2}^1) \cdot S_g) \backslash U_g(\Fp)/\calE^1. \]
One shows directly that the fibre of the double coset $(\diag(\F_{p^2}^1, \dots, \F_{p^2}^1) \cdot S_g)\cdot A\cdot \calE^1$ for an element $A\in U_g(\Fp)$  is isomorphic to $U_\Gamma+ \ol A U_H \ol A^{T}$. Since $\ol A U_H \ol A^{T}=U_H$ for $p=2$ by Lemma~\ref{lm:group_dcoset}.(2), we have $U_\Gamma+ \ol A U_H \ol A^{T}=U_\Gamma+U_H=U$ and hence $\pr$ is an isomorphism. 

Now let $g=3$ and $p=2$; we need to show that the target of $\pr$ has two double cosets.
Put $\F_4=\F_2[\zeta]$ with $\zeta^2+\zeta+1=0$ and 
\begin{equation}\label{eq:etaA}
 \eta:=
\begin{pmatrix}
  0 & 0 & \zeta \\
  1 & 0 & 0 \\
  0 & 1 & 0 \\
\end{pmatrix}, \quad A: =  \begin{pmatrix}
    1 & \zeta & \zeta \\
    \zeta & 1 & \zeta \\
    \zeta & \zeta & 1 \\ 
  \end{pmatrix}. 
\end{equation}
We choose $\calE^1=\< \eta\>$ and verify directly that 
\[ U_3(\F_2)= \Big( \diag(\F_4^\times, \F_4^\times, \F_4^\times)S_3\cdot \bbI_3 \cdot \calE^1 \Big) \, \coprod\,  \Big( \diag(\F_4^\times, \F_4^\times, \F_4^\times)S_3\cdot A 
  \cdot \calE^1 \Big). \]  
This shows that $|\Gamma\backslash G/H| = 2$; recall from Lemma~\ref{lm:group_dcoset}.(2) that the double coset space is independent of the choices made. 
\end{proof}
Recall that $\mathcal{P}_{\mu}$ (resp.~$\calP'_\mu$) denotes the moduli
space over $\F_{p^2}$ of 
three-dimensional (resp.~rigid) polarised flag
type quotients with respect to a principal polarisation $\mu$ on $E^3$. The moduli space $\mathcal{P}_{\mu}$ is a $\bbP^1$-bundle $\pi: \mathcal{P}_{\mu}\to C$ over the Fermat curve $C\subseteq \bbP^2$. We express each point $y\in \calP_\mu(k)$ as $(t,u)$, where $t=\pi(y)$ and $u\in \pi^{-1}(t)=: \bbP^1_t(k)$; see Subsection 3.3.2 for more details. 

\begin{proposition} \label{lm:a1DCFp6}
  Let $p=2$, let $x=(X,\lambda)\in \calS_{3}(k)$ with $a(x)=1$ and let $y=(t,u)\in \calP'_\mu(k)$ be a point over $x$ for the unique element $\mu=\lambda_{\rm can}$ in $P(E^3)$.
Assume that $y\in \calD$ and $t\in C(\F_{p^6})$. Then $|\Lambda_x|=2$. Moreover, the two members $(X',\lambda')$ and $(X'',\lambda'')$ of $\Lambda_x$ have automorphism groups
\[ \Aut(X',\lambda')\simeq C_2^3\rtimes C_9, \quad \Aut(X'',\lambda'')\simeq C_2^3 \times C_3, \]
where $C_9$ acts on $C_2^3$ by a cyclic shift.
\end{proposition}

\begin{proof}
  Let $x_2=(X_2,\lambda_2)\to x=(X,\lambda)$ be the minimal isogeny for $(X,\lambda)$. 
  As $a(X)=1$ and the class number $H_3(2,1)=1$, we have $(X_2,\lambda_2)\simeq (E^3, p\lambda_{\rm can})$. Again using $H_3(2,1)=1$, one has $|G^1(\Q)\backslash G^1(\A_f)/U_{x_2}|=1$ so $G^1(\A_f)=G^1(\Q)U_{x_2}$; recall that $U_x = G_x(\wh \Z)$ for any $x$. Hence,
\[ \Lambda_x\simeq  G^1(\Q)\backslash G^1(\A_f)/U_x= G^1(\Q)\backslash 
  G^1(\Q)U_{x_2} /U_{x}=G^1(\Z)\backslash G_{x_2}(\Zp) /G_{x}(\Zp), \]
where $G^1(\Z)=G^1(\Q)\cap U_{x_2}=\Aut(X_2,\lambda_2)$, as $G_{x_2}(\Z_\ell)=G_x(\Z_\ell)$ for all primes $\ell\neq~p$.

Let $\ul M_2=(M_2,\<\, ,\>_2)$ and $\ul M=(M,\<\, ,\>)$ be the (contravariant) polarised \dieu modules associated to $(X_2,\lambda_2)$ and $(X,\lambda)$, respectively. We have $M\subseteq M_2$ with $\dim_k (M_2/M)=3$. Furthermore, $\<\, ,\>$ is a perfect pairing for $M$ and it is the restriction of $\<\, ,\>_2$ on $M$.
Regarding $G_{x_2}(\Zp)=\Aut_{\rm DM}(\ul M_2)$, we have a reduction-modulo-$p$ map:
\[ m_p: G_{x_2}(\Zp)=\Aut_{\rm DM}(\ul M_2) \to \Aut(M_2/pM_2); \]
we write $G_{\ul M_2}$ for its image. 
For $G_{x}(\Zp)=\Aut_{\rm DM}(\ul M)$, it then follows from the construction that
\[ G_x(\Zp)=\{h\in G_{x_2}(\Zp): m_p(h)(M/pM_2)=M/pM_2\}. \]
Therefore, $G_{x}(\Zp)$ contains the kernel $\ker( m_p) \subseteq G_{x_2}(\Zp)$ and we obtain
\begin{equation}\label{eq:p2}
\Lambda_x\simeq \Gamma \backslash G_{\ul M_2} /G_{\ul M},  
\end{equation}
where $G_{\ul M}$ is the image $m_p(G_{x}(\Zp))$ and $\Gamma:=m_p(G^1(\Z))$. It follows from \cite[Lemma 6.1]{karemaker-yobuko-yu} that reduction modulo $p$ gives an exact sequence
\[
  \begin{CD}
    1 @>>> C_2^3 @>>> \Aut(X_2, \lambda_2) @>{m_p}>> \Gamma @>>> 1.   
  \end{CD}
\]
Let $O=\End(E)$ be a maximal order of $\End^0(E)\simeq B_{p,\infty}$ and let $\Pi\in O$ be the Frobenius endomorphism. 
Clearly, $G_{\ul M_2} = m_p(\Aut_{\rm DM}(\ul M_2))$ is a subgroup of $\GL_3(O/pO)=\GL_3(\F_{p^2}[\Pi])$. In fact, the group $G_{\ul M_2}$ is isomorphic to the group $G$ of Definition~\ref{def:EGH}; cf.~\cite[Definition 5.3]{karemaker-yobuko-yu}. 
By further reduction modulo $\Pi$, we obtain an exact sequence
\[
  \begin{CD}
    1 @>>> U:=\Sym_3(\F_4) @>>> G_{\ul M_2} @>{m_\Pi}>> U_3(\F_2) @>>> 1. 
  \end{CD}
\]
Let $\calE$ be the image of $\End_{\rm DM}(\ul M)$ in $m_\Pi(\End_{\rm DM}(\ul M_2))$. Since $p=2$ and $t\in C(\F_{2^6})$, $\calE\simeq \F_{4^3}$ is a subalgebra of $\Mat_3(\F_4)$ of degree $3$ which is stable under the induced involution $*$, and $U \cap G_{\ul M}=\Sym_3(\F_4)^0$. Therefore, $G_{\ul M}$ is isomorphic to the group $H$ in Definition~\ref{def:EGH}. 
As
\[
G^1(\Z) = \Aut(X_2, \lambda_2) \simeq \Aut(E^3,\lambda_{\rm can}) \simeq (O^\times)^3\cdot S_3,
\]
we further see that $\Gamma$ is the same as in Definition~\ref{def:EGH}. So by Lemma~\ref{lem:p2coset}, for $x = (X, \lambda)$,
the set 
\[ \text{$\Lambda_x \simeq \Gamma \backslash G / H$ has two elements}, \] 
represented by
\begin{equation}\label{eq:Xlreps}
 (X',\lambda') \leftrightarrow G^1(\Z) \cdot \bbI_3 \cdot G_x(\mathbb{Z}_p)  \text{ and } (X'',\lambda'') \leftrightarrow G^1(\Z) \cdot \tilde{A} \cdot G_x(\mathbb{Z}_p),
\end{equation}
where $\tilde{A}$ is a lift of $A$ as in Equation~\eqref{eq:etaA}. That is, we may take
\[
\tilde{A}: =  \frac{1}{a} \begin{pmatrix}
    1 & \zeta & \zeta \\
    \zeta & 1 & \zeta \\
    \zeta & \zeta & 1 \\ 
  \end{pmatrix}, \text{ for } 1 \neq \zeta \in O^{\times} \text{ such that $\zeta^3 = 1$ and } a = 2+\zeta \in O. 
\] 
The coset representation in \eqref{eq:Xlreps} also immediately implies that 
\[
\Aut(X', \lambda') \simeq G^1(\Z) \cap G_x(\mathbb{Z}_p) \text{ and } \Aut(X'', \lambda'') \simeq G^1(\Z) \cap \tilde{A} G_x(\mathbb{Z}_p) \tilde{A}^{-1}.
\]
The group $G^1(\Z) \cap G_x(\mathbb{Z}_p)$ sits in the short exact sequence
\[ \begin{tikzcd}
 1 \arrow[r] & C_2^3 \arrow[r] & G^1(\Z) \cap G_x(\mathbb{Z}_p) \arrow[r,"m_\Pi"] & \calE^1 \arrow[r] & 1
\end{tikzcd}\] 
and one has $|\Aut(X',\lambda')|=8\cdot 9$.
From the mass $\Mass(\Lambda_x)=1/(2\cdot 3^2)$ (see Theorem~\ref{introthm:a1}) and the fact that $\vert \Lambda_x \vert = 2$, we immediately see that $|\Aut(X'',\lambda'')|=8\cdot 3$.

To determine the automorphism groups precisely, we argue as follows. We have that $x = (X, \lambda)$ either equals $(X',\lambda')$ or equals $(X'',\lambda'')$. In either case, the group $\Aut(X,\lambda)$ is the subgroup of $\Aut(X_2,\lambda_2)$ consisting of elements $h$ such that $m_p(h)\in~H$. Since $U_\Gamma\cap U_H$ is trivial, its image $m_p(\Aut(X,\lambda))$ is the same as its image $m_\Pi(\Aut(X,\lambda))\subseteq \calE^1\simeq C_9$.
Moreover, we know that $G^1(\Z) = (O^\times)^3\cdot S_3$ and that 
\[ 
G_x(\Z_p) = m_p^{-1}(H) = m_p^{-1}( \Sym_3(\F_4)^0 \calE^1) = m_{\Pi}^{-1} (\calE^1),
\]
where 
\[
C_2^3 \simeq \diag(\pm 1, \pm 1, \pm 1) = \ker(m_p)\cap (O^\times)^3\cdot S_3 \subseteq \ker(m_{\Pi})\cap (O^\times)^3\cdot S_3
\]
and $C_9 \simeq \calE^1 = \langle \eta \rangle \subseteq (O^\times)^3\cdot S_3$ by construction.
For $(X', \lambda')$ we therefore must have
\[
\Aut(X', \lambda') \simeq G^1(\Z) \cap G_x(\mathbb{Z}_p) = C_2^3 \rtimes C_9 
\]
of cardinality $8\cdot 9$, since the conjugation action by $\eta$ on $\diag(\pm 1, \pm 1, \pm 1)$ is non-trivial.

For $(X'',\lambda'')$, we note that $\tilde{A} \in G_{x_2}(\Z_p)$ normalises $\ker(m_{\Pi}) = m_p^{-1}( \Sym_3(\F_4)^0)$ by construction  and compute that
\[
\tilde{A} \eta \tilde{A}^{-1} = \frac{1}{2+\overline{\zeta}} \begin{pmatrix}
    1 & 1 & 1 \\
    1 & \zeta & \overline{\zeta} \\
    \zeta & 1 & \overline{\zeta} \\ 
  \end{pmatrix} =: B,
\]
where $\overline{\zeta} = \zeta^{-1}$. Hence, we get
\[
\Aut(X'', \lambda'') \simeq G^1(\Z) \cap \tilde{A} G_x(\mathbb{Z}_p) \tilde{A}^{-1} = \diag(\pm 1, \pm 1, \pm 1) \cdot \{ B^3, B^6, B^9 = \bbI_3 \} \simeq C_2^3 \times C_3
\]
of cardinality $8 \cdot 3$, since the conjugation action by $\{\bbI_3, B^3, B^6\}$ is trivial.
\end{proof}

\begin{proposition}\label{prop:autg3}
  Let $p=2$, choose $x=(X,\lambda)\in \calS_{3}(k)$ and let $y=(t,u)\in \calP'_\mu(k)$ be a point over $x$ for the unique element $\mu=\lambda_{\rm can}$ in $P(E^3)$.

\begin{enumerate}
\item Suppose that $t\in C(\F_{p^2})$, that is, $a(x)\ge 2$. Then $\vert \Lambda_x \vert=1$ and we have that
  \begin{equation}
    \label{eq:auta23}
    \vert \Aut(X,\lambda)\vert=
\begin{cases}
  24^3\cdot 6=2^{10}\cdot 3^4 & \text{if } u\in
  \mathbb{P}_t^1(\mathbb{F}_{p^2}); \\ 
  24\cdot 160=2^8\cdot 3\cdot 5 & \text{if }
  u\in\mathbb{P}_t^1(\mathbb{F}_{p^4})\setminus
  \mathbb{P}_t^1(\mathbb{F}_{p^2}); \\ 
  24\cdot 32=2^8\cdot 3 & \text{ if
  } u \not\in 
  \mathbb{P}_t^1(\mathbb{F}_{p^4}).
\end{cases}
\end{equation}
\item Suppose that $t\not\in C(\F_{p^2})$, that is, $a(x)=1$. Then
\begin{equation}
  \label{eq:cna1}
  \vert \Lambda_x \vert=
  \begin{cases}
4 & \text{ if } y \notin \calD; \\
4 & \text{ if } t \notin C(\mathbb{F}_{p^6}) \text{ and } y \in \calD; \\
2 & \text{ if } t \in C(\mathbb{F}_{p^6}) \text{ and } y \in \calD.
\end{cases}
\end{equation}
\end{enumerate}
\end{proposition}

\begin{proof}
\begin{enumerate}
\item If $u \in \bbP^1_t(\F_{p^2})$, then $a(x)=3$ and $\vert \Lambda_x \vert =H_3(2,1)=1$, and one computes that $\Mass(\Lambda_x)=1/(2^{10}\cdot 3^4)$. Therefore, $\vert \Aut(X,\lambda) \vert=24^3\cdot 6$. Alternatively, this also follows from \cite[Lemma 7.1]{karemaker-yobuko-yu}. Now we assume that $a(x)=2$. Using the mass formula (cf.~Theorem~\ref{introthm:a2}), we compute that
  \begin{equation}
    \label{eq:massa2}
    \Mass(\Lambda_x)=
\begin{cases}
  1/(2^8\cdot 3\cdot 5) & \text{if }
  u\in\mathbb{P}_t^1(\mathbb{F}_{p^4})\setminus
  \mathbb{P}_t^1(\mathbb{F}_{p^2}); \\ 
  1/(2^8\cdot 3) & \text{ if
  } u \not\in 
  \mathbb{P}_t^1(\mathbb{F}_{p^4}).
\end{cases}
\end{equation}  
Let $(E_k^3,p\mu)\xrightarrow{\rho_2} (Y_1,\lambda_1)\xrightarrow{\rho_1} (Y_0,\lambda_0)\simeq (X,\lambda)$ be the PFTQ corresponding to the point $y=(t,u)$. Since $t\in C(\F_{p^2})$, $Y_1$ is superspecial and $(Y_1,\lambda_1)\simeq (E_k,\lambda_E)\times (E_k^2, \mu_1)$, where $\lambda_E$ is the canonical principal polarisation on $E$ and $\mu_1\in P_1(E^2)$. Since $p=2$, we have
$\vert \Aut(E,\lambda_E)\vert=\vert \Aut(E)\vert=24$ and $\vert \Aut(E^2,\mu_1)\vert=1920$, cf.~\cite{ibukiyama:autgp1989}. By Corollary~\ref{cor:Autsp} and Equation~\eqref{eq:AutXl},
we have 
\[ |\Aut\left ( (E,\lambda_E)\times (E^2, \mu_1)\right )|=|\Aut(E,\lambda_E)|\times |\Aut(E^2,\mu_1)|=24\cdot 1920. \] Notice that $\ker(\rho_1)$ is contained in $\ker(\mu_1)$ since $\ker(\lambda_E)$ is trivial. Therefore, $(X,\lambda)$ is isomorphic to $(E,\lambda_E)\times (X', \lambda')$, where $X'=E_k^2/\ker(\rho_1)$. The computation of $\Aut (X,\lambda)$ is now reduced to computing $\Aut(X',\lambda')$. By Corollary~\ref{cor:p2g2aut}, we have
  \[ 
  \vert \Aut(X',\lambda') \vert=
    \begin{cases}
      160 & \text{if $u\in \bbP^1(\F_{p^4})\setminus \bbP^1(\F_{p^2})$};\\
      32 & \text{if $u\in \bbP^1(k)\setminus \bbP^1(\F_{p^4})$}.
    \end{cases}
  \]
  Therefore,
  \[
  \vert \Aut(X,\lambda) \vert=
    \begin{cases}
      24\cdot 160 = 2^8 \cdot 3 \cdot 5 & \text{if $u\in \bbP^1(\F_{p^4})\setminus \bbP^1(\F_{p^2})$};\\
      24\cdot 32 = 2^8 \cdot 3 & \text{if $u\in \bbP^1(k)\setminus \bbP^1(\F_{p^4})$}.
    \end{cases}
  \]  
Comparing this result with the values of $\Mass(\Lambda_x)$ in \eqref{eq:massa2}, we conclude that $\vert \Lambda_x \vert=1$ in both cases. 

\item If $y\notin \calD$, by \cite[Corollary 7.5.(1)]{karemaker-yobuko-yu} we have that $\vert \Lambda_x \vert=4$.
Suppose then that $y\in \calD$ and $t\notin C(\F_{p^6})$.
For every point $x'$ in $\Lambda_x$, consider
the corresponding polarised abelian variety $(X',\lambda')$ satisfying
$(X',\lambda')[p^\infty]\simeq (X,\lambda)[p^\infty]$. If $y'=(t',u')\in \calP_\mu'(k)$ is a point over~$x'$, then again $y'\in \calD$ and  $t'\notin C(\F_{p^6})$.
Thus, by \cite[Theorem 7.9.(1)]{karemaker-yobuko-yu}, we have that $\Aut(X',\lambda')\simeq C_2^3 \times C_3$. Using the mass formula (cf.~Theorem~\ref{introthm:a1}), noting that $d(t)=3$ when $p=2$, we compute that
\[ 
\Mass(\Lambda_x)=\frac{1}{6}. 
\]
Therefore, $\vert \Lambda_x \vert =\vert C_2^3\times C_3\vert \cdot \Mass(\Lambda_x)=4$.

For the last case, where $y\in \calD$ and $t\in C(\F_{p^6})$, the assertion $\vert \Lambda_x \vert=2$ follows directly from Proposition~\ref{lm:a1DCFp6}.
\end{enumerate}
\end{proof}

\subsection{The case $\boldsymbol{g=4}$}\label{ssec:g4}\ 

\begin{definition}
\begin{enumerate}
    \item An \emph{elementary sequence}  is a map $\varphi: \{ 1, \ldots, g \} \to \mathbb{Z}_{\geq 0}$ such that $\varphi(0) =~0$ and $\varphi(i) \leq \varphi(i+1) \leq \varphi(i) + 1$ for all $0 \leq i < g$, cf.~\cite[Definition 5.6]{OortEO}. With each elementary sequence we associate an \emph{Ekedahl-Oort stratum}~$\mathcal{S}_{\varphi}$, which is a locally closed subset of the moduli space $\mathcal{A}_{g,1,n} \otimes \overline{\mathbb{F}}_p$ of principally polarised abelian varieties with level-$n$ structure. Roughly speaking, it consists of those varieties whose $p$-torsion has a canonical filtration described by $\varphi$. On $\mathcal{S}_g$ we consider the stratification induced by $\mathcal{S}_{\varphi} \cap \mathcal{S}_g$. 
    \item The $p$-divisible group of an abelian variety of dimension~$g$ is determined up to isogeny by its Newton polygon, which can be described as a set of slopes $(\lambda_1, \ldots, \lambda_{2g})$ with $0 \leq \lambda_i \leq 1$ for all $1 \leq i \leq 2g$ and $\sum_i \lambda_i = g$, cf.~\cite{manin}. These slopes moreover satisfy that $\lambda_i + \lambda_{2g+1-i} = 1$ for all  $1 \leq i \leq 2g$ and that the denominator of each $\lambda_i$ divides its multiplicity. All abelian varieties with the same Newton polygon form a \emph{Newton stratum} of $\mathcal{A}_g$.
    \item For $1 \leq a \leq g$, we will denote the \emph{$a$-number locus} of $\mathcal{S}_g$ by $\mathcal{S}_g(a) := \{ x \in \mathcal{S}_g(k) : a(x) = a \}$.
\end{enumerate}
\end{definition}

\begin{proposition}\label{prop:EO}
\begin{enumerate}
    \item The Ekedahl-Oort strata in dimension $g=4$ of $p$-rank zero are precisely the $\calS_{\varphi}$ for those $\varphi$ appearing in Figure~1.
    \item The stratum $\calS_{\varphi}$ for $\varphi = (0,1,2,3)$ has $a$-number $1$, those for $\varphi = (0,1,2,2)$, $(0,1,1,2)$, and $(0,0,1,2)$ have $a$-number $2$, those for $\varphi = (0,1,1,1)$, $(0,0,1,1)$, and $(0,0,0,1)$ have $a$-number $3$ and that for $\varphi = (0,0,0,0)$ has $a$-number $4$.
    \item The strata fully contained in the supersingular locus $\mathcal{S}_4$ are precisely the $\calS_{\varphi}$ for\\ $\varphi~=~(0,0,0,0), (0,0,0,1), (0,0,1,1)$, and $(0,0,1,2)$.
    \item The Newton strata of $p$-rank zero are those corresponding to the slope sequences 
    \[
    \left(\frac{1}{2},\frac{1}{2},\frac{1}{2},\frac{1}{2},\frac{1}{2},\frac{1}{2},\frac{1}{2},\frac{1}{2}\right), \left(\frac{1}{3},\frac{1}{3},\frac{1}{3},\frac{1}{2},\frac{1}{2},\frac{2}{3},\frac{2}{3},\frac{2}{3}\right), \text{ and } \left(\frac{1}{4},\frac{1}{4},\frac{1}{4},\frac{1}{4},\frac{3}{4},\frac{3}{4},\frac{3}{4},\frac{3}{4}\right),
    \]
    which we denote respectively by $\mathcal{N}_{\frac{1}{2}}$, $\mathcal{N}_{\frac{1}{3}}$, and $\mathcal{N}_{\frac{1}{4}}$.
    \item We have 
    \[
    \begin{split}
   \mathcal{S}_4 = \mathcal{N}_{\frac{1}{2}} & = \left( \mathcal{S}_{(0,1,2,3)} \cap \mathcal{S}_4 \right) \sqcup \mathcal{S}_{(0,0,0,0)}
    \sqcup \mathcal{S}_{(0,0,0,1)} \\
    & \sqcup \mathcal{S}_{(0,0,1,1)} \sqcup \mathcal{S}_{(0,0,1,2)} \sqcup \left( \mathcal{S}_{(0,1,1,2)} \cap \mathcal{S}_4 \right),
    \end{split}
    \]
and $\mathcal{S}_{(0,1,2,3)} \cap \mathcal{S}_4$ is dense. In particular, we have 
\[
\begin{split}
  \mathcal{S}_4(4) & = \mathcal{S}_{(0,0,0,0)}, \\
  \mathcal{S}_4(3) & = \mathcal{S}_{(0,0,0,1)} \sqcup \mathcal{S}_{(0,0,1,1)}, \\
  \mathcal{S}_4(2) & = \mathcal{S}_{(0,0,1,2)} \sqcup \left( \mathcal{S}_{(0,1,1,2)} \cap \mathcal{S}_4 \right).  
\end{split}
\]
    \item We have 
    \[
    \mathcal{N}_{\frac{1}{3}} = \left( \mathcal{S}_{(0,1,2,3)} \cap \mathcal{N}_{\frac{1}{3}} \right) \sqcup \mathcal{S}_{(0,1,1,1)}
    \sqcup \left( \mathcal{S}_{(0,1,1,2)} \cap \mathcal{N}_{\frac{1}{3}} \right),\] 
    and $\mathcal{S}_{(0,1,2,3)} \cap \mathcal{N}_{\frac{1}{3}}$ is dense.
    \item We have 
    \[
    \mathcal{N}_{\frac{1}{4}} = \left( \mathcal{S}_{(0,1,2,3)} \cap \mathcal{N}_{\frac{1}{4}} \right) \sqcup \mathcal{S}_{(0,1,2,2)},
    \]
    and $\mathcal{S}_{(0,1,2,3)} \cap \mathcal{N}_{\frac{1}{4}}$ is dense.
\end{enumerate}
All intersections appearing in (5)--(7) are non-empty.
\end{proposition}

\begin{proof}
 The $p$-rank of an Ekedahl-Oort stratum $\mathcal{S}_{\varphi}$ is $\max\{ i : \varphi(i) = i \}$ and its $a$-number is $g - \varphi(g)$, proving (1) and (2). By \cite[Step 2, p.\ 1379]{COirr} we have $\mathcal{S}_{\varphi} \subseteq \mathcal{S}_4$ if and only if $\varphi(2) = 0$, proving (3).
 The $p$-rank of a Newton stratum is the number of zero slopes, which implies~(4). 
 
We read off from Figure~1 that $\mathcal{S}_{(0,1,2,3)} \cap \mathcal{S}_4$, $\mathcal{S}_{(0,1,2,3)} \cap \mathcal{N}_{\frac{1}{3}}$, and $\mathcal{S}_{(0,1,2,3)} \cap \mathcal{N}_{\frac{1}{4}}$ are the respective $a$-number $1$ loci of $\mathcal{S}_4$, $\mathcal{N}_{\frac{1}{3}}$, and $\mathcal{N}_{\frac{1}{4}}$. Hence, density of these intersections follows from \cite[Theorem 4.9(iii)]{lioort} for $\mathcal{S}_4$, and from combining \cite[Remark 5.4]{OortNP} with \cite[Theorem~3.1]{COirr} for $\mathcal{N}_{\frac{1}{3}}$, and $\mathcal{N}_{\frac{1}{4}}$.

By \cite[Corollary 4.2 and Lemma 5.12]{harafirst} we see that $\mathcal{S}_{(0,1,2,2)} \subseteq \mathcal{N}_{\frac{1}{4}}$ by minimality of the associated $p$-divisble group, concluding the proof of~(7). 
Similarly, from \cite[Corollary 4.2 and Proposition 7.1]{harafirst}, we obtain that $\mathcal{S}_{(0,1,1,1)} \subseteq \mathcal{N}_{\frac{1}{3}}$, again by minimality. 
Finally, we read off from Figure~1 that 
\[
\mathcal{S}_{(0,1,1,2)} = \left( \mathcal{S}_{(0,1,1,2)} \cap \mathcal{N}_{\frac{1}{3}} \right) \sqcup \left( \mathcal{S}_{(0,1,1,2)} \cap \mathcal{S}_{4} \right).
\]
Now \cite[Theorem 4.17]{haraanumber} implies that $\mathcal{S}_4(2)$ has $H_4(1,p)+H_4(p,1)$ many irreducible components of two types, of which those of the type corresponding to $\mathcal{S}_{(0,0,1,2)}$ yield $H_4(1,p)$ many; see also~\cite[\S 9.9]{lioort}. Hence, the intersection $\mathcal{S}_{(0,1,1,2)} \cap \mathcal{S}_{4}$ must yield the other $H_4(p,1)$ components and thus be non-empty. On the other hand, since $\mathcal{S}_{(0,1,1,2)} \not\subseteq \mathcal{S}_4$ by (3), the intersection $\mathcal{S}_{(0,1,1,2)} \cap \mathcal{N}_{\frac{1}{3}}$ is also non-empty. This finishes the proof of (5) and (6), and hence, the proof of the proposition.
\end{proof}

\begin{figure}[H]\label{fig:EO}
\begin{center}
\begin{tikzcd}
 & & & |[blue]| (0,1,2,3) \arrow[dl, dash] \\
 & & |[orange]| (0, 1, 2, 2) \arrow[dl, dash] & \\
 & |[orange]| (0, 1, 1, 2) \arrow[dl, dash] \arrow[dr, dash] & & \\
|[purple]| (0,1,1,1) \arrow[dr, dash] & & |[orange]| (0,0,1,2) \arrow[dl, dash] & \\
& |[purple]| (0,0,1,1) \arrow[d, dash] & & \\
& |[purple]| (0,0,0,1) \arrow[d, dash] & & \\
& |[teal]| (0,0,0,0) & & \\
\end{tikzcd}
\end{center}
\caption{Ekedahl-Oort strata of $p$-rank zero in dimension $g=4$. The blue stratum has $a$-number $1$, the orange strata have $a$-number $2$, the red strata have $a$-number $3$ and the green stratum has $a$-number $4$. Strata are connected by a line if the lower one is contained in the Zariski closure of the upper one.}
\end{figure}
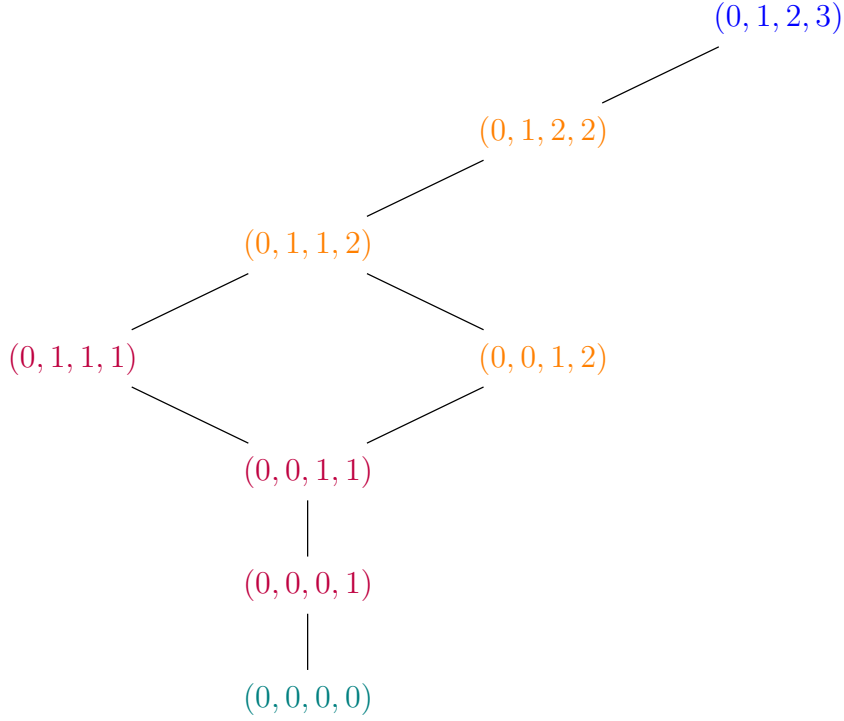

By Lemma~\ref{lem:Lgpc}, for every point $x\in \mathcal{S}_4(k)$, there exists an integer $0\leq c\leq 2$ such that there exists a surjective morphism $\pi:\Lambda_x \twoheadrightarrow \Lambda_{4,p^c}$. For Ekedahl-Oort strata with $g=4$ we have the following result:

\begin{lemma}\label{lem:c}
We have $c = 0$ for $x \in \mathcal{S}_{(0,0,0,0)}$, and $c =1$ for $x \in \mathcal{S}_{(0,0,0,1)}$, and $c = 2$ for $x \in \mathcal{S}_{(0,0,1,1)} \cup \mathcal{S}_{(0,0,1,2)}$.
\end{lemma}

\begin{proof}
This follows from \cite[Proposition 3.3.2]{harashita}; the Deligne-Lusztig varieties $X(w')$ in \emph{loc. cit.} are given by $w' = \mathrm{id}$ when $c=0$, by $w' = (12)$ when $c=1$ and by $w' = (1342)$ or $(13)(24)$ when $c=2$.
\end{proof}

\begin{remark} One might ask whether the surjection $\Lambda_x \twoheadrightarrow \Lambda_{g,p^c}$ is realised through the minimal isogeny $\tilde{x} \twoheadrightarrow x$ for $x$, i.e.,
whether there is a natural isomorphism $\Lambda_{g,p^c}\isoto \Lambda_{\tilde{x}}$, sending $[X',\lambda']\mapsto [X', p^r \lambda']$, for appropriate values of $c$ and $r$, or equivalently, whether the open compact subgroup $U_{\wt x}$ is maximal.
This is true if $g\le 3$. However, it is false in general, and we now give a counterexample.

Take $g=4$ and $x=(X,\lambda)=(E, \lambda_E)\times (X_1,\lambda_1)$, where $E$ is a supersingular elliptic curve with canonical principal polarisation $\lambda_E$, and where $(X_1,\lambda_1)$ is a principally polarised supersingular abelian threefold with $a(X_1)=1$. The minimal isogeny $\wt x=(\wt X,\wt \lambda)$ of $(X,\lambda)$ is equal to $(E, \lambda_E)\times (\wt X_1,\wt \lambda_1)$, where $(\wt X_1,\wt \lambda_1)$ is the minimal isogeny of $(X_1, \lambda_1)$. 
Note that $U_{\wt x,p}=\Aut((\wt X,\wt \lambda)[p^\infty])$ is a maximal open compact subgroup if and only if $\wt X[\sfF^s]\subseteq \ker \wt \lambda\subseteq \wt X[\sfF^{s+1}]$ for some $s\in \Z_{\ge 0}$. To see this: if $(L,h)$ is the Hermitian $O$-lattice corresponding to $(\wt X,\wt \lambda)$, where $O\subseteq B_{p,\infty}$ is a maximal order, then the latter condition is equivalent to $\Pi^{s+1}_p L_p \subseteq L_p^\vee \subseteq \Pi^s_p L_p$ for some integer $s\in \Z_{\ge 0}$, where $\Pi_p$ is a uniformiser of $O_p$.  The stabilisers of these 
$O_p$-lattices then give all maximal open compact subgroups of $G^1(\Q_p)$, as described in the proof of Lemma~\ref{lem:Lgpc}.
From \cite[Theorem~3.3]{oda-oort}, cf.~\cite[Proposition~3.16.(1)]{karemaker-yobuko-yu}, we have $\ker \wt \lambda_1=\wt X_1[\sfF^2]=E^3[\sfF^2]$. We see that $\ker \wt \lambda \subseteq \wt X[\sfF^2]$ but $\wt X[\sfF]=E^4[\sfF]\not \subseteq \ker \wt \lambda=E^3[\sfF^2]$, so $U_{\wt x}$ is not a maximal open compact subgroup. 
Furthermore, up to conjugacy, the group $U_{\wt x}$ is properly contained in $U_{4,1}$, so there is a surjective map $\Lambda_{\wt x}\onto \Lambda_{4,1}$.

Note that the above point $x$ lies in the stratum $\mathcal{S}_{(0,1,1,2)} \cap \mathcal{S}_4$. 
One can show that if $x\in \mathcal{S}_4\setminus \mathcal{S}_{(0,1,1,2)}$, then the surjection $\Lambda_x \twoheadrightarrow \Lambda_{4,p^c}$ is realised by the minimal isogeny. Indeed, if $x$ is contained in a supersingular Ekedahl-Oort stratum (i.e., one of the strata in Proposition~\ref{prop:EO}.(3)) this follows directly, cf.~\cite{harashita}. Otherwise, we have $x \in \mathcal{S}_{(0,1,2,3)} \cap \mathcal{S}_4$. In this case, we have $a(x) = 1$, so $\Lambda_{\tilde{x}} \simeq \Lambda_{4,p^2}$, as we will see in the proof of Theorem~\ref{thm:maing4}.
\end{remark}

\begin{lemma}\label{lem:a4}
Let $x \in \mathcal{S}_4(k)$. When $a(x) = 4$, we have $\vert \mathcal{C}(x) \vert > 1$.
\end{lemma}

\begin{proof}
By Proposition~\ref{prop:EO}.(5), we have $x \in \mathcal{S}_{(0,0,0,0)}$, so by Lemma~\ref{lem:c}, there exists a surjection $\Lambda_x \twoheadrightarrow \Lambda_{4,p^0}$. As observed in Subsection~\ref{ssec:sspmass}, it holds that $\lvert \Lambda_{4,p^0} \rvert= H_4(p,1)$, so it follows from Theorem~\ref{thm:mainarith}.(4) that $H_4(p, 1) > 1$. This implies the result.
\end{proof}

\begin{lemma}\label{lem:a2S0001}
Let $x \in \mathcal{S}_4(k)$. When $a(x) = 3$ and $x \in \mathcal{S}_{(0,0,0,1)}$, we have $\vert \mathcal{C}(x) \vert > 1$.
\end{lemma}

\begin{proof}
By Lemma~\ref{lem:c}, there exists a surjection $\Lambda_x \twoheadrightarrow \Lambda_{4,p}$. By Theorem~\ref{thm:sspmass} we get that
\[
\mathrm{Mass}(\Lambda_{4,p}) = \frac{(p-1)(p^2+1)(p^4+1)(p^6-1)^2}{2^{15}\cdot 3^5 \cdot 5^2 \cdot 7}.
\]
Since $(p^3-1)$ divides the numerator of $\mathrm{Mass}(\Lambda_{4,p})$, we may argue as in the proof of Theorem~\ref{thm:mainarith}.(3)+(4) to conclude that this numerator is always larger than $1$. This implies that $\vert \Lambda_{4,p} \vert > 1$, so the result follows.
\end{proof}

\begin{lemma}\label{lem:a2Sn0012}
Let $x \in \mathcal{S}_4(k)$. When $a(x) = 2$ and $x \not\in \mathcal{S}_{(0,0,1,2)}$, we have $\vert \mathcal{C}(x) \vert > 1$.
\end{lemma}

\begin{proof}
By Proposition~\ref{prop:EO}.(5), we know that $x \in \mathcal{S}_{(0,1,1,2)} \cap \mathcal{S}_4$. 
By \cite[Main results, p.~164]{haraanumber} every generic point of $\mathcal{S}_{(0,1,1,2)} \cap \mathcal{S}_4$ has a minimal isogeny from $(E^4,p\mu)$ with a principal polarisation $\mu$. It follows that the minimal isogeny $\wt x=(\wt X ,\wt \lambda)$ for $x=(X,\lambda)$ is either isomorphic to $(E^4, p\mu)$ or there exists an isogeny $(E^4, p\mu)\to (\wt X, \wt \lambda)$ for some principal polarisation~$\mu$ on $E^4$. We also have that $\ker \wt \lambda$ is not equal to $\wt X[\sfF]=E^4[\sfF]$, otherwise $x \in \mathcal{S}_{(0,0,1,2)}$ as $a(X)=2$. Therefore, we find a surjection, either $\Lambda_{\wt x}\onto \Lambda_{4,1}$ or $\Lambda_{\wt x}\onto \Lambda_{4,p}$. The two numbers $|\Lambda_{4,1}|$ and $|\Lambda_{4,p}|$ are both greater than one as shown in Theorem~\ref{thm:mainarith}.(4) and in Lemma~\ref{lem:a2S0001}. Thus, $\lvert \calC(x) \rvert \ge \lvert \Lambda_{\wt x} \rvert >1$.
\end{proof}

\begin{theorem}\label{thm:maing4}
For every $x \in \mathcal{S}_4(k)$, we have $\vert \mathcal{C}(x) \vert > 1$.
\end{theorem}

\begin{proof}
It follows from Proposition~\ref{prop:EO}.(5) and Lemmas~\ref{lem:a4}--\ref{lem:a2Sn0012} that it suffices to consider $x \in \mathcal{S}_4(k)$ such that one of the following holds:
\begin{itemize}
    \item[(i)] $x \in \mathcal{S}_{(0,0,1,2)} \sqcup \mathcal{S}_{(0,0,1,1)}$, or
    \item[(ii)] $a(x)=1$.
\end{itemize}

In Case (i), by Lemma~\ref{lem:c}, there exists a surjection $\Lambda_x \twoheadrightarrow \Lambda_{4,p^2}$, i.e., $c=2$. In Case (ii), by \cite[Theorem 2.2]{oda-oort} (also see \cite[Lemma 4.4]{katsuraoort}), 
there exists a unique four-dimensional rigid PFTQ 
\[ (Y_\bullet, \rho_\bullet): (Y_3, \lambda_3) \to (Y_2,\lambda_2)  \to (Y_1,\lambda_1)\to (X_0,\lambda_0)=x \]
extending $(X_0,\lambda_0)$. The construction in \emph{loc.~cit.} also shows that the composition 
\[
y_3=(Y_3,\lambda_3)\to (X_0,\lambda_0) = x
\]
is the minimal isogeny for $x$, and hence so is $y_3=(Y_3,\lambda_3)\to y_2=(Y_2,\lambda_2)$ for $y_2$. By Definition~\ref{def:PFTQ}, the polarisation $\lambda_3$ is $p$ times a polarisation $\mu$ on $E^4$.
Dividing the 
polarisation by $p$ therefore gives an isomorphism $\Lambda_{y_3}\simeq \Lambda_{4,p^2}$.
Thus, the minimal isogeny gives rise to surjective maps $\Lambda_x\twoheadrightarrow \Lambda_{y_2} \twoheadrightarrow \Lambda_{4,p^2}$. Hence, to show that $|\calC(x)|>1$, it suffices to show that $|\Lambda_{y_2}|>1$. Replacing $x$ with $y_2$, we now also have a surjection $\Lambda_x \twoheadrightarrow \Lambda_{4,p^2}$ in Case (ii).

Since we have $L_{4,p^c} = L_4(1,p)$ from Equation~\eqref{eq:npgc}, it follows immediately from Theorem~\ref{thm:mainarith}.(4) that $\vert \mathcal{C}(x) \vert > 1$ when $p>2$. So from now on, we assume that $p=2$. 

We use the same notation as in Subsection~\ref{ssec:4first}. Since $\Lambda_{4,4} \simeq G^1(\Q)\backslash G^1(\A_f)/U_{x_2}$ where the base point $x_2 \in \Lambda_{4,4}$ is taken from the minimal isogeny for $x$, and $\vert \Lambda_{4,4} \vert = 1$, 
we get that $G^1(\A_f) = G^1(\mathbb{Q}) U_{x_2}$. Hence, 
\[
\Lambda_x \simeq G^1(\Q) \backslash G^1(\mathbb{Q}) U_{x_2} / U_x \simeq G^1(\Z) \backslash G_{x_2}(\Z_p) / G_x(\Z_p),
\]
where $G_x(\Z_p)$ is the automorphism group of the polarised Dieudonn{\'e} module associated to $x$. Applying the reduction-modulo-$\Pi$ map $m_{\Pi}$ , we obtain $m_{\Pi}(G_{x_2}(\Z_p)) = \mathrm{Sp}_4(\F_4)$. 

Further, let $(X_2,\lambda_2)$ be the superspecial abelian variety corresponding to the unique element $x_2 \in \Lambda_{4,4}$. Then by Proposition~\ref{prop:np2}, and using the same notation, we know that $G^1(\Z) = \Aut(X_2, \lambda_2) \simeq \Aut((L,h)^{\oplus 2})\simeq \Aut(L,h)^2 \cdot C_2$ and $\Aut(L,h)$ is the group of cardinality $1920$ described in \cite[Section 5]{ibukiyama}. By \cite[Section 5, p.~1178]{ibukiyama} the reduction modulo $\Pi$ induces a surjective homomorphism $\phi_0:\Aut(L,h)\onto \SL_2(\F_4)$ whose kernel $\ker(\phi_0)$ has order 32 (also see \emph{loc. cit.} for the description of $\ker(\phi_0)$). Then it follows that $m_{\Pi}(G^1(\Z))  = m_{\Pi}(\Aut(X_2,\lambda_2)) \simeq \mathrm{SL}_2(\F_4)^2 \cdot C_2$. 

Writing $\overline{G} := m_\Pi(G_x(\Z_p))$, we obtain
\begin{equation}\label{eq:Lambdag4}
\Lambda_x \simeq (\mathrm{SL}_2(\F_4)^2 \cdot C_2)\backslash \mathrm{Sp}_4(\F_4) / \overline{G},
\end{equation}
since $\ker(m_\Pi) \subseteq G_x(\Z_p)$ (cf.~the proof of Proposition~\ref{lm:a1DCFp6}). Thus, 
\begin{equation}\label{eq:Massg4p2}
\mathrm{Mass}(\Lambda_x) = \mathrm{Mass}(\Lambda_{4,4}) \cdot [\mathrm{Sp}_4(\F_4) : \overline{G}].
\end{equation}
We compute that
\begin{equation}\label{eq:massL44}
\mathrm{Mass}(\Lambda_{4,4}) = \frac{1}{2^{15}\cdot3^2\cdot5^2}
\end{equation}
from Theorem~\ref{thm:sspmass}, using Equation~\eqref{eq:valuevn}. Standard computations also show that
\begin{equation}\label{eq:Sp4F4}
\vert \mathrm{Sp}_4(\F_4) \vert = 2^8 \cdot 3^2 \cdot 5^2 \cdot 17
\end{equation}
and that
\begin{equation}\label{eq:SL2S2}
\vert \mathrm{SL}_2(\F_4)^2 \cdot C_2 \vert = 2^5 \cdot 3^2 \cdot 5^2.
\end{equation}
By~\eqref{eq:massL44} and~\eqref{eq:Sp4F4}, Equation~\eqref{eq:Massg4p2} reduces to
\begin{equation}\label{eq:mass17}
 \mathrm{Mass}(\Lambda_{x}) = \frac{17}{2^7 \cdot \vert \overline{G} \vert}.   
\end{equation}
We deduce that $\vert \Lambda_x \vert > 1$ whenever $17 \nmid \vert \overline{G} \vert$. 
Suppose therefore that $17 \mid \vert \overline{G} \vert$, so that $\overline{G}$ contains a cyclic group $C_{17}$ of order $17$. We claim that then $\overline{G} = C_{17}$. This finishes the proof, since if $\vert \Lambda_x \vert = 1$, Equation~\eqref{eq:Lambdag4} would imply that $\mathrm{Sp}_4(\F_4) = (\mathrm{SL}_2(\F_4)^2 \cdot C_2)$ $\overline{G} = (\mathrm{SL}_2(\F_4)^2 \cdot C_2)$ $C_{17}$. Comparing the cardinalities from~\eqref{eq:Sp4F4} and~\eqref{eq:SL2S2} would then yield a contradiction.

Finally, we prove the claim. Let $(M_2, \<\, , \, \>_2)$
be the (contravariant) polarised \dieu module attached to~$x_2$. From $\ker \lambda_2\simeq \alpha_p^4$ and the fact that $M_2$ is superspecial, we have that $\sfF M_2=\sfV M_2=M_2^\vee\subseteq M_2$, where $M_2^\vee$ is the dual lattice of $M_2$, and that $\<\,, \>_2$ has polarisation type $(1/p,1/p,1,1)$ on $M_2$. One can see that $M_2^\vee/pM_2$ is the null-subspace for the symplectic pairing $\psi$ on $M_2/pM_2$ induced by $p\<\, , \, \>_2$. 
It follows that $p\<\, , \, \>_2$ induces a non-degenerate symplectic form $\psi$ on $V:=M_2/\sfV M_2$.
The four-dimensional symplectic space $(V, \psi)$ over~$k$ admits an $\F_{4}$-structure~$V_0$ induced by the skeleton of $M_2$. Inside $V$ we have an isotropic $k$-subspace $W=M/\sfV M_2$, where $M\subseteq M_2$ is the \dieu module associated to $x$ and the inclusion is induced from the minimal isogeny. Note that $\dim W=2$ in Case (i) and $\dim W=1$ in Case (ii), respectively. According to our definition,
\[ \ol G:=\{A\in \Sp_4(\F_{4}): A(W)=W\, \}=\Sp(V_0,W). \]
Thus, it follows from Proposition~\ref{prop:Cq^2+1} for $g=4$ 
that $\ol G=C_{17}$. This completes the proof of the claim and hence of the theorem.
\end{proof}

\subsection{Proof of the main result} 
\begin{theorem}\label{thm:main2}
\begin{enumerate}
\item The supersingular locus $\calS_g$ is geometrically irreducible if and only if 
one of the following three cases holds:
\begin{itemize}
\item [(i)] $g=1$ and $p\in \{2,3,5,7,13\}$;
\item [(ii)] $g=2$ and $p\in \{ 2, 3, 5, 7, 11\}$; 
\item [(iii)] $(g, p)=(3,2)$ or $(g,p)=(4,2)$. 
\end{itemize}

\item Let $\calC(x)$ be the central leaf of $\calA_{g}$ passing through a point $x=[X_0,\lambda_0]\in \calS_{g}(k)$. Then $\calC(x)$ consists of one element if and
only if one of the following three cases holds:

\begin{itemize}
\item [(i)] $g=1$ and $p\in \{2,3,5,7,13\}$;
\item [(ii)] $g=2$ and $p=2,3$; 
\item [(iii)] $g=3$, $p=2$ and $a(x)\ge 2$. 
\end{itemize}
\end{enumerate}
\end{theorem}

\begin{proof}
\begin{enumerate}
    \item  
    By \cite[Theorem 4.9]{lioort}, the number of irreducible components of $\mathcal{S}_g$ is equal to the class number $H_g(p,1)$ of the principal genus if $g$ is odd, and is equal to the class number $H_g(1,p)$ of the non-principal genus if $g$ is even. Thus, Statement (1) follows from Theorem ~\ref{thm:mainarith}.
    \item The cases where $g=1,2, 4$ or $g \geq 5$ follow from Lemma~\ref{lem:g1}, Lemma~\ref{lem:g2}, Theorem~\ref{thm:maing4} and Lemma~\ref{lem:g5+}, respectively. 

Suppose then that $g=3$. By Lemma~\ref{lem:Lgpc}, either $\vert \Lambda_x\vert \geq  \vert \Lambda_{3,1} \vert=H_3(p,1)$ or $\vert \Lambda_x \vert \geq \vert \Lambda_{3,p}\vert=H_3(1,p)$. Thus, by Theorem~\ref{thm:mainarith}, $\vert \Lambda_x \vert=1$ occurs only when $p=2$. Further assuming $p=2$, by Proposition~\ref{prop:autg3}, $\calC(x)$ has one element if and only if $a(x)\ge 2$.
\end{enumerate}
\end{proof}

\providecommand{\bysame}{\leavevmode\hbox to3em{\hrulefill}\thinspace}
\providecommand{\href}[2]{#2}

\end{document}